\documentclass{article}

\usepackage{microtype}
\usepackage{graphicx}
\usepackage{booktabs} 

\usepackage{hyperref}

\usepackage[accepted]{icml2024}

\usepackage{amsmath}
\usepackage{amssymb}
\usepackage{mathtools}
\usepackage{amsthm}

\usepackage[capitalize,noabbrev]{cleveref}

\theoremstyle{plain}
\newtheorem{theorem}{Theorem}[section]

\newtheorem{lemma}[theorem]{Lemma}
\newtheorem{corollary}[theorem]{Corollary}
\theoremstyle{definition}
\newtheorem{definition}[theorem]{Definition}
\newtheorem{assumption}[theorem]{Assumption}
\newtheorem{problem}[theorem]{Problem}
\theoremstyle{remark}
\newtheorem{remark}[theorem]{Remark}

\usepackage[textsize=tiny]{todonotes}

\icmltitlerunning{Distributionally Robust Infinite-Horizon Control}

\usepackage{tabularx}

\usepackage{bm}             
\usepackage{bbm}

\usepackage{amsfonts}
\usepackage{amscd}

\usepackage{mathrsfs}

\usepackage{tikz}
\usepackage{xcolor}

\usepackage{nicefrac}       

\usepackage{caption}
\usepackage{subcaption}

\usepackage{cleveref}

\usepackage{times}
\usepackage{xspace}
\usepackage{enumitem}

\usepackage{comment}

\newcommand{\QQ}{\clf{Q}}
\newcommand{\PP}{\clf{P}}
\newcommand{\Htwo}{\mathit{H}_2}
\newcommand{\Hinf}{\mathit{H}_\infty}
\newcommand{\RO}{\mathrm{RO}}
\newcommand{\DRO}{\textrm{DR-RO}}

\newcommand{\Ss}{\mathcal{S}}

\renewcommand{\u}{\mathbf{u}} 
\newcommand{\w}{\mathbf{w}} 
\newcommand{\s}{\mathbf{s}}

\newcommand{\F}{\mathcal{F}} 
\newcommand{\G}{\mathcal{G}} 
\newcommand{\K}{\mathcal{K}} 
\newcommand{\T}{\mathcal{T}} 
\newcommand{\NN}{\mathcal{N}}

\newcommand{\M}{\mathcal{M}} 
\renewcommand{\L}{\mathcal{L}} 
\newcommand{\I}{\mathcal{I}} 

\newcommand{\RR}{\mathcal{R}}

\newcommand{\W}{\mathscr{W}} 
 
\newcommand{\causal}{\mathscr{K}}

\newcommand{\trig}{\mathscr{T}}

\newcommand{\Reg}{\operatorname{\textsf{\small Reg}}}
\newcommand{\BW}{\operatorname{\mathsf{BW}}}
\newcommand{\Was}{\mathsf{W_2}}
\newcommand{\HS}{2} 

\newcommand{\cost}{\operatorname{cost}}
\newcommand{\regret} {\operatorname{{\textsc{\small Regret}}}}

\newcommand{\op}{\infty} 

\newcommand{\ejw}{\e^{{j}\omega}}

\DeclareMathOperator*{\argmax}{arg\,max}

\newcommand{\subjto}{\,\mathrm{s.t.}\,}

\newcommand{\sub}{\mathscr{S}}

\newcommand{\defeq}{\coloneqq}                      

\newcommand{\inn}{\!\in\!}
\newcommand{\+}{\!+\!}
\renewcommand{\-}{\!-\!}
\renewcommand{\=}{\!=\!}

\renewcommand{\>}{\!>\!}

\renewcommand{\leqq}{\!\leq\!}
\renewcommand{\geqq}{\!\geq\!}
\newcommand{\timess}{\!\times\!}

\newcommand{\ie}{\textit{i.e.}}

\newcommand{\eps}{\varepsilon}          

\newcommand{\N}{\mathbb{N}}     
\newcommand{\Z}{\mathbb{Z}}     
\newcommand{\R}{\mathbb{R}}     
\newcommand{\C}{\mathbb{C}}     

\newcommand{\Sym}{\mathbb{S}}   

\newcommand{\Prob}{\mathscr{P}} 
\newcommand{\TT}{\mathbb{T}}     

\newcommand{\pr}[1]{\left({#1}\right)}          
\newcommand{\br}[1]{\left[{#1}\right]}          
\newcommand{\cl}[1]{\left\{{#1}\right\}}        

\newcommand{\abs}[1]{\vert{#1}\vert}                    
\newcommand{\Abs}[1]{\left\vert{#1}\right\vert}         

\newcommand{\norm}[2][\text{}]{\Vert{#2}\Vert_{#1}}     
\newcommand{\Norm}[2][\text{}]{\left\Vert{#2}\right\Vert_{#1}} 

\newcommand{\clf}[1]{\mathcal{#1}} 

\newcommand{\tr}{\operatorname{tr}}         
\newcommand{\Tr}{\operatorname{Tr}}         

\newcommand{\tp}{\intercal}                 

\newcommand{\inv}{\mathrm{\-1}}             

\newcommand{\psdgeq}{\succcurlyeq} 
\newcommand{\psdg}{\succ}          

\newcommand{\E}{\operatorname{\mathbb{E}}} 
\renewcommand{\Pr}{\operatorname{\mathbb{P}}} 

\newcommand{\sampled}[1][\text{}]{\stackrel{#1}{\sim}}

\newcommand{\beq}[1]{\begin{align*}\label{eq:#1}}
\newcommand{\eeq}{\end{align*}}

\newcommand{\suml}{\sum\nolimits}

\newcommand{\e}{\mathrm{e}}             
\newcommand{\half}{\frac{1}{2}} 

\makeatletter
\newcommand{\xMapsto}[2][]{\ext@arrow 0599{\Mapstofill@}{#1}{#2}}
\def\Mapstofill@{\arrowfill@{\Mapstochar\Relbar}\Relbar\Rightarrow}
\makeatother

\begin{document}

\twocolumn[
\icmltitle{Infinite-Horizon Distributionally Robust Regret-Optimal Control}



\icmlsetsymbol{equal}{*}

\begin{icmlauthorlist}
\icmlauthor{Taylan Kargin}{equal,yyy}
\icmlauthor{Joudi Hajar}{equal,yyy}
\icmlauthor{Vikrant Malik}{equal,yyy}
\icmlauthor{Babak Hassibi}{yyy}
\end{icmlauthorlist}

\icmlaffiliation{yyy}{California Institute of Technology}

\icmlcorrespondingauthor{Taylan Kargin}{tkargin@caltech.edu}

\icmlkeywords{Wasserstein Distance, Distributionally Robust Control}

\vskip 0.3in
]



\printAffiliationsAndNotice{\icmlEqualContribution} 

\begin{abstract}
We study the infinite-horizon distributionally robust (DR) control of linear systems with quadratic costs, where disturbances have unknown, possibly time-correlated distribution within a Wasserstein-2 ambiguity set. We aim to minimize the worst-case expected regret—the excess cost of a causal policy compared to a non-causal one with access to future disturbance. Though the optimal policy lacks a finite-order state-space realization (\textit{i.e.}, it is non-rational), it can be characterized by a finite-dimensional parameter. Leveraging this, we develop an efficient frequency-domain algorithm to compute this optimal control policy and present a convex optimization method to construct a near-optimal state-space controller that approximates the optimal non-rational controller in the $\mathit{H}_\infty$-norm. This approach avoids solving a computationally expensive semi-definite program (SDP) that scales with the time horizon in the finite-horizon setting.

\end{abstract}

\section{{Introduction }} \label{sec:intro}
Addressing uncertainty is a core challenge in decision-making. Control systems inherently encounter various uncertainties, such as external disturbances, measurement errors, model disparities, and temporal variations in dynamics \citep{grinten1968uncertainty,doyle1985structured}. Neglecting these uncertainties in policy design can result in considerable performance decline and may lead to unsafe and unintended behavior \citep{samuelson2017}.

Traditionally, the challenge of uncertainty in control systems has been predominantly approached through either stochastic or robust control frameworks \citep{kalman_new_1960,zames_feedback_1981,doyle_state-space_1988}. Stochastic control (e.g., Linear–Quadratic–Gaussian (LQG) or $H_2$-control) aims to minimize an expected cost, assuming disturbances follow a known probability distribution \citep{blackbook}. However, in practical scenarios, the true distribution is often estimated from sampled data, introducing vulnerability to inaccurate models. On the other hand, robust control minimizes the worst-case cost across potential disturbance realizations, such as those with bounded energy or power ($H_\infty$ control) \citep{zhou_robust_1996}. While this ensures robustness, it can be overly conservative. Two recent approaches have emerged to tackle this challenge.
\enlargethispage{\baselineskip}

\textbf{Regret-Optimal (RO) Control.} Introduced by \cite{sabag2021regret, goel2023regret}, this framework offers a promising strategy to tackle both stochastic and adversarial uncertainties. It defines regret as the performance difference between a causal control policy and a clairvoyant, non-causal one with perfect knowledge of future disturbances. In the full-information setting, RO controllers minimize the worst-case regret across all bounded energy disturbances \citep{sabag2021regret, goel2023regret}. The infinite-horizon RO controller also takes on a state-space form, making it conducive to efficient real-time computation \citep{sabag2021regret}.

Extensions of this framework have been investigated in various settings, including measurement-feedback control \citep{ROMF, hajar_regret-optimal_2023}, dynamic environments \citep{goel_regret-optimal_2021-1}, safety-critical control \citep{martin2022safe,didier2022system}, filtering \citep{sabag_regret-optimal_2022, goel2023regret}, and distributed control \citep{martinelli_closing_2023}. While these controllers effectively emulate the performance of non-causal controllers in worst-case disturbance scenarios, they may exhibit excessive conservatism when dealing with stochastic ones.

\textbf{Distributionally Robust (DR) Control.} In contrast to traditional approaches such as $H_2$ or $H_\infty$ and RO control that focus on a single distribution or worst-case disturbance realization, the DR framework addresses uncertainty in disturbances by considering ambiguity sets – sets of plausible probability distributions \citep{yang2020wasserstein,kim_distributional_2021, hakobyan2022wasserstein, tacskesen2023distributionally, aolaritei_wasserstein_2023,aolaritei_capture_2023,brouillon2023distributionally}. This methodology aims to design controllers with robust performance across all probability distributions within a given ambiguity set. The size of the ambiguity set provides control over the desired robustness against distributional uncertainty, ensuring that the resulting controller is not excessively conservative.

The controller's performance is highly sensitive to the chosen metric for quantifying distributional shifts. Common choices include the total variation (TV) distance \cite{tzortzis_dynamic_2014,tzortzis_robust_2016}, the  Kullback-Leibler (KL) divergence \cite{falconi_new_2022,liu_data-driven_2023}, and the Wasserstein-2 ($\Was$) distance \cite{DRORO,tacskesen2023distributionally,aolaritei_wasserstein_2023, hajar_wasserstein_2023, brouillon2023distributionally, kargin2023wasserstein}. The controllers derived from KL-ambiguity sets \cite{petersen_minimax_2000,falconi_new_2022} have been linked to the well-known risk-sensitive controller \cite{jacobson_optimal_1973, speyer_optimization_1974, whittle_risk-sensitive_1981}, which minimizes an exponential cost (see \cite{hansen_robustness_2008} and the references therein). However, distributions in a KL-ambiguity set are restricted to be absolutely continuous with respect to the nominal distribution \cite{hu_kullback-leibler_2012}, significantly limiting its expressiveness.

In contrast, $\Was$-distance, which quantifies the minimal cost of transporting mass between two probability distributions, induces a Riemannian structure on the space of distributions \cite{villani_optimal_2009} and allows for ambiguity sets containing distributions with both discrete and continuous support. Thanks to this versatility and the rich geometric framework, it has found widespread adoption across various fields, including machine learning \cite{arjovsky17wasserstein}, computer vision \cite{liu_semantic_2020,ozaydin_omh_2024},
estimation and filtering \cite{shafieezadeh_2018, lotidis_wasserstein_2023, prabhat_optimal_2024}, data compression \cite{blau2019rethinking,lei2021out,malik_distributionally_2024}, and robust optimization \cite{zhao_data-driven_2018,kuhn2019wasserstein,gao_distributionally_2022,blanchet_unifying_2023}. Moreover, the $\Was$-distance has emerged as a theoretically appealing statistical distance for DR linear-quadratic control problems \cite{DRORO} due to its compatibility with quadratic objectives and the resulting tractability of the associated optimization problems \cite{gao_distributionally_2022}.

\subsection{Contributions}
This paper explores the framework of Wasserstein-2 distributionally robust regret-optimal (\DRO) control of linear dynamical systems in the infinite-horizon setting. Initially introduced by \cite{DRORO} for the full-information setting, \DRO~control was later adapted to the partially observable case by \cite{hajar_wasserstein_2023}. Similarly, \cite{tacskesen2023distributionally} derived a DR controller for the partially observed linear-quadratic-Gaussian (LQG) problem, assuming time-independent disturbances. These prior works, focusing on the finite-horizon setting, are hampered by the requirement to solve a semi-definite program (SDP) whose complexity scales with the time horizon, prohibiting their applicability for large horizons.

Our work addresses this limitation by considering the infinite-horizon setting where the probability distribution of the disturbances over the entire time horizon is assumed to lie in a $\Was$-ball of a specified radius centered at a given nominal distribution. We seek a linear time-invariant (LTI) controller that minimizes the worst-case expected regret for distributions adversarially chosen within the $\Was$-ambiguity set. Our contributions are summarized as follows.

\textbf{1. Stabilizing time-invariant controller.} As opposed to the finite-horizon controllers derived in \citep{hakobyan2022wasserstein, tacskesen2023distributionally,aolaritei_capture_2023,DRORO,hajar_wasserstein_2023}, the controllers obtained in the infinite-horizon setting stabilize the underlying dynamics (\cref{thm:stabilizable})

\textbf{2. Robustness against non-iid disturbances.} In contrast to several prior works that assume time-independence of disturbances \citep{yang2020wasserstein,kim_distributional_2021, hakobyan2022wasserstein, tacskesen2023distributionally, zhong2023nonlinear,aolaritei_wasserstein_2023,aolaritei_capture_2023}, our approach does not impose such assumptions, thereby ensuring that the resulting controllers are robust against time-correlated disturbances. 

\textbf{3. Characterization of the optimal controller.} We cast the \DRO~control problem as a max-min optimization and derive the worst-case distribution and the optimal controller using KKT conditions (\cref{thm:kkt}). While the resulting controller is non-rational, lacking a finite-order state-space realization (\cref{thm:irrational}), we show it admits a finite-dimensional parametric form (\cref{thm:fixed_point}).

\textbf{4. Efficient computation of the optimal controller.} Utilizing the finite-dimensional parametrization, we propose an efficient algorithm based on the Frank-Wolfe method to compute the optimal non-rational \DRO~controller in the frequency-domain with arbitrary fidelity (\cref{alg:fixed_point_detailed}).

\textbf{5. Near-optimal state-space controller.} We introduce a novel convex program that finds the best rational approximation of any given order for the non-rational controller in the $\Hinf$-norm (\cref{thm:rational approx via feasibility}). Therefore, our approach enables efficient real-time implementation using a near-optimal state-space controller (\cref{lemma:state space K}).

\textbf{\textit{Notations:}} The letters $\N$, $\Z$, $\R$, and $\C$ denote the set of natural numbers, integers, real, and complex numbers, respectively. $\TT$ denotes the complex unit circle. For $z\inn\C$, $\abs{z}$ is its magnitude, and $z^\ast$ is the conjugate. $\Sym_+^n$ denotes the set of positive semidefinite (psd) matrices of size $n\timess n$.
Bare calligraphic letters ($\K$, $\M$, etc.) are reserved for operators. $\I$ is the identity operator with a suitable block size. For an operator $\M$, its adjoint is $\M^\ast$. For a matrix $A$, its transpose is $A^\tp$, and its Hermitian conjugate is $A^\ast$. For psd operators/matrices, $\psdgeq$ denotes the Loewner order. For a psd operator $\M$, both $\sqrt{\M}$ and $\M^{\half}$ denote the PSD square-root. $\{\M\}_{+}$ and $\{\M\}_{-}$ denote the causal and strictly anti-causal parts of an operator $\M$.
$M(z)$ denotes the z-domain transfer function of a Toeplitz operator $\M$.
$\tr(\cdot)$ denotes the trace of operators and matrices. $\norm{\cdot}$ is the usual Euclidean norm. $\norm[\op]{\cdot}$ and $\norm[\HS]{\cdot}$ are the $\Hinf$ operator) and $\Htwo$ (Frobenius) norms, respectively. 
Probability distributions are denoted by $\Pr$. $\Prob_p(\R^d)$ denotes the set of distributions with finite $p^\textrm{th}$ moment over a $\R^d$. $\E$ denotes the expectation. The Wasserstein-2 distance between distributions $\Pr_1,\Pr_2 \!\in\! \R^{d}$ is denoted by $\Was(\Pr_1,\Pr_2)$ such that
\vspace{-2mm}
\begin{equation}\label{eq:wasserstein}
    \Was(\Pr_1,\Pr_2) \triangleq \pr{ \inf\, \E\br{ \norm{\w_1 \- \w_2}^2} }^{1/2} ,
\end{equation}
where the infimum is over all joint distributions of $(\w_1,\w_2)$ with marginals $\w_1 \!\sampled \!\Pr_1$ and $\w_2\! \sampled \!\Pr_2$.

\section{{Preliminaries} } \label{sec:prelim}

\subsection{Linear-Quadratic Control }
Consider a discrete-time, linear time-invariant (LTI) dynamical system expressed as a state-space model given by:
\vspace{-2mm}
\begin{equation}\label{eq:state_space}
\begin{aligned}
    x_{t+1} &= A x_{t} + B_u u_{t} + B_w w_{t}, \quad s_{t} = C x_{t}.
\end{aligned}
\vspace{-2mm}
\end{equation}
Here, $x_{t} \inn \R^{d_x}$ is the \emph{state}, $s_{t} \inn \R^{d_s}$ is the \emph{regulated output},  $u_t \inn \R^{d_u}$ is the \emph{control input}, and $w_{t} \inn \R^{d_w}$ is the \emph{exogenous disturbance} at time $t$. The state-space parameters $(A,B_u,B_w,C)$ are known with stabilizable $(A,B_u)$, controllable $(A,B_w)$, and observable $(A,C)$. The disturbances are generated from an unknown stochastic process.

We focus on the infinite-horizon setting, where the time index spans from the infinite past to the infinite future, taking values in $\Z$\footnote{The doubly-infinite horizon is chosen for simplicity in derivations, but the results are extendable to a semi-infinite horizon.}. Defining the doubly-infinite column vectors of regulated output $\s \!\defeq\! (s_t)_{t\in\Z}$, control input $\u \!\defeq\! (u_t)_{t\in\Z}$, and disturbance process $\w \!\defeq\! (w_t)_{t\in\Z}$ trajectories, we can express the temporal interactions between these variables globally by representing the dynamics \eqref{eq:state_space} as a \emph{causal linear input/output model}, described by:
\begin{equation} \label{eq: infinite horizon model}
    \s = \F \u + \G \w, 
\end{equation}
where $\F$ and $\G$ are \emph{strictly causal} (\ie, strictly lower-triangular) and doubly-infinite $d_s\timess d_u$ and $d_s\timess d_w$-block \emph{Toeplitz operators}, respectively. These operators describe the influence of the control input and disturbances on the regulated output through convolution with the \emph{impulse response} of the dynamical system \eqref{eq:state_space}, which are completely determined by the model parameters $(A,B_u,B_w,C)$.

\textit{\textbf{Control Policy.}} We restrict our attention to the full-information setting where the control input $u_t$ at time $t\in \Z$ has access to the past disturbances $(w_s)_{s=-\infty}^{t}$. In particular, we consider linear time-invariant (LTI) \emph{disturbance feedback control}\footnote{Youla parametrization enables the conversion between a DFC controller and a state-feedback controller \cite{youla_modern_1976}.} (DFC) policies that map the disturbances to the control input via a causal convolution sum:
\begin{equation}\label{eq:controller}
    u_t = \suml_{s=-\infty}^{t} \widehat{K}_{t-s} w_s, \;\; \textrm{ for all } \;\; t\in\Z.
\end{equation}
The sequence $\{ \widehat{K}_{t}\}_{t=0}^{\infty}$ of $d_u\timess d_w$ matrices are known as the \emph{Markov parameters} of the controller. Similar to the causal linear model in \eqref{eq: infinite horizon model}, the controller equation in \eqref{eq:controller} can be expressed globally by $\u\=\K \w$, where $\K$ is a bounded, \emph{strictly causal}, $d_u\timess d_w$-block \emph{Toeplitz operator} with lower block-diagonal entries given by the Markov parameters. The set of causal DFC policies is denoted by $\causal$.

\textit{\textbf{Cost.}} At each time step, the control inputs and disturbances incur a quadratic instantaneous cost $s_t^\tp s_t \+ u_t^\tp R u_t$, where $ R\!\psdg \!0$. Without loss of generality, we take $R\=I$ by redefining $B_u R^{-\!\half} \!\rightarrow\! B_u$ and $R^{\half}u_{t}\! \rightarrow \! u_{t}$. By defining the truncated sequences $\s_T\!\defeq\!(s_t)_{t=0}^{T-1}$ and $\u_T\!\defeq\!(u_t)_{t=0}^{T-1}$ the \emph{cumulative cost} over a horizon of $T\inn\N$ is simply given by
\begin{equation}\label{eq:cost}
    \cost_T(\u, \w) \defeq \norm{\s_T}^2 + \norm{\u_T}^2.
\end{equation}

\subsection{The Regret-Optimal Control Framework}
We aim to design controllers that reduce the regret against the best offline sequence of control inputs selected in hindsight. For a horizon $T$, the \emph{cumulative regret} is given by
\begin{equation}\label{eq:regret}
    \regret_T(\u,\!{\w}) \!\defeq \!\cost_T(\u,\!\w) \- \min_{\u_T^\prime} \cost_T(\u^\prime, \!\w). 
    \vspace{-2mm}
\end{equation}
We highlight that the minimization on the right-hand side is among all control input sequences, including \emph{non-causal} (offline) ones. The regret-optimal (RO) control framework, introduced by \cite{sabag2021regret}, aims to craft a causal and time-invariant controller $\K\inn\causal$ that minimizes the steady-state worst-case regret across all \emph{bounded energy disturbances}. This can be formally cast as
\vspace{-2mm}
\begin{equation} \label{eq:regret_optimal}
     {\gamma}_{\RO} \!\defeq \! \inf_{\K \in \causal} { \limsup_{T\to \infty}  \frac{1}{T}   \sup_{ {\norm{\w_T}^2 \leq 1}} \regret_T(\K\w,\w)}.
\end{equation}
In the full-information setting, the best sequence of control inputs selected in hindsight is given by $\u_\circ = \K_\circ \w$ where
\begin{equation}  \label{eq:non-causal}
\K_\circ \defeq - (\I + \F^\ast \F)^\inv \F^\ast \G,
\end{equation}
is the optimal non-causal policy {\citep{blackbook}}. Since a non-causal controller lacks physical realization, the optimal RO controller, $\K_{\RO}$ represents the "best" causal policy, attaining performance levels akin to the optimal non-causal policy $\K_\circ$, which enjoys complete access to the disturbance trajectory in advance. 

Exploiting the time-invariance of dynamics in~\eqref{eq:state_space} and the controller $\K\in\causal$, \citet{sabag2021regret} demonstrates the equivalence of \eqref{eq:regret_optimal} to the following:
\begin{equation}
    \inf_{\K \in \causal} \sup_{\norm{\w}^2 \leq 1}\! \w^\ast \RR_\K \w \= \inf_{\K\in\causal} \norm[\op]{\RR_\K},
\end{equation}
where $\norm{\w}$ is the $\ell_2$-norm, and $\RR_\K$, which we call \emph{the regret operator}, is given as 
\begin{equation}  \label{eq:regret operator}
\RR_\K \defeq (\K - \K_\circ)^\ast (\I + \F^\ast \F) (\K - \K_\circ).
\end{equation}
The resulting controller closely mirrors the non-causal controller's performance under the worst-case disturbance sequence but may be conservative for stochastic disturbances.

\subsection{Distributionally Robust Regret-Optimal Control}
This paper investigates distributionally robust regret-optimal control, seeking to devise a causal controller minimizing the worst-case expected regret within a Wasserstein-2 ($\Was$) ambiguity set of disturbance probability distributions. The $\Was$-ambiguity set $\W_T(\Pr_\circ,r)$ for horizon $T$ is defined as a $\Was$-ball of radius of $r_T\>0$ centered at a nominal distribution $\Pr_{\circ,T} \!\in\! \Prob(\R^{T d_w})$, namely:
\begin{equation}\label{eq:wass ambiguity set} 
    \W_T(\Pr_\circ,r_T) \!\defeq\! \cl{\Pr\! \in\! \Prob(\R^{T d_w}) \!\mid \! \Was(\Pr,\, \Pr_\circ) \leqq r_T}.
\end{equation}
In contrast to \eqref{eq:regret_optimal}, which addresses the worst-case regret across all bounded energy disturbances, our focus is on the \emph{worst-case expected regret} across all distributions within the $\Was$-ambiguity set, as defined by \citet{DRORO}
\begin{equation*}\label{eq:worst case exp regret finite}
    \Reg_T(\K,r_T) \!\defeq\!\! \!\!\!\!\sup_{\Pr_T \in \W_T(\Pr_{\circ,\!T},r_T)}\!\! \!\!\E_{\Pr_T}\br{\regret_T(\K\w, {\w})},
\end{equation*}
where $\E_{\Pr_T}$ denotes the expectation such that ${\w_T} \!\sampled\! \Pr_T$. In the infinite-horizon case, this cumulative quantity diverges to infinity. Therefore, we focus on the \emph{steady-state worst-case expected regret}, as defined by \cite{kargin2023wasserstein}:
\begin{definition} \label{def:worst case regret}
The \emph{steady-state worst-case expected regret} suffered by a policy $\K\inn\causal$ is given by the ergodic limit of the cumulative worst-case expected regret, \ie,
\begin{equation}\label{eq:worst case exp regret infinite}
    \overline{\Reg}_{\infty}(\K,r) \defeq \limsup_{T\to \infty}  \frac{1}{T}  \, \Reg_T(\K,r_T).
\end{equation}
\end{definition}

\vspace{-3mm}
To ensure the limit in \eqref{eq:worst case exp regret infinite} is well-defined, the asymptotic behavior of the ambiguity set must be specified. For this purpose, we make the following assumption.
\begin{assumption}\label{asmp:nominal} 
    The nominal disturbance process $\w_\circ\!\defeq\!(w_{\circ,t})_{t\in\Z}$ forms a zero-mean weakly stationary random process with an auto-covariance operator $\M_\circ \!\defeq\!(\widehat{M}_{\circ,t-s})_{t,s\in\Z}$, \ie,  $\E_{\Pr_\circ}[w_{\circ,t} w_{\circ,s}^\tp]= \widehat{M}_{\circ,t-s}$. Moreover, the size of the ambiguity set for horizon $T$ scales as $r_T 
\!\sampled r \sqrt{T}\!$ for a $r\>0$.
\end{assumption}
The choice of $r_T \!\propto \!\sqrt{T}$ aligns with the fact that the $\Was$-distance between two random vectors of length $T$, each sampled from two different iid processes, scales proportionally to $\sqrt{T}$. 

While the limit \eqref{eq:worst case exp regret infinite} is well-defined under \cref{asmp:nominal}, it can still be infinite depending on the chosen controller $\K$. Notably, a finite value for \eqref{eq:worst case exp regret infinite} implies closed-loop stability. In \cref{prob:DR-RO}, we formally state the infinite-horizon Wasserstein-2 \DRO~problem.
\begin{problem}[\textbf{Distributionally Robust Regret-Optimal Control }]\label{prob:DR-RO} 
Find a causal LTI controller, $\K\inn\causal$, that minimizes the steady-state worst-case expected regret~\eqref{eq:worst case exp regret infinite}, \ie,
    \begin{equation} \label{eq:DR-RO}\vspace{-2mm}
    \inf_{\K \in \causal}   \overline{\Reg}_{\infty}(\K,r) \= \inf_{\K \in \causal}  \limsup_{T\to \infty}  \frac{1}{T}  \, \Reg_T(\K,r_T) .
\end{equation}
\end{problem}
In \cref{sec:theory}, we provide an equivalent max-min optimization formulation of \cref{prob:DR-RO}.

\section{A Saddle-Point Problem} \label{sec:theory}

This section presents a tractable convex reformulation of the infinite-horizon \DRO~problem. Concretely, \cref{thm:strong_duality} introduces an equivalent single-variable variational characterization of the steady-state worst-case expected regret \eqref{eq:worst case exp regret infinite} incurred by a fixed time-invariant controller. Exploiting this, we show in \cref{thm:kkt} that \cref{prob:DR-RO} reduces to a convex program over positive-definite operators via duality. Moreover, we characterize the optimal controller and the worst-case distribution via KKT conditions. All the proofs of the subsequent theorems are deferred to the Appendix.

Two major challenges are present in solving the \cref{prob:DR-RO}: ergodic limit in \eqref{eq:worst case exp regret infinite} and causality constraint in \eqref{eq:DR-RO}. Firstly, the ergodic limit definition of the worst-case expected regret for a fixed policy $\K\inn\causal$ requires successively solving optimization problems with ever-increasing dimensions. To address this challenge, we leverage the asymptotic convergence properties of Toeplitz matrices and derive an equivalent formulation of \eqref{eq:worst case exp regret infinite} as an optimization problem over a single decision variable as in \citet{kargin2023wasserstein}. Similar to the time-domain derivations of $\Htwo$ and risk-sensitive controllers in the infinite horizon, the resulting formulations involve the Toeplitz operators $\RR_\K$. This result is presented formally in the subsequent theorem.
\begin{theorem}[\textbf{A Variational Formula for }$\overline{\Reg}_{\infty}$\textbf{ \citep[Thm.5]{kargin2023wasserstein}}] \label{thm:strong_duality}
Under \cref{asmp:nominal}, the steady-state worst-case expected regret $\overline{\Reg}_{\infty}(\K,r)$ incurred by a causal policy $\K\inn\causal$ is equivalent to the following:
\begin{equation}\label{eq:dual_worst_case_regret} 
        \inf_{\substack{\gamma \geq 0,\; \gamma\I \psdg \RR_\K } }\gamma \tr\br{ ((\I-\gamma^\inv \RR_\K)^\inv- \I) \M_\circ}  + \gamma r^2.
\end{equation}
which takes a finite value whenever $\RR_\K$ is bounded. Additionally, the worst-case disturbance is obtained from $w_\star \defeq (\I-\gamma_\star^\inv \RR_\K)^\inv w_\circ$ where $\gamma_\star$ is the optimal solution of \eqref{eq:dual_worst_case_regret} satisfying $\tr\br{((\I-\gamma_\star^\inv \RR_\K)^\inv- \I)^2 \M_\circ} = r^2$.
\end{theorem}
Notice that the optimization in \eqref{eq:dual_worst_case_regret} closely mirrors the finite-horizon version presented by \citet[Thm. 2]{DRORO}, with the key difference being the substitution of finite-horizon matrices with Toeplitz operators.

The second challenge is addressing the causality constraint on the controller. When the causality assumption on the controller is lifted, the non-causal policy $\K_\circ$ achieves zero worst-case expected regret since $\RR_\K$ becomes zero and so the worst-case regret by \cref{thm:strong_duality}. While this example illustrates the triviality of non-causal \DRO~problem, the minimization of worst-case expected regret objective in \eqref{eq:dual_worst_case_regret} over causal policies is, in general, not a tractable problem. 

Leveraging Fenchel duality of the objective in \eqref{eq:dual_worst_case_regret}, we address the causality constraint by reformulating \cref{prob:DR-RO} as a concave-convex saddle-point problem in \cref{thm:kkt} so that the well-known Wiener-Hopf technique \cite{wiener1931klasse,kailath_linear_2000} can be used to obtain the optimal \DRO~controller (see \cref{lemma:wiener} for details). To this end, let $\Delta^\ast \Delta = \I \+ \F^\ast \F$ be the \emph{canonical spectral factorization}\footnote{Analogues to Cholesky factorization of finite matrices.}, where both $\Delta$ and its inverse $\Delta^\inv$ are causal operators. We also introduce the \emph{Bures-Wasserstein} ($\BW$) distance for positive-definite (pd) operators defined as
\vspace{-2mm}
\begin{equation*}
    \BW(\M_1, \M_2)^2 \!\defeq \!{\tr\br{\M_1 \+ \M_2 \- 2 {(\M_2^{\half} \M_1 \M_2^{\half})}^{\half}}}.
    \vspace{-2mm}
\end{equation*}
where $\M_1,\M_2\! \psdg\! 0$ with finite trace \cite{bhatia_bures-wasserstein_2017}.

\begin{theorem}[\textbf{A saddle-point problem for \DRO}]\label{thm:kkt}
    Under \cref{asmp:nominal}, \cref{prob:DR-RO} reduces to a feasible concave-convex saddle-point problem given as
    \begin{equation}\label{eq:maxmin}
        \sup_{\M \psdg 0}  \inf_{\K \in \causal} \tr(\RR_\K \M) \quad \mathrm{s.t.} \quad \BW(\M,\M_\circ) \leq r.
    \end{equation}
    Letting $\K_{\Htwo}\!\defeq\!\Delta^\inv \{\Delta \K_\circ \}_{\!+}$ be the $\Htwo$ controller, the unique saddle point $(\K_\star,\M_\star)$ of  \eqref{eq:maxmin} satisfies:
    \begin{subequations}\label{eq:optimal K and M}
        \begin{align}
             \K_\star &= \K_{\Htwo} +  \Delta^{-1}\cl{  \{\Delta \K_\circ \}_{-} \L_\star}_{+} \L_\star^\inv,\label{eq:optimal K} \\
            \M_\star &= (\I-\gamma_\star^\inv \RR_{\K_\star})^{-1}\M_\circ (\I-\gamma_\star^\inv \RR_{\K_\star})^{-1}, \label{eq:optimal M}
        \end{align}
    \end{subequations}
where $\L_\star \L_\star^\ast \!=\! \M_\star$ is the canonical spectral factorization with causal and unique \footnote{See the note in \cref{remark:unique L} about the uniqueness of $\L$} $\L_\star$ and $\L_\star^{-1}$, and $\gamma_\star\>0$ uniquely satisfies $\tr\br{((\I-\gamma_\star^\inv \RR_{\K_\star})^\inv- \I)^2 \M_\circ} = r^2$.
\end{theorem}
This result demonstrates that the optimal \DRO~controller integrates the $\Htwo$ controller with an additional correction term that accounts for the time correlations in the worst-case disturbance, $\w_\star$, which are encapsulated by the auto-covariance operator $\M_\star$.

\begin{remark}\label{remark:limiting_r}
As $r\!\to\!\infty$, the optimal $\gamma_\star$ approaches the lower bound $\gamma_{\RO}\= \inf_{\K\in\causal}\norm[\op]{\RR_\K}$ and $\K_\star$ recovers the regret-optimal ($\RO$) controller. Conversely, as $r\!\to\!0$, the ambiguity set collapses to the nominal model as $\gamma_\star\!\to\!\infty$ and $\K_\star$ recovers the $\Htwo$ controller when $\M_\circ \= \I$. Thus, adjusting $r$ facilitates the \DRO~controller to interpolate between the $\RO$ and $\Htwo$ controllers.
\end{remark}

We conclude this section by asserting the closed-loop stability of \eqref{eq:state_space} under the optimal \DRO~controller, $\K_\star$. This stability directly results from the saddle-point problem \eqref{eq:maxmin} achieving a finite optimal value.
\begin{corollary}\label{thm:stabilizable}
    $\K_\star$ stabilizes the closed-loop system.  
\end{corollary}
\vspace{-4.5mm}

\section{An Efficient Algorithm} \label{sec:fixed_point}

In this section, we introduce a numerical method to compute the saddle-point $(\K_{\star}, \M_{\star})$ of the max-min problem in \eqref{eq:maxmin}. While both $(\K_{\star}, \M_{\star})$ are non-rational, \ie, do not admit a finite order state-space realization, \cref{thm:fixed_point} states that $\M_{\star}$ possesses a \emph{finite-dimensional parametric} form in the frequency domain. Exploiting this fact, we conceive \cref{alg:fixed_point_detailed}, a procedure based on the Frank-Wolfe method, to compute the optimal $\M_\star$ in the frequency domain. Furthermore, we devise a novel approach to approximate the non-rational $\M_\star$ in $\Hinf$-norm by positive rational functions, from which a near-optimal state-space \DRO~controller can be computed using \eqref{eq:optimal K}. We leave the discussion on the rational approximation method to \cref{sec:state_space}.

To enhance the clarity of our approach, we assume for the remainder of this paper that the nominal disturbances are uncorrelated, \ie, $\M_\circ \= \I$. Additionally, we utilize the frequency-domain representation of Toeplitz operators as transfer functions, denoting $\M$ as $M(z)$, $\L$ as $L(z)$, $\K$ as $K(z)$, and similarly for other operators, where $z\inn \C$.
\vspace{-3mm}

\subsection{Finite-Dimensional Parametrization of $\M_\star$}\label{subsec::subK}

We first obtain an equivalent condition of optimality of $\M_\star$.
\begin{lemma}\label{lemma:kkt2}
    Define the anti-causal operator $\mathcal T\defeq\{\Delta \K_\circ\}_{\-}$. The optimality condition in \eqref{eq:optimal K and M} takes the equivalent form:
    \vspace{-2mm}
    \begin{equation}\label{eq:optimal N sqrt form}
        \L_\star^\ast \L_\star = \frac{1}{4}\pr{\I \+ \sqrt{\I \+ 4 \gamma_\star^\inv \{\mathcal T \L_{\star}\}_{\-}^\ast \{\mathcal T \L_{\star}\}_{\-}} }^{\!2}
        \vspace{-2mm}
    \end{equation} where $\gamma_\star\>0$ is such that $ \BW(\L_\star \L_\star^\ast,\I )=r$.
\end{lemma}
\vspace{-2mm}
Denoting $\NN_{\star} \!\defeq\!  \L_{\star}^\ast \L_{\star}$, there exists a one-to-one mapping between $\M_{\star}\=\L_\star \L_\star^\ast$ and $\NN_{\star}$ due to the uniqueness of the spectral factorization. Consequently, we interchangeably refer to both $\NN_{\star}$ and $\M_{\star}$ as the optimal solution. The following theorem characterizes the optimal $\NN_\star$ in the frequency domain, implying a finite-dimensional parametrization.
\begin{theorem}\label{thm:fixed_point}
Denoting by $T(z)\=\overline{C}(z^{\-1}I \- \overline{A})^\inv \overline{B}$ the transfer function of the anti-causal operator $\mathcal T=\{\Delta \K_\circ\}_{\!-}$, let $f:(\gamma, \Gamma)\inn \R\timess \R^{d_x \times d_w } \mapsto \NN $ return a pd operator with a transfer function $z\inn \TT \mapsto N(z)$ taking the form
\vspace{-1mm}
\begin{equation*}
\vspace{-2mm}
    \frac{1}{4}\pr{I \+ \sqrt{I \+ 4 \gamma^\inv {\Gamma}^\ast {(z^{\-1}I \- \overline{A})}^{\!-\ast}\overline{C}^\ast \overline{C}{(z^{\-1}I \- \overline{A})}^\inv\Gamma }}^{\!2},
\end{equation*}
where $(\overline{A},\overline{B},\overline{C})$ are obtained from the state-space parameters of the system in \eqref{eq:state_space} (see \cref{ap:fixed}). Then, the optimal solution $\NN_\star\=\L_\star^\ast \L_\star$ in \eqref{eq:optimal N sqrt form} satisfies $\NN_\star = f(\gamma_\star, \Gamma_\star)$ where
\vspace{-2mm}
\begin{equation}\label{eq:finite parameter computation}
    \Gamma_\star \defeq \frac{1}{2\pi} \int_{0}^{2\pi} (I-\ejw \overline{A})^\inv \overline{B} L_\star(\ejw) d\omega ,
\end{equation}
and $\gamma_\star\>0$ is such that $ \BW(\L_\star \L_\star^\ast,\I )=r$.
\end{theorem}

\vspace{-2mm}
Notice in \cref{thm:fixed_point} that $N_{\star}(z)$ involves the square root of a rational term. In general, the square root does not preserve rationality. We thus get Corollary~\ref{thm:irrational}.
\begin{corollary}\label{thm:irrational}
   The optimal \DRO~controller, $K_{\star}(z)$, and $N_{\star}(z)$ are non-rational. Thus, $K_{\star}(z)$ does not admit a finite-dimensional state-space form.
\end{corollary}

\vspace{-2mm}
Given the non-rationality of the controller $K_{\star}(z)$, \citet{kargin2023wasserstein} proposes a fixed-point algorithm exploiting the finite-parametrization of the controller. In the next section, we propose an alternative efficient optimization algorithm, which, in contrast to the fixed-point algorithm, has \emph{provable convergence guarantee} to the saddle point $(\K_{\star}, \NN_{\star})$.

\subsection{An Iterative Optimization in the Frequency Domain}\label{subsec:FW}
Although the problem is concave, its infinite-dimensional nature complicates the direct application of standard optimization tools. To address this challenge, we employ frequency-domain analysis via transfer functions,  allowing for the adaptation of standard optimization techniques. Specifically, we utilize a variant of the Frank-Wolfe method \cite{frank_algorithm_1956,jaggi_revisiting_2013}. Our approach is versatile and can be extended to other methods, such as projected gradient descent \cite{goldstein1964convex} and the fixed-point method in \cite{kargin2023wasserstein}. Furthermore, the convergence of our method to the saddle point $(\K_{\star}, \NN_{\star})$ can be demonstrated using standard tools in optimization. Detailed pseudocode is provided in \cref{alg:fixed_point_detailed} in \cref{app:detailed pseudocode}.

\textbf{Frank-Wolfe:} We define the following function and its (Gateaux) gradient \cite{danskin}:
\begin{align}\label{eq:wiener-hopf function}
    \Phi(\M) &\triangleq \inf_{\K \in \causal} \tr\pr{ \RR_\K \M }\\
    \nabla \Phi(\M) &\= \L^{-\ast} \cl{ \Delta \K_\circ \L  }^\ast_{-}\cl{\Delta \K_\circ \L  }_{-} \L^{-1} \ .
\end{align}
where $\L\L^\ast \=\M$ is the spectral factorization. Rather than directly solving the optimization \eqref{eq:maxmin}, the Frank-Wolfe method solves a linearized subproblem in consecutive steps. Namely, given the $k^\textrm{th}$ iterate $\M_k$, the next iterate $\M_{k+1}$ is obtained via
\vspace{-2mm}
\begin{subequations} \label{eq: FW updates}
    \begin{equation}
     \vspace{-2mm}
    \widetilde{\M}_{k} \!=\! \argmax_{\substack{\M \psdgeq \I,\; \BW(\M,\I )\leq r}} \,\tr\pr{\nabla \Phi(\M_k)\,   \M }    \label{eq:linear Opt}  
    \end{equation}
    \vspace{-2mm}
\begin{equation}
    \vspace{-0mm}
    \quad\quad\M_{k+1} = (1-\eta_k) \M_k + \eta_k \widetilde{\M}_k,  \label{eq:fw linear combination}
\end{equation}
\end{subequations}
where $\eta_k\inn [0,1]$ is a step-size, commonly set to $\eta_k \= \frac{2}{k+2}$ \cite{jaggi_revisiting_2013}. Letting $\RR_{k} \!\defeq \!\nabla \Phi(\M_k)$ be the gradient as in \eqref{eq:wiener-hopf function}, Frank-Wolfe updates can be expressed equivalently using spectral densities as:
\begin{align}\label{eq:frank wolfe frequency}
    &\widetilde{M}_{k}(z)\= (I \- \gamma_k^\inv R_{k}(z))^{-2} \\
    &\quad M_{k+1}(z) \= (1\-\eta_k) M_k(z) \+ \eta_k \widetilde{M}_k(z), \quad \forall z\in\TT
\end{align}
where $\gamma_k\>0$ solves $\tr\br{((\I\-\gamma_k^\inv \RR_{k})^{-1} \- \I)^2} \= r^2$. See \cref{app: gradients in fw} for a closed-form $R_{k}(z)$.

\textbf{Discretization:} Instead of the continuous unit circle $\TT$, we use its uniform discretization with $N$ points, $\TT_N \defeq \{\e^{j 2\pi n /N} \mid n = 0, \dots, N-1\}$. Updating $M_{k+1}(z)$ at a frequency $z$ using the gradient $R_{k}(z)$ at the same $z$ requires $M_k(z^\prime)$ at all frequencies $z^\prime \in \TT$ due to spectral factorization. Thus, $M_{k+1}(z)$ depends on $M_k(z^\prime)$ across the entire circle. This can be addressed by finer discretization.

\textbf{Spectral Factorization:} For the non-rational spectral densities $M_k(z)$, we can only use an \emph{approximate} factorization \cite{sayed_survey_2001}. Consequently, we use the DFT-based algorithm from \citet{rino_factorization_1970}, which efficiently factorizes scalar densities (\ie, $d_w \= 1$), with errors diminishing rapidly as $N$ increases. Matrix-valued spectral densities can be factorized using various other algorithms \cite{wilson_factorization_1972, ephremidze_elementary_2010}. See \cref{app: spectral factorization} for a pseudocode.

\textbf{Bisection: } We use bisection method to find the $\gamma_k\>0$ that solves $\tr\br{((\I\-\gamma_k^\inv \RR_{k})^{-1} \- \I)^2} \= r^2$ in the Frank-Wolfe update \eqref{eq:frank wolfe frequency}. See \cref{app:bisection} for a pseudocode.

\begin{remark}
    The gradient $R_{k}(z)$ requires the computation of the finite-dimensional parameter via \eqref{eq:finite parameter computation}, which can be performed using $N$-point trapezoidal integration. See \cref{app: gradients in fw} for details.
\end{remark}
We conclude this section with the following convergence result due to \cite{jaggi_revisiting_2013, lacoste-julien_affine_2014}.
\begin{theorem}[\textbf{Convergence of $\M_k$}]\label{thm: convergence of FW}
    There exists constants $\delta_N\>0$, depending on discretization $N$, and $\kappa\>0$, depending only on state-space parameters \eqref{eq:state_space} and $r$, such that, for a large enough $N$, the iterates in \eqref{eq: FW updates} satisfy
    \vspace{-2mm}
    \begin{equation}\label{eq:convergence rate}
    \vspace{-2mm}
        \Phi(\M_\star) - \Phi(\M_k) \leq \frac{2\kappa}{k+2}(1+\delta_N).
        \vspace{-2mm}
    \end{equation}
    \vspace{-3mm}
\end{theorem}

\section{Rational Approximation} \label{sec:state_space}

The preceding section determined that the optimal solution, denoted as $\NN_{\star}$, is non-rational and lacks a state-space representation. Nevertheless, \cref{alg:fixed_point_detailed} introduced in \cref{subsec:FW} can effectively approximate it in the frequency domain. Indeed, after convergence, the algorithm returns the optimal finite parameter, $\Gamma_\star$, which can be used to compute $N_\star(z)$ at \emph{any arbitrary frequency} using \cref{thm:fixed_point}, and thus $K_\star(z)$ (see \cref{alg:fixed_point_detailed} in \cref{app:detailed pseudocode}). However, a state-space controller must be devised for any practical real-time implementation.

This section introduces an efficient method to obtain state-space controllers approximating the non-rational optimal controller. Instead of directly approximating the controller itself, our method involves an initial step of \emph{approximating the power spectrum $N_\star(z)$ of the worst-case disturbance} to minimize the $\Hinf$-norm of the approximation error using positive rational functions. While problems involving rational function approximation generally do not admit a convex formulation, we show in \cref{thm:rational approx via feasibility} that approximating positive power spectra by a ratio of positive fixed order polynomials can be cast as a convex feasibility problem. After finding a rational approximation of $N_\star(z)$, we compute a state-space controller according to \eqref{eq:optimal K}. For the sake of simplicity, we focus on scalar disturbances, \ie, $d_w\=1$.

\subsection{State-Space Models from Rational Power Spectra}

As established in \cref{thm:kkt}, the derivation of a optimal controller $K_{\star}$ is achieved through the positive operator $\NN_{\star} = \L_{\star}^{\ast}\L_{\star}$ using the Wiener-Hopf technique. Specifically, we have $\K_{\star} =\K_{\Htwo} +  \Delta^{-1}\cl{  \{\Delta \K_\circ \}_{-} \L_\star}_{+} \L_\star^\inv \L_{\star}^\inv$. Since other controllers of interest, including $\Htwo$, $\Hinf$, and $\RO$, can all be formulated this way, we focus on obtaining approximations to positive power spectra.

It is worth noting that a positive and symmetric rational approximation $\widehat{N}(z)$ of order $m\in\N$ can be represented as a ratio $\widehat{N}(z) = P(z)/ Q(z)$ of two positive symmetric polynomials  $P(z)= p_0 + \sum_{k=1}^{m} p_k (z^k+z^{\-k})$, and $Q(z)= q_0 + \sum_{k=1}^{m} q_k (z^k+z^{\-k})$. When such $P(z),Q(z)$ exist, we can obtain a rational spectral factorization of $\widehat{N}(z)$ by obtaining spectral factorization for $P(z)$, and $Q(z)$.
 
Finally, we end this section by stating an exact characterization of positive trig. polynomials. While verifying the positivity condition for general functions might pose challenges, the convex cone of positive symmetric trigonometric polynomials, $\trig_{m,+}$, possess a characterization through a linear matrix inequality (LMI), as outlined below:
\begin{lemma}[{Trace parametrization of $\trig_{m,+}$ \citep[Thm. 2.3]{dumitrescu_positive_2017}}]\label{lem:positive_poly}
For $k\= [-m,m]$, let $\bm{\Theta}_k \in \R^{(m\+1)\timess (m\+1)}$  
be the primitive Toeplitz matrix with ones on the $k^{\text{th}}$ diagonal and zeros everywhere else. Then, $P(z)= p_0 + \sum_{k=1}^{m} p_k (z^{k}+z^{\-k}) >0 $ if and only if there exists a real positive definite matrix $\mathbf{P} \in \Sym^{m+1}_{+}$ such that
\vspace{-2mm}
\begin{equation} 
        p_k = \tr(\mathbf{P} \bm{\Theta}_k), \;\, k =0,\dots,m.
        \vspace{-3mm}
\end{equation}
\end{lemma}
According to \cref{lem:positive_poly}, any positive trig. polynomial of order at most $m$ can be expressed (non-uniquely) as $P(z) \= \suml_{k=-r}^{r} \tr(\mathbf{P} \bm{\Theta}_k) z^{\-1} \= \tr\pr{\mathbf{P}\bm{\Theta}(z)}$. Here, $\bm{\Theta}(z)\defeq \sum_{k=-r}^{r} \bm{\Theta}_k z^{\-1}$.
\vspace{-2mm}

\subsection{Rational Approximation using $\Hinf$-norm}
\label{sec:rational_approx}
In this context, we present a novel and efficient approach for deriving rational approximations of non-rational power spectra. Our method bears similarities to the flexible uniform rational approximation approach described in \cite{sharon2021flexible}, which approximates a function with a rational form while imposing the positivity of the denominator of the rational form as a constraint. Our method uses $\Hinf$-norm as criteria to address the approximation error effectively. First, consider the following problem:
\begin{problem}[Rational approximation via $H_\infty$-norm minimization] \label{prob:hinf_rational_opt}
Given a positive spectrum $\NN$, find the best rational approximation of order at most $m\in\N$ with respect to $H_\infty$ norm, \ie,
\vspace{-2mm}
\begin{equation}\label{eq:hinf_rational_opt}
\vspace{-5mm}
    \inf_{\PP, \QQ \in \trig_{m,+}} \Norm[\op]{{\PP/\QQ} - \clf{N} } \subjto \tr(\QQ)=1
\end{equation}
\end{problem}
Note that the constraint $\tr(\QQ)\=1$, equivalent to $q_0\=1$, eliminates redundancy in the problem since the fraction ${\PP}/{\QQ}$ is scale invariant.

\begin{figure}[ht]
	\centering
	\begin{subfigure}[b]{0.9\columnwidth}
		\centering
		\includegraphics[width=0.8\columnwidth]{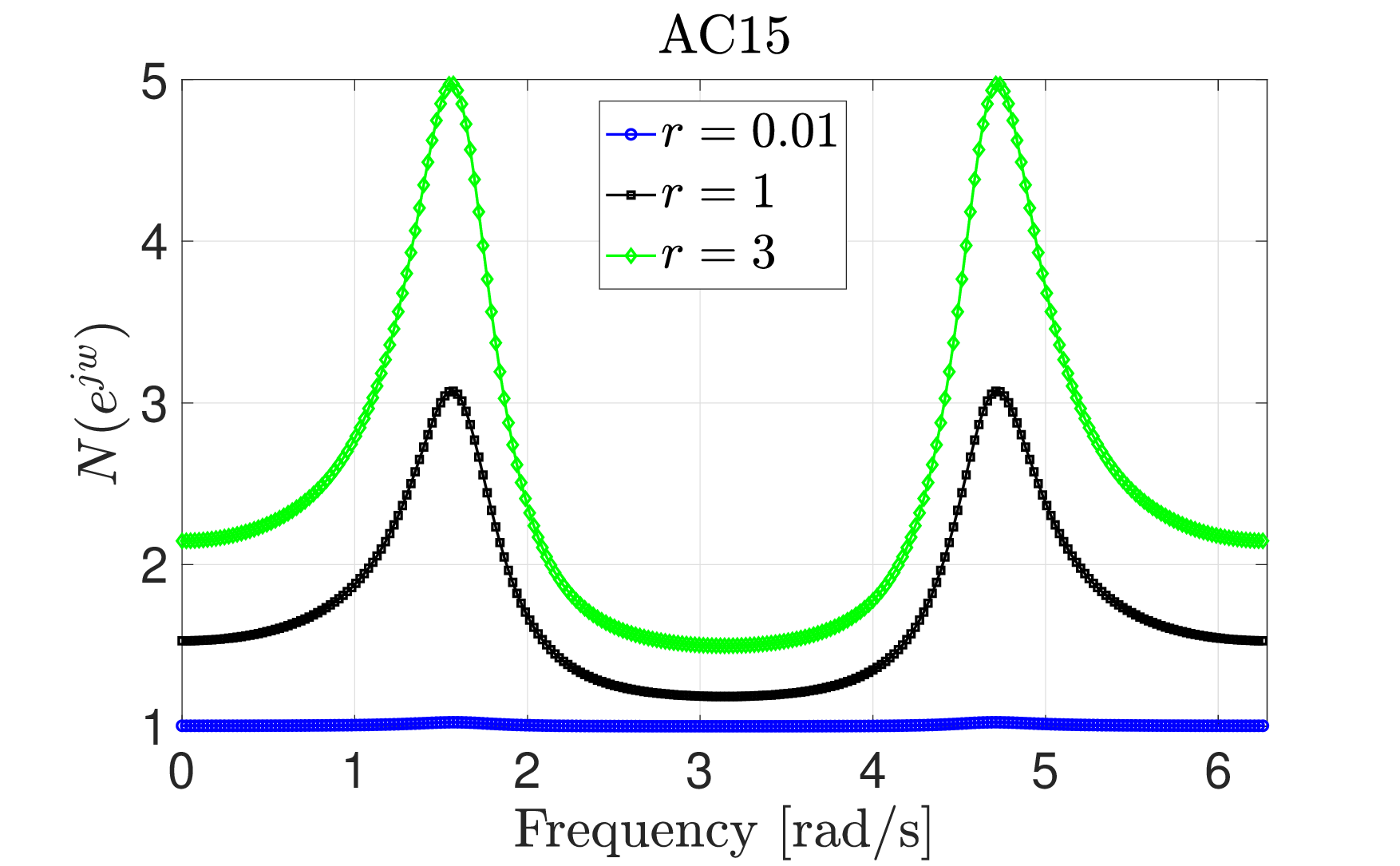}
		\caption{The frequency domain representation of $\mathcal{N}$ for $r=0.01, 1, 3$ for system [AC15].}
		\label{fig:fig_m_freq_1}
	\end{subfigure}
	\vfill
	\begin{subfigure}[b]{0.9\columnwidth}
		\centering
		\includegraphics[width=0.8\columnwidth]{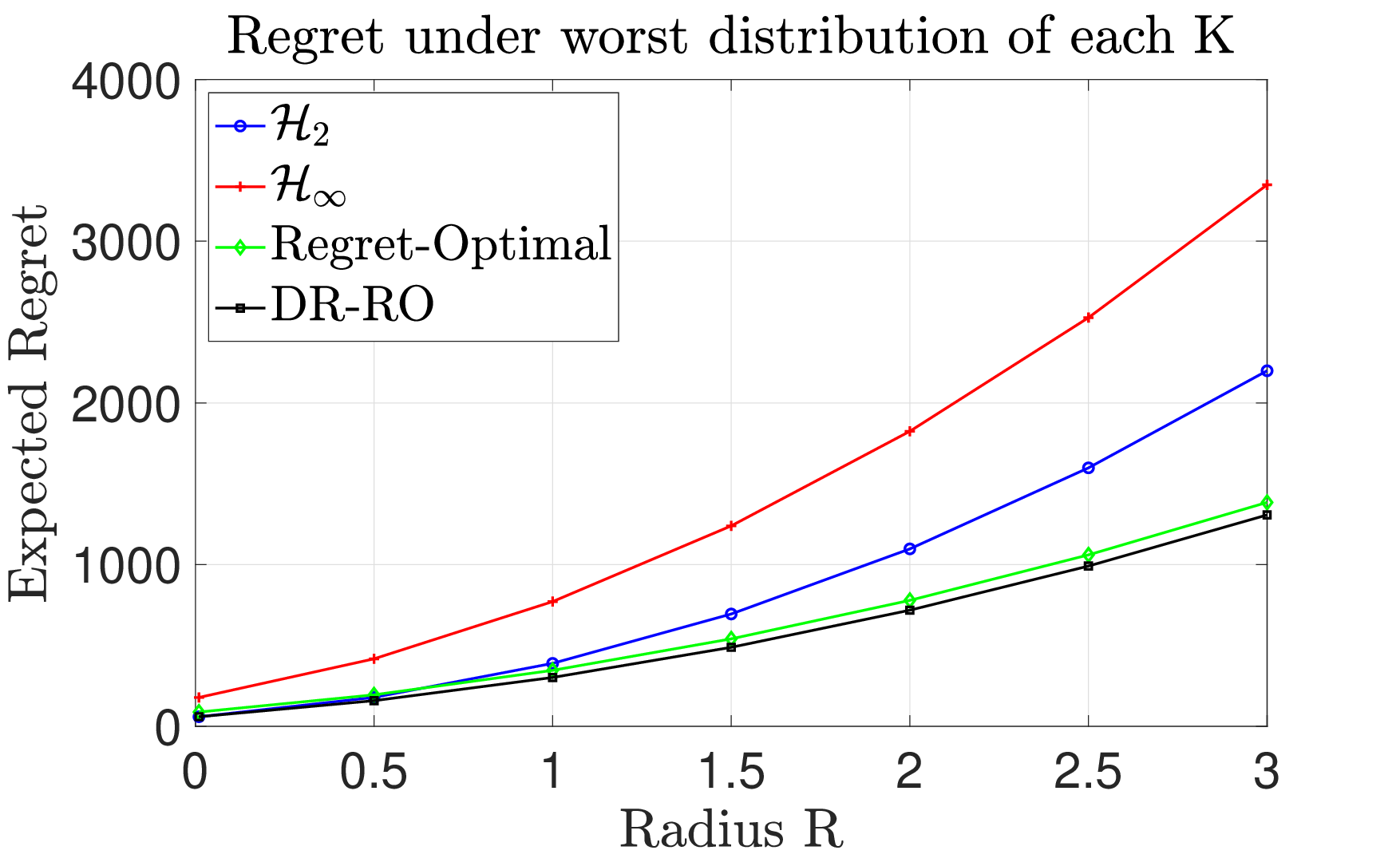}
		\caption{The worst-case expected regret of different controllers for the system [AC15].}
		\label{fig:ER}
	\end{subfigure}
	\caption{Variation of $\mathcal{N}$ with $r$ and the performance of the \DRO~controller versus the $\mathit{H}_2, \mathit{H}_{\infty}$, and $RO$ controller.}
	\label{fig:teaser}
 \vspace{-5mm}
\end{figure}

While the objective function in \cref{eq:hinf_rational_opt} is convex with respect to $\PP$ and $\QQ$ individually, \emph{it is not jointly convex in $(\PP,\QQ)$}. In this form, \cref{prob:hinf_rational_opt} is not amenable to standard convex optimization tools. 

To circumvent this issue, we instead consider the sublevel sets of the objective function in \cref{eq:hinf_rational_opt}.
\begin{definition}\label{def:sublevel}
    For a given $\epsilon>0$ approximation bound, the $\epsilon$-sublevel set of \cref{prob:hinf_rational_opt} is defined as
    \begin{equation*}
    \vspace{-3mm}
    \sub_\epsilon \!\defeq\! \cl{(\PP,\QQ)  \,\left|\; \Norm[\op]{{\PP/\QQ} \- \clf{N} } \!\!\leq\! \epsilon, \; \tr(\QQ)\=1 \right.}. 
\end{equation*}
\end{definition}
By applying the definition of $\Hinf$-norm, we have that
\begin{align}
    &\Norm[\op]{{\PP /\QQ} \- \clf{N} } \!\!\= \max_{z\in \TT} \Abs{{P(z)/Q(z)} - {N(z)} } \leq \epsilon \nonumber\\
    & \iff  \left\{
    \begin{aligned}
    &P(z) \- \pr{ N(z)  \+ \epsilon } {Q(z)}\leqq 0, \\
    &{P(z)} \- \pr{ N(z) \- \epsilon } {Q(z)}\geqq 0,
    \end{aligned}\right.\label{eq:hinf_affine}
\end{align}
where the last set of inequalities hold for all $z\in \TT$. Notice that the inequalities in \cref{eq:hinf_affine} and the positivity constraints on $\PP,\QQ$ are jointly affine in $(\PP,\QQ)$. Moreover, the equation $\tr(\QQ)=1$ is an affine equality constraint. Therefore, we have the following claim.
\begin{lemma}\label{thm:convex_sublevel_set}
The set $\sub_\epsilon$ is jointly convex in $(\PP,\QQ)$.
\end{lemma}
\vspace{-2mm}
Unlike its non-convex optimization counterpart in \cref{prob:hinf_rational_opt}, a membership oracle for the convex set $\sub_\epsilon$ offers a means to obtain accurate rational approximations for non-rational functions. According to \cref{lem:positive_poly}, the positive trig. polynomials $(\PP, \QQ)\in \sub_\epsilon$ can be parameterized by psd matrices $\mathbf{P}$ and $\mathbf{Q}$. This allows the equality constraint $\tr(\QQ)$ and the affine inequalities in \eqref{eq:hinf_affine} to be expressed as Linear Matrix Inequalities (LMIs) in terms of $\mathbf{P}$ and $\mathbf{Q}$. The resulting theorem characterizes the $\epsilon$-sublevel sets.

\begin{theorem}[Feasibility of $\sub_{\epsilon}$]\label{thm:rational approx via feasibility}
    Let $\epsilon\>0$ be a given accuracy level, and $m\in \N$ is a fixed order. The trig. polynomials $\PP$ and $\QQ$ of order $m$ belong to the $\epsilon$-sublevel set, $(\PP,\QQ) \in \sub_\epsilon$ if and only if there exists $\mathbf{P},\mathbf{Q} \in \Sym_{+}^{m +1}$ such that $\tr\pr{\mathbf{Q}} = 1$ and for all $z\in \TT$, .
\begin{align}
\vspace{-3mm}
    1)\,&\tr\pr{\mathbf{P} \bm{\Theta}(z)} \- \pr{ N(z)  \+ \epsilon } \tr\pr{\mathbf{Q} \bm{\Theta}(z)}\leqq 0, \\
    2)\,&\tr\pr{\mathbf{P} \bm{\Theta}(z)} \- \pr{ N(z)  \- \epsilon } \tr\pr{\mathbf{Q} \bm{\Theta}(z)}\geqq 0.
    \vspace{-6mm}
\end{align}
\end{theorem}

\begin{figure}[ht]
    \centering
    \vspace{-0.2cm}
    \begin{subfigure}[b]{0.9\columnwidth}
        \centering
        \includegraphics[width=0.8\columnwidth]{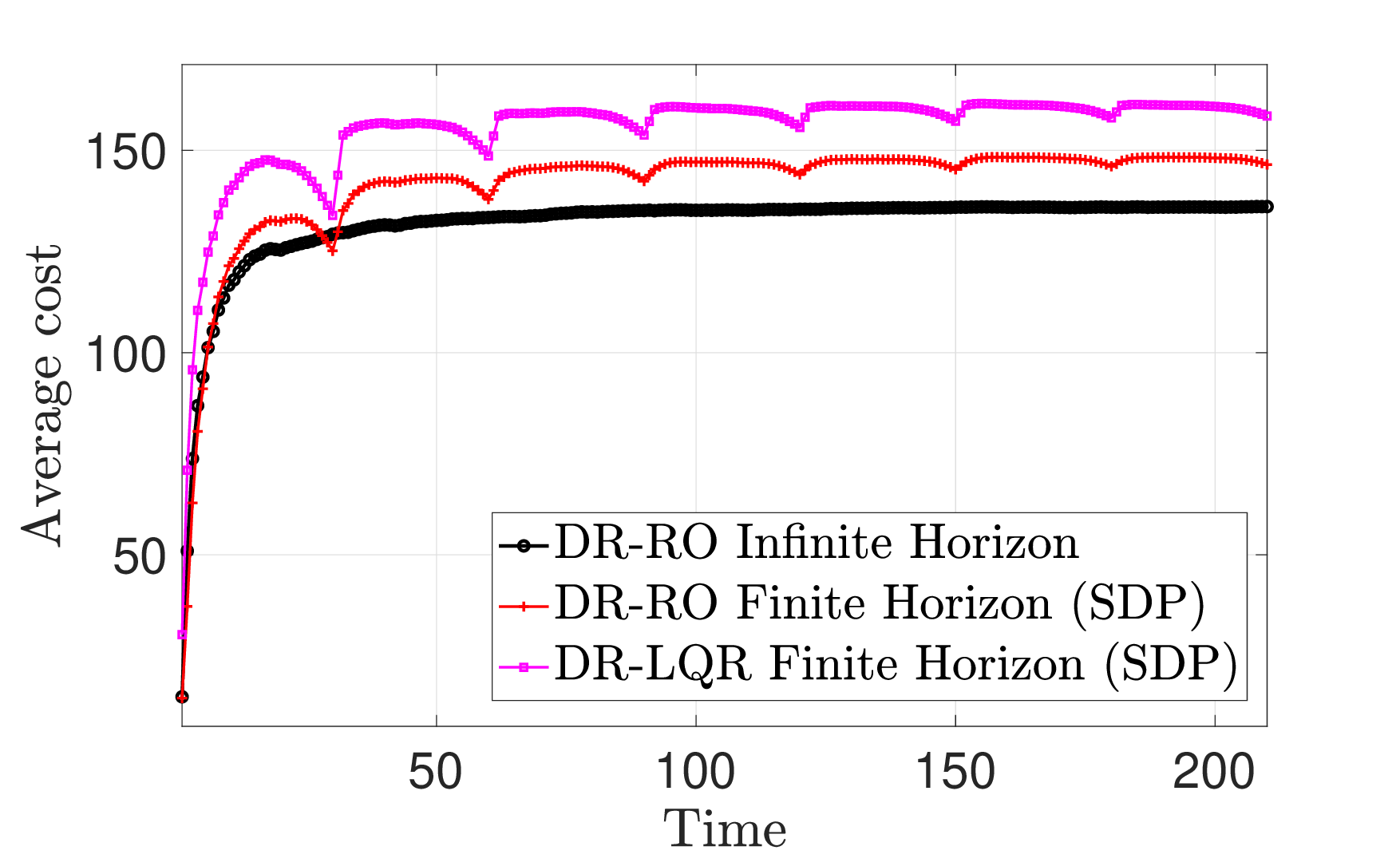}
        \caption{White noise}
        \label{fig:figa}
    \end{subfigure}
    \hfill
    \begin{subfigure}[b]{0.9\columnwidth}
        \centering
        \includegraphics[width=0.8\columnwidth]{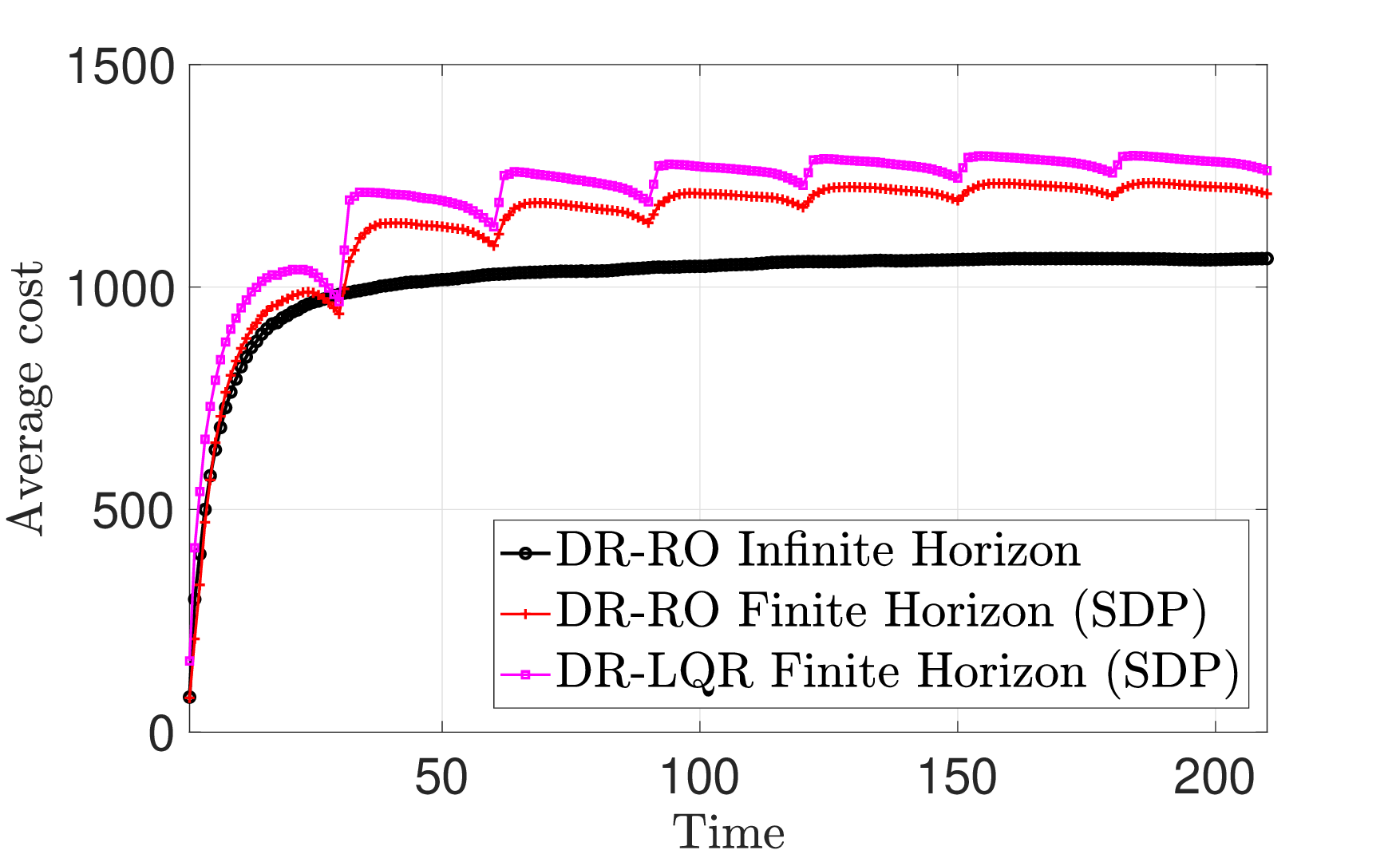}
        \caption{Worst disturbance for \DRO, infinite horizon}
        \label{fig:figb}
    \end{subfigure}

    \caption{The control costs of different DR controllers under (a) white noise and (b) worst disturbance for \DRO~in infinite horizon, for system [AC15]. The finite-horizon controllers are re-applied every $s=30$ steps. The infinite horizon \DRO~controller achieves the lowest average cost compared to the finite-horizon controllers.}
    \label{fig:time_domain_1}
\end{figure}
    \vskip\baselineskip
\begin{figure}[ht]
    \centering
    \begin{subfigure}[b]{0.9\columnwidth}
        \centering
        \includegraphics[width=0.8\columnwidth]{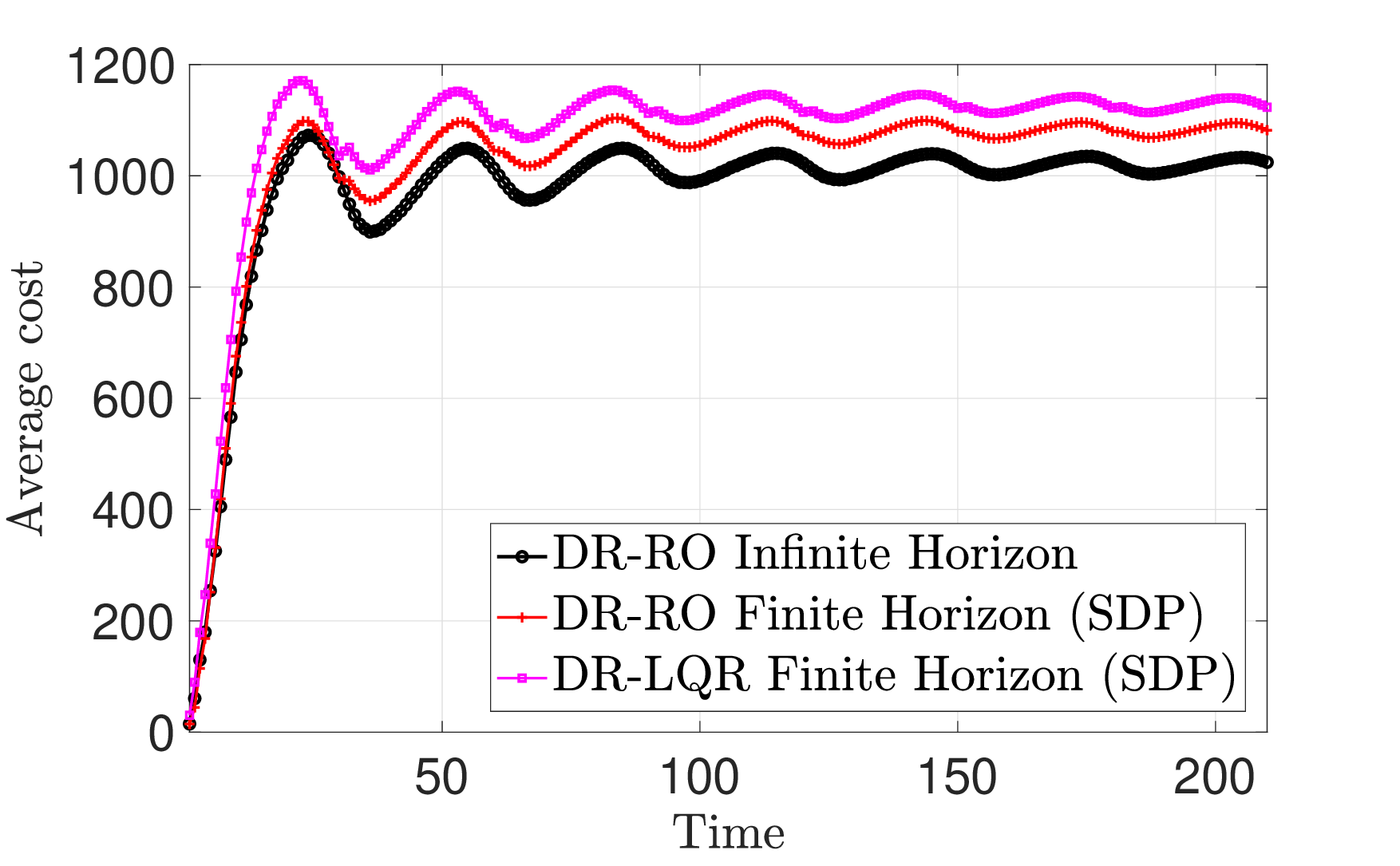}
        \caption{Worst disturbance for \DRO, finite horizon}
        \label{fig:figc}
    \end{subfigure}
    \hfill
    \begin{subfigure}[b]{0.9\columnwidth}
        \centering
        \includegraphics[width=0.8\columnwidth]{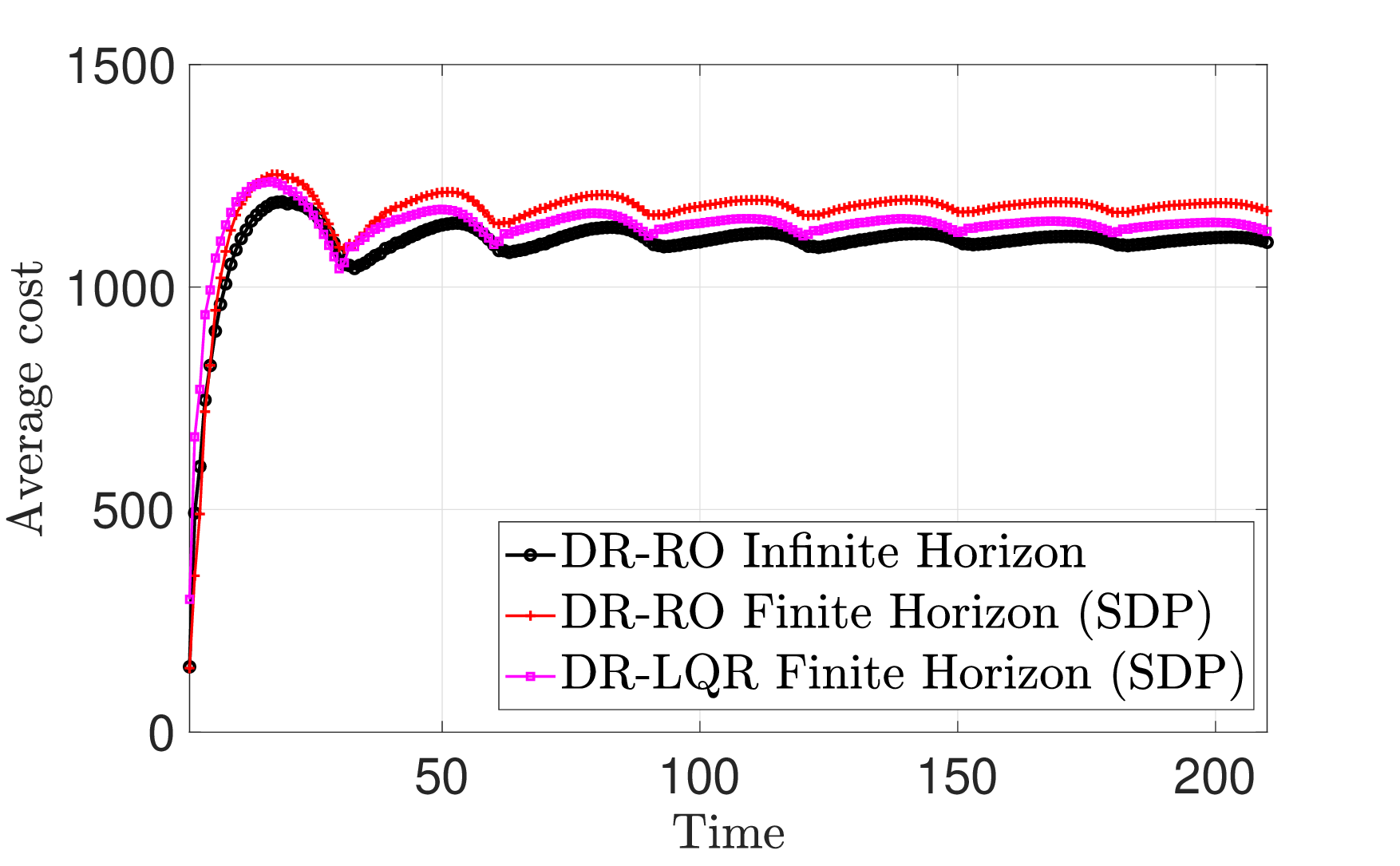} 
        \caption{Worst disturbance for DR-LQR, finite horizon}
        \label{fig:figd}
    \end{subfigure}

    \caption{The control costs of different DR controllers under (a) worst disturbances for \DRO~in finite horizon and (b) worst disturbances for DR-LQR in finite horizon, for system [AC15]. The finite-horizon controllers are re-applied every $s=30$ steps. Despite being designed to minimize the cost under specific disturbances, the finite horizon DR controllers are outperformed by the infinite horizon \DRO~controller.}
    \label{fig:time_domain_2}
    \vspace{-5mm}
\end{figure}

\vspace{-7mm}
The sole limitation in this approach arises from the fact that for an non-rational $N(z)$, the set of infinitely many inequalities in \eqref{eq:hinf_affine} cannot be precisely characterized by a finite number of constraints, as seen in the trace parametrization of positive polynomials. To overcome this challenge, one can address the inequalities in \eqref{eq:hinf_affine} solely for a finite set of frequencies, such as $\TT_N = \{\e^{j 2\pi n /N} \mid n = 0, \dots, N-1\}$ for $N\gg m$. While this introduces an approximation, the method's accuracy can be enhanced arbitrarily by increasing the frequency samples. By taking this approach, the problem of rational function approximation can be reformulated as a convex feasibility problem involving LMIs and a finite number of affine (in)equality constraints.

It is crucial to note our algorithm can be used in the following two modes. These operational modes highlight the algorithm's adaptability for the given two use cases.

\vspace{-4mm}
\begin{enumerate}
    \item \textbf{Best Precision for a given degree} By adjusting the parameter $\epsilon$, which signifies our tolerance for deviations from $M(e^{jw})$, we can refine the approximation's accuracy. This method is particularly valuable when we need to find the best possible polynomial representation of $M(e^{jw})$ for a given degree.
    \vspace{-2mm}
    \item \textbf{Lowest Degree for a given precision} In contrast, we can ask for the lowest degree polynomial which achieves a certain precision level $\epsilon$. This mode is advantageous when the priority is to minimize computational overhead or when we need a simpler polynomial approximation, as long as the approximation remains within acceptable accuracy bounds
    \vspace{-3mm}
\end{enumerate}

\subsection{Obtaining State-Space Controllers} \label{app:from L to K in state space params}

Note that given the polynomial z-spectra, we require its spectral factorization to obtain the state-space controller that approximates the \DRO~controller. The following Lemma introduces a simple way to obtain such an approximation

\begin{lemma}[{Canonical factor of polynomial z-spectra \citep[Lem. 1]{sayed_survey_2001}}]\label{lem:spect}
    Consider a Laurent polynomial of degree $m$, $P(z)=\sum_{k=-m}^{m} p_k z^{-k}$, with $p_k = p_{-k} \in \R$, such that $P(z)> 0$. Then, there exists a canonical factor $L(z)= \sum_{k=0}^{m} \ell_k z^{-k}$ such that $P(z) = \abs{L(z)}^2$ and $L(z)$ has all of its root in $\TT$.
\end{lemma}

\vspace{-2mm}
Using \cref{lem:spect}, we can compute spectral factors by factorizing the symmetric positive polynomials and multiplying all the factors with stable roots together. Consequently, this rational spectral factor enables the derivation of a rational controller, denoted as $K(z)$ (refer to \cref{app:from L to K in state space params}).

Now we present the \DRO~controller in state-space form.

\begin{lemma}\label{lemma:state space K}
 Let $\Tilde L(z)$ 
 be the rational factor of the spectral factorization $\Tilde N(z)=\Tilde L(z)^\ast \Tilde L(z) \= P(z)/Q(z)$  of a degree $m$ rational approximation $P(z)/Q(z)$. The controller obtained from $\Tilde L(z)$ using \eqref{eq:optimal K and M}, \ie,  $K(z) \= K_{\Htwo}(z) \+  \Delta(z)^{-1}\cl{  \{\Delta(z) K_\circ(z) \}_{-} \Tilde L(z)}_{+} \Tilde L(z)^\inv$ is rational and can be realized as a state-space controller as follows: 
\vspace{-2mm}
    \begin{equation}\label{eq:RationalK}
        e(t+1)=\widetilde{F}e(t)+\widetilde{G}w(t), \quad u(t)=\widetilde{H}e(t)+ \widetilde{J}w(t)) 
        \vspace{-2mm}
    \end{equation}
    where $e_t$ is the controller state, and $(\widetilde{F},\widetilde{G},\widetilde{H},\widetilde{J})$ are determined from $(A,B_u,B_w)$ and $\Tilde L(z)$.
\end{lemma}
\vspace{-4mm}

\section{{Numerical Experiments}} \label{sec:simul}
In this section, we present the performance of the \DRO~controller, compared to $ H_2$, $ H_\infty$,  regret-optimal and other finite-horizon DR controllers. We present  frequency domain and time-domain evaluations, and we showcase the performance of the rational approximation method. We employ benchmark models such as [REA4], [AC15], and [HE3] from \cite{aircraft}. In the frequency domain simulations, results for [REA4] and [HE3] are presented. In the time domain simulations for the aircraft model [AC15] are presented, with additional simulations provided in Appendix \ref{ap:sim}. The [REA4] is a chemical reactor model and [HE3] is a helicopter model with 8 states each. The [AC15] is an aircraft model with 4 states. We perform all experiments using MATLAB, on an Apple M1 processor with 8 GB of RAM. We specify the nominal distribution as a Gaussian, with zero mean and identity covariance.
\vspace{-4mm}
\subsection{Frequency Domain Evaluations}

We investigate the behaviour of the \DRO~controller and its rational approximation for various values of the radius $r$.

To show the behavior of the worst-case disturbance we plot its power spectrum $N(e^{j\omega})$ for three different values of the radius $r$ for the [AC15] system in Figure \ref{fig:fig_m_freq_1}. As can be seen for $r=0.01$, the worst-case disturbance is almost white, since that is the case for the nominal disturbance. As $r$ increases, the time correlation of the worst-case disturbance increases, and the power spectrum becomes peaky. 

For the [AC15] system, the worst-case expected regret cost, as outlined in \eqref{def:worst case regret}, for \DRO, the $\mathit H_2$, $\mathit H_\infty$, and RO controllers. are depicted in Figure \ref{fig:ER}. 
We observe that for smaller $r$, the \DRO~performs close to the $\mathit{H}_2$ controller. However, as $r$ increases, the worst-case regret is close to the regret achieved by the RO controller. Throughout the variation in $r$, the \DRO~achieves the lowest worst-case expected regret among all the other mentioned controllers.

To implement the \DRO~controller in practice, we need a rational controller. We find the rational approximation of $N(\ejw)$ as $\frac{P(\ejw)}{Q(\ejw)}$ using the method of Section \ref{sec:rational_approx} for [AC15] and degrees $m=1,2,3$. The performance of the resulting rational controllers is compared to the non-rational \DRO~in \cref{table:worstregret}. As can be seen, the rational approximation with an order greater than 2 achieves an expected regret that well matches that of the non-rational for all values of $r$. 

\begin{table}[ht]
    \centering
    \tiny
    \setlength\tabcolsep{4pt} 
    \begin{tabular}{|c||c|c|c|c|c|c|c|} 
        \hline
        \textbf{ } & \textbf{r=0.01} & \textbf{r=1} & \textbf{r=1.5} & \textbf{r=2} & \textbf{r=3} \\
        \hline \hline
        $DRRO$ & 59.16 &302.08 & 488.57 &718.20 &1307.12  \\
        \hline
        \textbf{RA(1)} & 60.49 &33394.74 &4475.70 &9351.89  &2376.77 \\
        \hline
        \textbf{RA(2)} &   59.58 &303.33 &491.75 &723.96  &1318.98
  \\
        \hline
        \textbf{RA(3)} &  59.57 &302.41 &489.49 &719.72 &1309.85
 \\
        \hline
    \end{tabular}
        \caption{The worst-case expected regret of the non-rational \DRO~controller, compared to the rational controllers RA(1), RA(2), and RA(3), obtained from degree 1, 2, and 3 rational approximations to $N(e^{j\omega})$.}
    \label{table:worstregret}
    \vspace{-4mm}
\end{table}

\subsection{Time Domain Evaluations}

We compare the time-domain performance of the infinite horizon \DRO~controller to its finite horizon counterparts, namely \DRO~and DR-LQR, as outlined in \cite{cornell_drro_old}.
The latter controllers are computed through an SDP whose dimension scales with the time horizon.  We plot the average LQR cost over 210 time steps, aggregated over 1000 independent trials. 
Figure \ref{fig:figa} illustrates the performance of DR controllers under white Gaussian noise, while \ref{fig:figb}, \ref{fig:figc}, and \ref{fig:figd} demonstrate responses to worst-case noise scenarios dictated by each of the controllers, using $r=1.5$. For computational efficiency, the finite horizon controllers operate over a horizon of only $s=30$ steps and are re-applied every $s$ steps. Their worst-case disturbances in \ref{fig:figc} and \ref{fig:figd} are also generated every $s$ steps, resulting in correlated disturbances only within each $s$ steps.  Our findings highlight the infinite horizon \DRO~controller's superior performance over all four scenarios. Note that extending the horizon of the SDP for longer horizons to come closer to the infinite horizon performance is extremely computationally inefficient. These underscore the advantages of using the infnite horizon \DRO~controller.
\vspace{-3mm}

\section{{Future Work}} \label{sec:conc}
Our work presents a complete framework for solving the DR control problem in the full-information setting. Future generalizations would address our limitations. One is to extend the rational approximation method from single to multi input systems. Another is to extend the results to partially observable systems where the state is not directly accessible. Finally, it would be useful to incorporate adaptation as the controller learns disturbance statistics through observations.

\newpage

\section*{Impact Statement}

This paper presents work whose goal is to advance the field of Machine Learning. There are many potential societal consequences of our work, none which we feel must be specifically highlighted here.

\bibliography{refs}
\bibliographystyle{icml2024}

\newpage
\appendix
\onecolumn
\begin{center}
{\huge Appendix}
\end{center}

\section{Organization of the Appendix} \label{ap:org}
This appendix is organized into several sections:

First, \cref{ap:defs} provides notations, definitions, and remarks about the problem formulation and uniqueness of the spectral factorization.

Next, \cref{ap:pkkt} contain proofs of the duality and optimality theorems  in \cref{sec:theory}.

Subsequently, \cref{ap:fixed} is dedicated to proofs of lemmas and theorems related to the efficient algorithm discussed in \cref{sec:fixed_point}, and \cref{ap:alg} describes the pseudo-code of the algorithm.

Further, \cref{ap:rational} contains the proof of the state-space representation of the controller presented in \cref{sec:state_space}.

Finally, additional simulation results are presented in \cref{ap:sim}.

\section{Notations, Definitions and Remarks} \label{ap:defs}
\subsection{Notations}
In the paper, we use the notations in \cref{tab:notation} for brevity. 
\begin{table}[ht]
\centering
\begin{tabular}{|c|l|}
\hline
\textbf{Symbol} & \textbf{Description} \\ \hline
\( x_t \) & State at time \( t \) \\ \hline
\( s_t \) & Regulated output at time \( t \) \\ \hline
\( u_t \) & Control input at time \( t \) \\ \hline
\( w_t \) & Exogenous disturbance at time \( t \) \\ \hline
\( A \) & State transition matrix \\ \hline
\( B_u \) & Control input matrix \\ \hline
\( B_w \) & Disturbance input matrix \\ \hline
\( C \) & Regulated output matrix \\ \hline
\( R \) & Control input cost matrix \\ \hline
\( \mathcal{F}_T \) & Finite-horizon operator for control input \\ \hline
\( \mathcal{G}_T \) & Finite-horizon operator for disturbance \\ \hline
\( \mathcal{F} \) & Infinite-horizon operator for control input \\ \hline
\( \mathcal{G} \) & Infinite-horizon operator for disturbance \\ \hline
\( \| \cdot \| \) & Euclidean norm \\ \hline
\( \| \cdot \|_2 \) & $\Htwo$ (Frobenius) norm \\ \hline
\( \| \cdot \|_\infty \) & \( \textit H_\infty \) (operator) norm \\ \hline
\( \mathbb{E} \) & Expectation \\ \hline
\( \causal \) & Set of causal (online) and time-invariant DFC policies \\ \hline
\( \causal_T \) & Set of causal DFC policies over a horizon \( T \)\\ \hline
\( \mathcal{R}_\K \) & Regret operator \\ \hline
\( \mathcal{M} \) & Auto-covariance operator for disturbances \\ \hline
\( \mathcal{L} \) & Unique, causal and causally invertible spectral factor of $\M =\L \L^\ast$ \\ \hline
\( \mathcal{N} \) &  The unique positive definite operator equal to $\L^\ast \L$ \\ \hline
\( \Was \) & Wasserstein-2 metric \\ \hline
\( \mathcal{S} \) & Symmetric positive polynomial matrix \\ \hline
\( \text{tr}(\cdot) \) & Trace of a Toeplitz operator \\ \hline
\( \{ \cdot \}_+ \) & Causal part of an operator \\ \hline
\( \{ \cdot \}_- \) & Strictly anti-causal part of an operator \\ \hline
\( \sqrt{\M}, \textrm{ or } \M^{\half} \) & Symmetric positive square root of an operator or matrix \\ \hline
\( \Sym_n^+ \) & The set of positive semidefinite matrices \\ \hline
\( \trig_{m,+} \) & The set of positive trigonometric polynomials of degree $m$ \\ \hline
\end{tabular}
\caption{Notation Table}
\label{tab:notation}
\end{table}

\subsection{Explicit Form of Finite-Horizon State Space Model }\label{app:explicit H and L}
Consider the restrictions of the infinite-horizon dynamics in \eqref{eq: infinite horizon model} to the finite horizon as \begin{equation}\label{eq:ouff}
    \s_T=\F_T \u_T+ \G_T \w_T.
\end{equation}
The causal linear measurement model for the finite-horizon case in \eqref{eq:ouff} can be stated explicitly as follows:
\begin{equation}
    \underbrace{\begin{bmatrix}
        s_{0} \\ s_{1} \\ s_{2} \\ \vdots  \\ s_{T}
    \end{bmatrix}}_{\s_T} = 
    \underbrace{\begin{bmatrix}
        0    & 0   & 0  & \dots & 0  \\
        B_u   & 0  & 0  & \dots & 0  \\
        A B_u & B_u & 0 & \dots & 0  \\
        \vdots & \vdots & \vdots & \ddots & \vdots \\
         A^{T-1} B_u & A^{T-2} B_u & A^{T-3} B_u & \ddots & 0  
    \end{bmatrix}}_{\F_T}
    \underbrace{\begin{bmatrix}
         u_{0} \\ u_{1} \\ u_{2} \\ \vdots  \\ u_{T}
    \end{bmatrix}}_{\u_T} + 
    \underbrace{\begin{bmatrix}
        0    & 0   & 0  & \dots & 0  \\
        B_w   & 0  & 0  & \dots & 0  \\
        A B_w & B_w & 0 & \dots & 0  \\
        \vdots & \vdots & \vdots & \ddots & \vdots \\
         A^{T-1} B_w & A^{T-2} B_w & A^{T-3} B_w & \ddots & 0  
    \end{bmatrix}}_{\G_T}
    \underbrace{\begin{bmatrix}
         w_{0} \\ w_{1} \\ w_{2} \\ \vdots  \\ w_{T}
    \end{bmatrix}}_{\w_T}
\end{equation}

\subsection{A Note about Robustness to Disturbances vs Robustness to Model Uncertainties}\label{remark:dist}
In our approach, we consider distributional robustness against disturbances, which provides flexibility, adaptability, and dynamic responses to unforeseen events. Although we do not explicitly address model uncertainties, these uncertainties can be effectively lumped together as disturbances—a technique known as uncertainty/disturbance lumping. This approach is particularly effective when the model uncertainties are relatively small. By treating parameter uncertainties as disturbances, we simplify system analysis and ensure that the controller is robust not only to known uncertainties but also to unexpected variations and modeling errors.

\subsection{A note about the Uniqueness of the spectral factor $\L$}\label{remark:unique L}
In \cref{thm:kkt}, given that $\L$ is the causal and causally invertible spectral factor of $\M=\L \L^\ast$, it is unique up to a unitary transformation of its block-elements from the right. Fixing the choice of the unitary transformation in the spectral factorization (eg. positive-definite factors at infinity \citep{ephremidze_algorithmization_2018}) results in a unique $\L$.


\section{Proof of Optimality Theorems}\label{ap:pkkt}

\subsection{Proof of \cref{thm:strong_duality}}\label{ap:dual}

This result is proven in detail in \citet[Appendix A1]{kargin2023wasserstein}. Due to its length, we provide only a brief sketch here. Interested readers can refer to \citet{kargin2023wasserstein} for the complete proof.For completeness, we provide the following proof sketch.

First, we provide a finite-horizon counterpart of the strong duality result from \cite{DRORO}. Then, we reformulate the objective functions of both the finite-horizon and infinite-horizon dual problems using normalized spectral measures. We demonstrate the pointwise convergence of the finite-horizon dual objective function to the infinite-horizon objective by analyzing the limiting behavior of the spectrum of Toeplitz matrices. Finally, we show that the infinite-horizon dual problem attains a finite value and that the limits of the optimal values (and solutions) of the finite-horizon dual problem coincide with those of the infinite-horizon dual problem.

\subsection{Proof of Theorem \ref{thm:kkt}}
The proof involves four main steps.
\begin{itemize}
    \item \textbf{Reformulation using Lemma \ref{thm:full_optimality}}: Using Lemma \ref{thm:full_optimality}, we reformulate the original optimization problem. This lemma allows the expression of the convex mapping $\clf{X} \mapsto \tr(\clf{X}^{-1})$ using Fenchel duality, which transforms the objective function into a form that involves the supremum over a positive semi-definite matrix $\M$ and the only term depending on $\K$ remains $\tr(\RR_{\K} \M)$.
    
    \item \textbf{Application of Wiener-Hopf Technique}: We then Lemma \ref{lemma:wiener}, which provides a method to approximate a non-causal controller by a causal one, minimizing the cost $\tr(\RR_{\K} \M)$. The optimal causal controller $\K_\star$ is derived using the Weiner-Hopf Technique.
    
    \item \textbf{Karush-Kuhn-Tucker (KKT) Conditions}: We then find the conditions on the optimal $\M$. This involves simplifying the objective function and finding the optimal $\M_{\gamma}$ and $\K_{\gamma}$ for the level $\gamma$.
    
    \item \textbf{Final Reformulation and Duality}: We further simplify the problem and apply strong duality to achieve the final form. The optimal $\K_\star$ is then derived from the Wiener-Hopf technique, with $\gamma_\star$ and $\M_\star$ obtained through duality arguments.
\end{itemize}

Before proceeding with the proof, we first state two useful lemmas
\begin{lemma}\label{thm:full_optimality}
    The convex mapping $\clf{X}\! \mapsto\! \tr \clf{X}^{\inv}$ for $\clf{X}\!\psdg\! 0$ can be expressed via Fenchel duality as
\begin{equation}\label{eq:trace_of_inverse}
   \sup_{\M \psdg 0 } -\tr(\clf{X} \M ) + 2\tr(\sqrt{\M}) = \begin{cases} \tr(\clf{X}^{-1}), & \quad\text{if } \clf{X}\psdg 0 \\  +\infty, &\quad \text{o.w.} \end{cases}
\end{equation}
 
\end{lemma}
\begin{proof}
  Observe that the objective $-\tr(\clf{X} \M ) + 2\tr(\sqrt{\M})$ is concave in $\M$, and the expression on the right-hand side can be obtained by solving for $\M$. When $\clf{X} \not\succeq 0$, \ie, $\clf{X}$ may have negative eigenvalues, then the expression $\tr(\clf{X} \M )$ can be made arbitrarily negative, and $\tr(\sqrt{\M})$ arbitrarily large, by chosing an appropriate $\M$.
\end{proof}

The following lemma will be useful in the proof of Theorem \ref{thm:kkt}.
    \begin{lemma}[{Wiener-Hopf Technique \citep{kailath_linear_2000}}]\label{lemma:wiener}
    Consider the problem of approximating a non causal controller $\K_\circ$ by a causal controller $\K$, such that $\K$ minimises the cost $\tr(\RR_{\K} \M )$, i.e.,
    \begin{align}
        \inf_{\K \in\causal} \tr(\RR_{\K} \M )
    \end{align}
    where $\M \psdg 0$, $\RR_{\K} = \left(\K - \K_\circ \right)^{\ast}\Delta^{\ast}\Delta\left(\K - \K_\circ \right)$ and $\K_\circ$ is the non-causal controller that makes the objective defined above zero. Then, the solution to this problem is given by 
    \begin{align}
        \K_\star = \Delta^\inv\cl{\Delta \K_\circ \L}_{\!+} \L^\inv,
        \label{eq::nehari_controller}
    \end{align}
    where $\L$ is the unique causal and causally invertible spectral factor of $\M$ such that $\M = \L \L^\ast$ and $\cl{ \cdot }_{\!+}$ denotes the causal part of an operator.
    Alternatively, the controller can be written as,
    \begin{align}
        \K_\star = \K_{\Htwo} +  \Delta^{-1}\cl{  \{\Delta \K_\circ \}_{-} \L_\star}_{+},
        \label{eq::nehari_controller_h2}
    \end{align}
    where $\K_{\Htwo}\!\defeq\!\Delta^\inv \{\Delta \K_\circ \}_{\!+}$.
    \label{lemma::nehari}
\end{lemma}
\begin{proof}
    Let $\L$ be the unique causal and causally invertible spectral factor of $\M$, $\ie$ $\M = \L \L^{\ast}$. Then, using the cyclic property of $\tr$, the objective can be written as,
    \begin{align}
        \inf_{\K \in\causal} \tr( \Delta\left(\K - \K_\circ \right) \M \left(\K - \K_\circ \right)^{\ast}\Delta^{\ast}) &= \inf_{\K \in\causal} \tr( \left(\Delta\K - \Delta\K_\circ \right) \L \L^{\ast}\left(\Delta\K - \Delta\K_\circ \right)^{\ast}) \\
        &= \inf_{\K \in\causal} \tr( \left(\Delta\K\L - \Delta\K_\circ\L \right) \left(\Delta\K\L - \Delta\K_\circ\L \right)^{\ast}) \\
        &= \inf_{\K \in\causal} \left\| \Delta\K\L - \Delta\K_\circ \L \right\|^2_{2}.
    \end{align} 
 Since $\Delta, \K$ and $\L$ are causal, and $\Delta\K_\circ\L$ can be broken into causal and non-causal parts, it is evident that the controller that minimizes the objective is the one that makes the term $\Delta\K\L - \Delta\K_\circ \L$ strictly anti-causal, cancelling off the causal part of $\Delta\K_\circ \L$. This means that the optimal controller satisfies,
 \begin{align}
    \Delta\K_\star\L = \cl{\Delta \K_\circ \L}_{\!+}.
\end{align}
Also, since $\L^\inv$ and $\Delta^\inv$ are causal, the optimal causal controller is given by \eqref{eq::nehari_controller}. Finally, using the fact that $\Delta \K_\circ = \cl{\Delta \K_\circ}_{\!+} + \cl{\Delta \K_\circ}_{\!-}$ and simplifying, we get \eqref{eq::nehari_controller_h2}.
\end{proof}

\paragraph{\textit{Proof of Theorem \ref{thm:kkt}.}}
We first simplify our optimization problem \eqref{eq:DR-RO} using Lemma \ref{thm:full_optimality}. We then find the conditions on the optimal optimization variables using Karush-Kuhn-Tucker (KKT) conditions. Using Lemma \ref{thm:full_optimality}, we can write,
\begin{align*}
       \inf_{\substack{\K \in\causal, \\ \gamma \I \psdg \RR_{\K} }} \tr(  ( \I - \gamma^\inv \RR_{\K} )^{\-1} \M_{\circ}) &=  \inf_{\K \in\causal}  \sup_{\M \psdg 0 } -\tr((\I - \gamma ^\inv \RR_{\K} ) \M ) + 2\tr\left(\sqrt{\M^{\half}_{\circ}\M \M^{\half}_{\circ}}\right) \\
       &=  \sup_{\M \psdg 0 }   -\tr(\M)  + 2\tr\left(\sqrt{\M^{\half}_{\circ}\M \M^{\half}_{\circ}}\right) +\inf_{\K \in\causal}\gamma^\inv\tr(\RR_{\K} \M )
\end{align*}

Fixing $\gamma\geq 0$, we focus on the reduced subproblem of \eqref{eq:dual_worst_case_regret},
\begin{equation}\label{eq: subproblem for causal infinite horizon}
    \sup_{\M \psdg 0 }   -\gamma \tr(\M) -\gamma \tr(\M_\circ)  + 2\gamma\tr\left(\sqrt{\M^{\half}_{\circ}\M \M^{\half}_{\circ}}\right) +\inf_{\K \in\causal}\tr(\RR_{\K} \M ).
\end{equation}
Using the definition of the Bures-Wasserstein distance, we can reformulate \eqref{eq: subproblem for causal infinite horizon} as 
\begin{equation}
    \sup_{\M \psdg 0}  \inf_{\K \in\causal} \tr(\RR_{\K}\M ) \- \gamma\BW(\M, \M_\circ)^2 \coloneqq \sup_{\M \psdg 0} \Phi(\M).
    \label{eq:sup_inf_m}
\end{equation}
Thus, the original formulation in \eqref{eq:dual_worst_case_regret} can be expressed as
\begin{equation}
    \inf_{ \gamma \geq 0 } \sup_{\M \psdg 0}  \inf_{\K \in\causal}    \tr(\RR_{\K}\M ) + \gamma\pr{r^2- \BW(\M, \M_\circ)^2}.
\end{equation}
Note that the objective above is affine in $\gamma\geq0$ and strictly concave in $\M$. Moreover, primal and dual feasibility hold, enabling the exchange of $ \inf_{ \gamma \geq 0 } \sup_{\M \psdg 0} $ resulting in
\begin{align}
    \sup_{\M \psdg 0}  \inf_{\K \in\causal}   \inf_{ \gamma \geq 0 }  \tr(\RR_{\K}\M ) + \gamma\pr{r^2- \BW(\M, \M_\circ)^2},
\end{align}
where the inner minimization over $ \gamma$ reduces the problem to its constrained version in \cref{eq:maxmin}.

Finally, the form of the optimal $\K_\star$ follows from the Wiener-Hopf technique in \cref{lemma:wiener} and the optimal $\gamma_\star$ and $\M_\star$ can be obtained using the strong duality result in \cref{ap:dual}. To see the optimal form of $\M_\star$, consider the gradient of $\Phi(\M)$ in \eqref{eq:sup_inf_m} with respect to $\M$ and setting it to $0$. Using Danskin theorem \cite{danskin}, we have,
\begin{align}
        \nabla \Phi(\M) = \M^{\half}_{\circ} \left(\M_\circ^{\half} \M_\star \M_\circ^{\half}\right)^{-\half} \M^{\half}_{\circ} - \I + \gamma^{-1}\RR_{\K_\star} = 0.
\end{align}

Taking inverse on both sides, we get,
\begin{align}
          \M^{-\half}_{\circ} \left(\M_\circ^{\half} \M_\star \M_\circ^{\half}\right)^{\half} \M^{-\half}_{\circ} = \left(\I - \gamma^{-1}\RR_{\K_\star}\right)^{-1}.
\end{align}
We can now obtain two equations. First, by right multiplying by $\M^{\half}_\circ$ and second, by left multiplying by $\M^{\half}_\circ$. On multiplying these two equations together and simplifying, we get \eqref{eq:optimal M}.

\section{Proofs related to the Efficient Algorithm in \cref{sec:fixed_point}}\label{ap:fixed}

\subsection{Proof of \cref{lemma:kkt2}}
With $\T \defeq \{\Delta\K_\circ  \}_{-}$, the optimality condition in \eqref{eq:optimal K and M} takes the equivalent form:
\begin{subequations}
    \begin{align}
    &\textit{i.}\; \M_\star =  \pr{ \I - \gamma_\star^\inv \mathcal R_{\K_\star} }^{-2},  \\
    &\textit{ii.}\; \mathcal R_{\K_\star}\= \L_\star^{\-\ast} \cl{ \T\L_\star  }_{\-}^\ast \cl{ \T \L_\star}_{\-} \L_\star^{\-1} ,  \\
    &\textit{iii.}\; \tr\br{\pr{(\I - \gamma_\star^\inv \mathcal R_{\K_\star} (z) )^\inv - \I}^2 } = r^2, 
\end{align}
\end{subequations}
Using the spectral factorization $\M_\star \L_\star \L_\star^\ast$, the conditions $\textit{i.}$ and $\textit{ii.}$ can be equivalently re-expressed as
\begin{align}
    &\textit{i.}\;(\L_\star  \L_\star^\ast)^{-1/2} =   I - \gamma_\star^\inv \mathcal R_{\K_\star} \\
    &\textit{ii.}\;  \mathcal R_{\K_\star} \= \L_\star^{\-\ast}\cl{ \T\L_\star  }_{\-}^\ast \cl{ \T\L_\star  }_{\-}  \L_\star^{\-1} 
\end{align}
By plugging $\textit{ii.}$ into $\textit{i.}$, we get
\begin{align}
      0=\I - (\L_\star  \L_\star^\ast)^{-1/2} -  \gamma_\star^\inv \pr{\L_\star^{\-\ast}\cl{ \T\L_\star  }_{\-}^\ast \cl{ \T\L_\star  }_{\-}   \L_\star^{\-1} } = 0,
\end{align}
Multiplying by $\L_\star^\ast$ from the left and by $\L_\star$ from the right, we get 
\begin{align*}
   0&= \L_\star^\ast \L_\star \-(\L_\star^\ast \L_\star)^{1/2} \-  \gamma_\star^\inv \cl{ \T\L_\star  }_{\-}^\ast\cl{ \T\L_\star  }_{\-} ,
\end{align*}
where we used the identity $\L_\star^\ast (\L_\star \L_\star^\ast)^{-1/2} \L_\star = (\L_\star^\ast \L_\star)^{1/2}$. Letting  $\NN_\star=\L_\star^\ast \L_\star$, this expression can be solved for $\N_\star$, yielding the following implicit equation,
\begin{align}
    \NN_\star=\L_\star^\ast \L_\star  = \frac{1}{4}\pr{\I + \sqrt{ \I + 4\gamma_\star^\inv \cl{ \T\L_\star  }_{\-}^\ast\cl{ \T\L_\star  }_{\-}  }}^2,
\end{align}
implying thus \eqref{eq:optimal N sqrt form}, with $\gamma_\star\>0$ satisfying $\tr\br{((\I-\gamma_\star^\inv \RR_{\K_\star})^\inv- \I)^2} = r^2$ (or equivalently, $ \BW(\L_\star \L_\star^\ast,\I )=r$).
\qed

\subsection{Note on Frequency Domain Representation of Toeplitz Operators}
We start this section of the appendix by justifying our choice of working out our results in the frequency domain.

Let $\mathcal{V}=[V_{ij}]_{i,j=-\infty}^\infty$ be a doubly infinite block matrix, \ie a \emph{Toeplitz} operator, which represents a discrete, linear, time-invariant system (\ie  $V_{ij}=V_{i-j}$), and which maps a sequence of inputs to a sequence of outputs.

In this case of a time-invariant system, the representation of the operator in the  z-domain (or the so-called bilateral z-transform) is
\begin{equation}
    V(z)=\sum_{i=-\infty}^\infty V_i z^{-i},
\end{equation}
defined for the regions of the complex plane where the above series converges
absolutely, known as the ROC: region of convergence. $V(z)$ is also known as the transfer matrix. The causality of $\mathcal V$ can be readily given
in terms of $V(z)$. Indeed, we have the following: $\mathcal V$ is causal if and only if $V(z)$ is analytic in the exterior of some
annulus, $|z| > \alpha > 0$. Likewise, $\mathcal V$ is anticausal if and only if $V(z)$ is analytic
in the interior of some annulus, $|z| < \alpha < 0$. Moreover, $\mathcal V$ is strictly causal
(anticausal) if and only if it is causal (anticausal) and $V(\infty) = 0 (V(0) = 0)$.

We also define the trace of a Toeplitz operator $\M$ as follows
\begin{equation}
    \tr(\M) = \frac{1}{2\pi} \int_{0}^{2\pi} \Tr(M(\ejw)) d\omega.
\end{equation}

In the coming sections, we use the frequency domain counterparts of our Toeplitz operators (such as $\F, \G, \M...$) by setting $z=e^{j\omega}$ for $\omega\in[0,2\pi)$.

\subsection{Frequency-Domain Characterization of the Optimal Solution of \cref{prob:DR-RO}} \label{app: frequency domain details}

We present the frequency-domain formulation of the saddle point  $(\K_\star,\M_\star)$ derived in \cref{thm:kkt} to reveal the structure of the solution. We first introduce the following useful results:

Denoting by $M_\star(z)$ and $R_{K_\star}(z)$ the transfer functions corresponding to the optimal $\M_\star$ and $\mathcal R_{\K_\star}$, respectively, the optimality conditions in \eqref{eq:optimal K and M} and \eqref{eq:optimal N sqrt form} take the equivalent forms:
\begin{subequations}\label{eq: frequency domain optimal M and K}
    \begin{align}
    &\textit{i.}\; M_\star(z) =  \pr{ I - \gamma_\star^\inv R_{K_\star}(z) }^{-2},  \label{eq: optimal M(z)}\\
    &\textit{ii.}\; R_{K_\star}(z) \= L_\star(z)^{\-\ast} \cl{ \T\L_\star  }_{\-}(z)^\ast \cl{ \T \L_\star}_{\-}(z) L_\star(z)^{\-1} ,  \label{eq: optimal K(z)}\\
    &\textit{iii.}\; \tr\br{\pr{(I - \gamma_\star^\inv R_{K_\star} (z) )^\inv - I}^2 } = r^2,  \label{eq: optimal gamma star},\\
    &\textit{iv.}\;  N_\star(z)=L_\star(z)^\ast L_\star(z)  = \frac{1}{4}\pr{I + \sqrt{ I + 4\gamma^\inv \cl{ \T\L_\star  }_{\-}(z)^\ast\cl{ \T\L_\star  }_{\-} (z) }}^2,\label{eq:Nfreq App}
\end{align}
\end{subequations}
where \begin{equation}\label{eq:L}
    L_\star(z)=\sum_{t=0}^{\infty} \widehat{L}_{\star,t} z^{- t}
\end{equation} is the transfer function corresponding to the causal canonical factor $\L_\star$ and $\T = \{\Delta\K_\circ  \}_{-}$ is the strictly anticausal operator where its transfer function, $T(z)$, is found from the following: 

\begin{lemma}[Adapted from lemma 4 in \cite{sabag2021regret}] The transfer function $\Delta(z) K_\circ(z)$ can be written as the sum of a causal and strictly anticausal transfer functions:
\begin{align}
     &\Delta(z) K_\circ(z)= T(z)+U(z)\label{eq:deltaK}\\
 &T(z)= \{\Delta \K_\circ \}_{\-}(z)=\overline{C}{(z^{-1}I-\overline{A})}^{-1}\overline{B} \label{eq:T}\\
 &U(z)=\{\Delta \K_\circ \}_{\+}(z)=\Delta(z)K_{H_2}(z)=\overline{C}P (A {(z I-A)}^{-1} + I)B_w\label{eq:U}
\end{align}

where $K_{H_2}(z)=\Delta^{\-1}\{\Delta \K_\circ\}_{\+}(z)$, and $K_{\textrm{lqr}}\!\defeq\!{(I\+B_u^\ast PB_u)}^{\inv}B^\ast_u P A$, with $P \!\succ \!0$ is the unique stabilizing solution to the LQR Riccati equation $P=Q+A^\ast PA-A^\ast PB_u{(I+B_u^\ast PB_u)}^{-1}B_u^\ast PA,$ $Q=C^\tp C,$ $A_k=A\-B_u K_{\textrm{lqr}}$, and
\begin{equation}\label{eq:newABC}
    \begin{aligned}
        \overline{A}=A_k^\ast, \quad \overline{B}=A_k^\ast P B_w, \quad \overline{C}=- (I+B_u^\ast P B_u)^{-\ast/2} B_u^\ast.
    \end{aligned}
\end{equation}
\end{lemma}

Notice that given the causal $L(z)$ and strictly anti-causal $T(z)$, the strictly anti-causal part $\{T(z)L(z)\}_{-}$ has a state space representation, shown in the following lemma. 

\begin{lemma}\label{eq: finite parameter of anticausal (TL)}
Let $\L$ be a causal operator. The strictly anti-causal operator $\cl{\T \L}_{\-}$ possesses a state space representation as follows: 
    \begin{equation}\label{eq:plug S}
         \cl{\T  \L}_{\-}(z)=\overline{C}(z^{-1}I-\overline{A})^{-1}\Gamma,
    \end{equation}
    where 
    \begin{equation}
         \Gamma = \frac{1}{2\pi} \int_{0}^{2\pi} (I-\ejw \overline{A})^\inv \overline{B} L(\ejw) d\omega.
    \end{equation}
\end{lemma}

\begin{proof}
Let $L(z)= \sum_{t=0}^{\infty} \widehat{L}_t z^{-t}$ be the transfer function of $\L$. Using equations \eqref{eq:deltaK} and \eqref{eq:T},\eqref{eq:U}, $S(z)\defeq \cl{\Delta \K_\circ \L}_{\-}(z)$, can be written as: 
\begin{align}
     S(z) &=  \cl{T  L}_{\-}(z) +  \cl{U L}_{\-}(z)\\
    &\stackrel{(a)}{=} \cl{\overline{C}(z I - \overline{A})^\inv \overline{B}L(z)}_{\-}\\
    &\stackrel{(b)}{=} \cl{\overline{C}\sum_{t=0}^{\infty} z^{ (t+1)} \overline{A}^t \overline{B} \sum_{m=0}^{\infty}\widehat{L}_m z^{- m}}_{\-}\label{eq:ac}\\
    &\stackrel{(c)}{=} {\overline{C}\left(\sum_{t=0}^{\infty} z^{(t+1)} \overline{A}^t \right) \left(\sum_{m=0}^{\infty}\overline{A}^m \overline{B}\widehat{L}_m\right) }\\
    &\stackrel{(d)}{=}\overline{C}(z^{-1}I-\overline{A})^{-1}\Gamma \label{eq:ad}
\end{align}
Here, (a) holds because both $U(z)$ and $L(z)$ are causal, so the strictly anticausal part of $U(z) L(z)$ is zero. (b) holds as we do the Neumann series expansion of $(z I-\overline{A})$ and replace $L(z)$ by its equation \eqref{eq:L}. (c) holds as we take the anticausal part of expression \eqref{eq:ac} to be the strictly positive exponents of $z$. (d) completes the result by using the Neuman series again, defining $\Gamma \defeq \sum_{t=0}^{\infty} \overline{A}^t \overline{B} \widehat{L}_t $, and leveraging Parseval's theorem to conclude the equation of the finite parameter
\begin{equation}\label{eq:finite parameter Gamma}
    \Gamma = \frac{1}{2\pi} \int_{0}^{2\pi} (I-\ejw  \overline{A})^\inv \overline{B} L(\ejw) d\omega.
\end{equation}   
\end{proof}

\paragraph{\textit{Proof of \cref{thm:fixed_point}:}}
Using \cref{eq: finite parameter of anticausal (TL)}, and plugging \eqref{eq:plug S} into \eqref{eq:Nfreq App}, the frequency-domain optimality equation \eqref{eq:Nfreq App} can be reformulated explicitly as follows
\begin{equation}\label{eq:N for corr}
    N_\star(z)=L_\star(z)^\ast L_\star(z)  = \frac{1}{4}\pr{I + \sqrt{ I + 4\gamma_\star^\inv \Gamma_\star^\ast  (z^{-1}I-\overline{A})^{\-\ast}\overline{C}^\ast \overline{C}(z^{\-1}I-\overline{A})^{\-1}\Gamma_\star }}^2
\end{equation}
where $\Gamma_\star$ as in \eqref{eq:finite parameter Gamma}, and $\gamma_\star\>0$ satisfying $\tr\br{((\I-\gamma_\star^\inv \RR_{\K_\star})^\inv- \I)^2 } = r^2$ (or equivalently, $ \BW(\L_\star \L_\star^\ast,\I )=r$), which gives the desired result. 
\qed

\paragraph{\textit{Proof of \cref{thm:irrational}:}}
Notice that the rhs of \eqref{eq:N for corr} involves the positive definite square-root of the rational term $\Gamma_\star^\ast  (z^{-1}I-\overline{A})^{\-\ast}\overline{C}^\ast \overline{C}(z^{\-1}I-\overline{A})^{\-1}\Gamma_\star $. The square root does not preserve rationality in general, implying the desired result.
\qed

\subsection{Proof of Theorem \ref{thm: convergence of FW}}

Before proceeding with the proof, we state the following useful lemma.
\begin{lemma}\label{lem:useful lemma }
For a positive-definite Toeplitz operator $\M\psdg 0$ with $\tr(\M)< \infty$ and $\tr(\log(\M))>-\infty$, let $\M \mapsto \Phi(\M)$ be a mapping defined as 
\begin{equation}\label{eq:wiener-hopf function_method}
    \Phi(\M) \triangleq \inf_{\K \in \causal} \tr\pr{ \RR_\K \M }.
\end{equation}
Denote by $\M = \L \L^\ast$ and $\Delta \Delta^\ast = \I + \F^\ast \F$ the canonical spectral factorizations where $\L$, $\Delta$ as well as their inverses $\L^\inv$, $\Delta^\inv$ are causal operators. The following statements hold:
\begin{itemize}
    \item[i.] The function $\Phi$ can be written in closed form as 
    \begin{equation}
        \Phi(\M) = \tr\br{ \cl{ \Delta \K_\circ \L }_{-}^\ast \cl{ \Delta \K_\circ \L}_{-}}.
    \end{equation}
    \item[ii.] The gradient of $\Phi$ has the following closed form
    \begin{equation}\label{eq: grad Phi}
        \nabla \Phi(\M) = \RR_{\K} = \L^{-\ast} \cl{ \Delta \K_\circ \L }_{-}^\ast \cl{ \Delta \K_\circ \L }_{-} \L^{-1}.
    \end{equation}
    \item[iii.] The function $\Phi$ is concave, positively homogeneous, and
    \begin{equation}\label{eq: grad Phi2}
         \Phi(\M) = \tr(\M \,\nabla \Phi(\M)).
    \end{equation}
\end{itemize}
\end{lemma}

\paragraph{\textit{Proof of Theorem \ref{thm: convergence of FW}.}}

Our proof of convergence follows closely from the proof technique used in \cite{jaggi_revisiting_2013}. In particular, since the unit circle is discretized and the computation of the gradients are approximate, the linear suboptimal problem is solved upto an approximation, $\delta_N$ which depends on the problem parameters and the discretization level $N$. Namely, for a large enough $N$, we have
\begin{align}
    \tr(\nabla \Phi(\M_k)  \widetilde{\M}_{k+1}) \geq   \sup_{\substack{\M \in  \Omega_r}} \tr(\nabla \Phi(\M_k)  {\M})  - \delta_N
\end{align}
where
\begin{align}
\Omega_r \defeq \{\M\psdg 0 \mid \tr(\M-2\sqrt{\M}+\I) \leq r^2\} ,   
\end{align}
Therefore, using Theorem 1 of \cite{jaggi_revisiting_2013}, we obtain 
\begin{equation}\label{eq:convergence rate_app}
        \Phi(\M_\star) - \Phi(\M_k) \leq \frac{2\kappa}{k+2}(1+\delta_N).
\end{equation}
where $\kappa>0$ is the so-called curvature constant associated with the problem which is defined as follows
\begin{align}
    \kappa &\defeq  \sup_{\substack{\M, \widetilde{\M} \in \Omega_r \\ \eta \in [0,1] \\ \M^\prime = \M + \eta(\widetilde{\M}-\M) }} \frac{2}{\eta^2} \br{-\Phi(\M^\prime) + \Phi(\M) + \tr(\nabla \Phi (\M)\, (\M^\prime - \M)) },\\
    &=\sup_{\substack{\M, \widetilde{\M} \in \Omega_r \\ \eta \in [0,1] \\ \M^\prime = \M + \eta(\widetilde{\M}-\M) }} \frac{2}{\eta^2} \pr{\tr(\M^\prime \nabla \Phi(\M))-\Phi(\M^\prime)},\\
    &=\sup_{\substack{\M, \widetilde{\M} \in \Omega_r \\ \eta \in [0,1] \\ \M^\prime = \M + \eta(\widetilde{\M}-\M) }} \inf_{\K\in\causal} \frac{2}{\eta^2} \tr\pr{\M^\prime ( \nabla \Phi(\M)- \RR_{\K})}
\end{align}
where the last two equalities follow from \cref{lem:useful lemma }.

\section{Algorithms}\label{ap:alg}
\subsection{Pseudocode for Frequency-domain Iterative Optimization Method Solving \eqref{eq:maxmin}}\label{app:detailed pseudocode}
The pseudocode for Frequency-domain tterative optimization method is presented in \cref{alg:fixed_point_detailed}. 
\begin{algorithm}[ht]
   \caption{Frequency-domain iterative optimization method solving \eqref{eq:maxmin}}
   \label{alg:fixed_point_detailed}
\begin{algorithmic}[1]
   \STATE {\bfseries Input:} Radius $r\>0$, state-space model $(A,B_u,B_w)$, discretizations $N\>0$ and $N'\>0$ tolerance $\epsilon\>0$ 
   \vspace{0.5mm}
   \STATE Compute $(\overline{A},\overline{B},\overline{C})$ from $(A,B_u,B_w)$ using \eqref{eq:newABC}
   \vspace{0.5mm}
   \STATE Generate frequency samples $\TT_N \defeq \{\e^{j 2\pi n /N} \mid n \!=\! 0,\dots,N\!-\!1\}$
   \vspace{0.5mm}
   \STATE Initialize $M_{0}(z) \gets I$ for $z\in\TT_N$, and $k\gets 0$
   \vspace{0.5mm}
   \REPEAT 
   \vspace{0.5mm}
   \STATE Set the step size $\eta_{k}\gets \frac{2}{k+2}$
   \vspace{0.5mm}
   \STATE Compute the spectral factor $\begin{aligned}L_{k}(z)\gets \texttt{SpectralFactor}(M_{k})\end{aligned}$ (see \cref{alg:spectral factor method})
   \vspace{0.5mm}
   \STATE Compute the parameter $\begin{aligned}\Gamma_k \gets \frac{1}{N} \suml_{z\in \TT_N}  (I-z \overline{A})^{-1} \overline{B} L_k(z) .\end{aligned}$ (see \cref{app: gradients in fw})
   \vspace{1mm}
   \STATE Compute the gradient for $z\in\TT_N$ (see \cref{app: gradients in fw})\\
   \vspace{0.5mm}
    $\begin{aligned}\quad\quad R_k(z) \gets L_k(z)^{-\ast} \cl{  \Delta \K_\circ \L_k }_{-}(z)^\ast\cl{ \Delta \K_\circ \L_k }_{-}(z) L_k(z)^{-1} \end{aligned}$
    \vspace{1mm}
   \STATE Solve the linear subproblem \eqref{eq:linear Opt} via bisection (see \cref{app:bisection})\\
   \vspace{0.5mm}
    $\begin{aligned}\quad\quad \widetilde{M}_{k}(z) \gets (I - \gamma_k^{-1} R_k(z))^{-2}\end{aligned}$ for $z\in\TT_N$ and $\gamma_k \text{ through } \texttt{Bisection}$  
    \vspace{0.5mm}
    \STATE Set $M_{k+1}(z)\gets (1-\eta_k) M_k(z) + \eta_k \widetilde{M}_{k}(z) $ for $z\in\TT_N$.
    \vspace{0.5mm}
   \STATE Increment $k\gets k+1$
   \vspace{0.5mm}
   \UNTIL{$\norm{M_{k+1} - M_{k}}/\norm{M_{k}} \leq \epsilon $}
   \vspace{0.5mm}
    \STATE Compute $N_k(z) = \frac{1}{4}\pr{I + \sqrt{ I + 4\gamma_k^{-1} \Gamma_k^\ast  (z^{-1}I-\overline{A})^{\-\ast}\overline{C}^\ast \overline{C}(z^{-1}I-\overline{A})^{-1}\Gamma_k }}^2$ for $z \in \TT_{N'} \defeq \{\e^{j 2\pi n /N'} \mid n \!=\! 0,\dots,N'\!-\!1\}$
    \vspace{0.5mm}
   \STATE Compute $K(z) \gets \texttt{RationalApproximate}(N_k(z))$ (see \cref{app: rational approximation})
\end{algorithmic}
\end{algorithm}

\subsection{Additional Discussion on the Computation of Gradients}\label{app: gradients in fw}
By the Wiener-Hopf technique discussed in \cref{lemma:wiener}, the gradient $\RR_k = \nabla \Phi (\M_k)$ can be obtained as
\begin{align}
    R_k(z) = L_k(z)^{-\ast} \cl{  \Delta \K_\circ \L_k }_{-}(z)^\ast\cl{ \Delta \K_\circ \L_k }_{-}(z) L_k(z)^{-1} ,
\end{align}
where $\L_k \L_k^\ast= \M_k$ is the unique spectral factorization. Furthermore, using \eqref{eq:ad},\eqref{eq:finite parameter Gamma}, we can reformulate the gradient $ R_k(z)$ more explicitly as
\begin{align}
    R_k(z) = L_k(z)^{-\ast} \Gamma_k^\ast ( I \- z\overline{A})^{\-\ast}\overline{C}^\ast \,  \overline{C} ( I \- z\overline{A})^\inv \Gamma_k     L_k(z)^{-1},
\end{align}
where
$\Gamma_{k} = \frac{1}{2\pi} \int_{0}^{2\pi} (I-\e^{j\omega} \overline{A})^\inv \overline{B} L_k(\e^{j\omega}) d\omega$ as in \eqref{eq:finite parameter Gamma}.
Here, the spectral factor $L_k(z)$ is obtained for $z\in\TT_N$ by \cref{alg:spectral factor method}. Similarly, the parameter $\Gamma_k$ can be computed numerically using trapezoid rule over the discrete domain $\TT_N$, \ie,
\begin{equation}
         \Gamma_k \gets \frac{1}{N} \sum_{z\in \TT_N}  (I-z \overline{A})^\inv \overline{B} L_k(z) .
\end{equation}
The gradient $R_k(z)$ can thus be efficiently computed for $z\in\TT_N$.

\subsection{Implementation of Spectral Factorization}\label{app: spectral factorization}

To perform the spectral factorization of an irrational function $M(z)$, we use a spectral factorization method via discrete Fourier transform, which returns samples of the spectral factor on the unit circle. First, we compute $\Lambda(z)$ for $z\in\TT_N$, which is defined to be the logarithm of $M(z)$, then we take the inverse discrete Fourier transform $\lambda_k$ for $k=0,\dots,N-1$ of $\Lambda(z)$ which we use to compute the spectral factorization as $$L(z_n) \gets \exp\pr{\frac{1}{2} \lambda_0 + \sum_{k=1}^{N/2-1} \lambda_k z_n^{-k} + \frac{1}{2} (-1)^{n} \lambda_{N/2}}$$ for $k=0,\dots,N-1$ where $z_n= \e^{j 2\pi n /N}$ .

The method is efficient without requiring rational spectra, and the associated error term, featuring a purely imaginary logarithm, rapidly diminishes with an increased number of samples. It is worth noting that this method is explicitly designed for scalar functions.

\begin{algorithm}[ht]\label{alg:spectral factor method}
   \caption{\texttt{SpectralFactor}: Spectral Factorization via DFT}
\begin{algorithmic}[1]
   \STATE {\bfseries Input:} Scalar positive spectrum $M(z)>0$ on  $\TT_N \defeq \{\e^{j 2\pi n /N} \mid n \!=\! 0,\dots,N\!-\!1\}$ 
   \vspace{0.5mm}
   \STATE {\bfseries Output:} Causal spectral factor $L(z)$ of $M(z)>0$ on  $\TT_N$ 
   \vspace{0.5mm}
   \STATE Compute the cepstrum $\begin{aligned}\Lambda(z) \gets \log(M(z))\end{aligned}$ on $z\in \TT_N$.
   \vspace{0.5mm}
   \STATE Compute the inverse DFT  \\
    \vspace{0.5mm}
    $\begin{aligned}\lambda_k \gets \operatorname{IDFT}(\Lambda(z))\end{aligned}$ for $k=0,\dots,N\!-\!1$
   \vspace{0.5mm}
   \STATE Compute the spectral factor for $z_n= \e^{j 2\pi n /N}$ \\
    \vspace{0.5mm}
   $\begin{aligned}L(z_n) \gets \exp\pr{\frac{1}{2} \lambda_0 + \sum_{k=1}^{N/2-1} \lambda_k z_n^{-k} + \frac{1}{2} (-1)^{n} \lambda_{N/2}} \end{aligned}$, \quad $n=0,\dots,N\!-\!1$
\end{algorithmic}
\end{algorithm}

\subsection{Implementation of Bisection Method}\label{app:bisection}

To find the optimal parameter $\gamma_k$ that solves $\tr\br{((I\-\gamma_k^\inv \RR_k)^{-1} \- I)^2} \= r^2$ in the Frank-Wolfe update \eqref{eq:frank wolfe frequency}, we use a bisection algorithm. The pseudo code for the bisection algorithm can be found in Algorithm \ref{alg:bisection}. We start off with two guesses of $\gamma$ \ie ($\gamma_{left}, \gamma_{right}$) with the assumption that the optimal $\gamma$ lies between the two values (without loss of generality).
\begin{algorithm}[H]
\caption{\texttt{Bisection Algorithm}}\label{alg:bisection}
\begin{algorithmic}[1]
   \STATE {\bfseries Input: $h(\gamma), \gamma_{right}, \gamma_{left}$} 
   
   \vspace{0.5mm}
   \STATE Compute the gradient at $\gamma_{right}$: $\nabla h(\gamma)|_{\gamma_{right}}$

   \vspace{0.5mm}

   \vspace{0.5mm}
   \WHILE {$\mid \gamma_{right} - \gamma_{left} \mid > \epsilon$}
       \STATE Calculate the midpoint $\gamma_{mid}$ between $\gamma_{left}$ and $\gamma_{right}$
       \STATE Compute the gradient at $\gamma_{mid}$: $\nabla h(\gamma)|_{\gamma_{mid}}$
       
       \vspace{0.5mm}
       \IF {$\nabla h(\gamma)|_{\gamma_{mid}}$ = 0}
           \STATE \textbf{return} $\gamma_{mid}$ 
       \ELSIF {$\nabla h(\gamma)|_{\gamma_{mid}} > 0$}
           \STATE Update $\gamma_{right}$ to $\gamma_{mid}$
       \ELSE
           \STATE Update $\gamma_{left}$ to $\gamma_{mid}$
       \ENDIF
   \ENDWHILE
   
   \vspace{0.5mm}
   \STATE \textbf{return} the average of $\gamma_{left}$ and $\gamma_{right}$ 
\end{algorithmic}
\end{algorithm}

\subsection{Implementation of Rational Approximation}\label{app: rational approximation}
We present the pseudocode of \texttt{RationalApproximation}.
\begin{algorithm}[ht]\label{alg:rational approximation method}
   \caption{\texttt{RationalApproximation}}
\begin{algorithmic}[1]
   \STATE {\bfseries Input:} Scalar positive spectrum $N(z)>0$ on  $\TT_{N'} \defeq \{\e^{j 2\pi n /N'} \mid n \!=\! 0,\dots,N'\!-\!1\}$, and a small positive scalar $\epsilon$
   \vspace{0.5mm}
   \STATE {\bfseries Output:} Causal controller $K(z)$ on  $\TT_{N'}$ 
   \vspace{0.5mm}
   \STATE  Get $P(z),Q(z)$  by solving the convex optimization in \eqref{eq:hinf_rational_opt}, for \emph{fixed} $\epsilon$, given $M(z)$, \ie:
   \begin{align}   
    &\min_{ \substack{p_0,\dots p_m \in \R, q_0,\dots q_m \in \R, \eps \geq0 }} \; \eps 
    \\&\subjto \quad \begin{aligned}
    &q_0 = 1, & P(z), Q(z)>0,\quad   &  P(z) \- \pr{ N(z)  \+ \epsilon } {Q(z)}\leqq 0 , & {P(z)} \- \pr{ N(z) \- \epsilon } {Q(z)}\geqq 0 \quad \forall z\in \TT_{N'}\nonumber
    \end{aligned}
\end{align}
    \vspace{-5mm}
   \vspace{0.5mm}
   \STATE Get the rational spectral factors of $P(z),Q(z)$, which are $S_P(z),S_Q(z)$ using the canonical Factorization method in \cite{sayed_survey_2001}\\
   \vspace{0.5mm}
   \STATE Get $L^r(z)$,the rational spectral factor of$N(z)$, as $S_P(z)/S_Q(z)$\\
   \STATE Get $K(z)$ from the formulation in  \eqref{eq:RationalK},\eqref{eq:Kstatespace}
\end{algorithmic}
\end{algorithm}

\section{Proof of the State-Space Representation of the Controller} \label{ap:rational}

\subsection{Proof of \cref{lemma:state space K}}
Let the spectral factor $\Tilde L(z)$ in rational form be given as \begin{equation}\Tilde L(z)=(I+\Tilde{C} (z I -\Tilde{A})^{-1}\Tilde{B})\Tilde{D}^{1/2},\end{equation} with its inverse given by:
\begin{align}\label{eq:linv}
    \Tilde L^{-1}(z)=\Tilde{D}^{-1/2}(I-\Tilde{C}(zI- (\Tilde{A}-\Tilde{B}\Tilde{C}))^{-1}\Tilde{B}),
\end{align}
and its operator form denoted by $\Tilde \L$.

We write the DR-RO controller, $K(z)$, as a sum of causal functions:
\begin{align}
      K(z)&=\Delta^{-1}(z) \{\Delta \K_\circ \Tilde \L\}_{+}(z)\Tilde L^{-1}(z)\\
      &=\Delta^{-1}(z) \left(\{\Delta \K_\circ \}_{+}(z)\Tilde L(z)+ \{\{\Delta \K_\circ \}_{-}\Tilde \L\}_{+}(z)\right)\Tilde L^{-1}(z)\\
      &=\Delta^{-1}(z) \{\Delta \K_\circ \}_{+}(z) + \Delta^{-1} \{\{\Delta \K_\circ \}_{-}\Tilde \L\}_{+}(z)\Tilde L^{-1}(z)\label{eq:last}.
\end{align}

From Lemma 4 in \cite{sabag2023regretoptimal}, we have:
\begin{align}\label{eq:dl}
    \{\Delta \K_\circ\}_{-}(z)=- \Bar{R} B_u^{\ast} (z^{\-1}I-A_k^\ast)^{-1}A_k^\ast P B_w
\end{align}

where the LQR controller is defined as $K_{lqr} = (I+B_u^\ast P B_u)^{-1}B_u^\ast P A$ and the closed-loop
matrix $A_K = A-B_u K_{lqr}$ with $P \succ 0$ is the unique stabilizing solution to the LQR Riccati equation $P = Q + A^\ast P A - A^\ast P B_u(I + B_u^\ast P B_u)^{-1}B_u^\ast P A$, $Q=C^\tp C,$ and with $\Bar{R}=(I+B_u^\ast P B_u)^{-\ast/2}$.

Multiplying equation \eqref{eq:dl} with $\Tilde L$, and taking its causal part, we get:
\begin{align}
    \{\{\Delta \K_\circ\}_{-}\Tilde \L\}_{+}(z)= \{- \Bar{R} B_u^\ast (z^{\-1}I-A_k^\ast)^{-1}A_k^\ast P B_w \Tilde{C} (zI-\Tilde{A})^{-1}\Tilde{B}\Tilde{D}^{1/2}- \Bar{R}B_u^\ast (z^{\-1}I-A_k^\ast)^{-1}A_k^\ast P B_w \Tilde{D}^{1/2}  \}_{+}.
\end{align}

Given that the term $\Bar{R}B_u^\ast (z^{\-1}I-A_k^\ast)^{-1}A_k^\ast P B_w \Tilde{D}^{1/2}$ is strictly anticausal, and considering the matrix $\Tilde{U}$ which solves the lyapunov equation: $A_k^\ast P B_w \Tilde{C }+ A_k^\ast \Tilde{U} A=\Tilde{U}$, we get $\{\{\Delta \K_\circ\}_{-}\Tilde \L\}_{+}(z)$ as:

\begin{align}
    \{\{\Delta \K_\circ\}_{-}\Tilde \L\}_{+}(z)&=  \{- \Bar{R} B_u^\ast ((z^{\-1}I-A_k^\ast)^{-1}A_k^\ast \Tilde{U} + \Tilde{U} \Tilde{A}(zI-\Tilde{A})^{-1}+\Tilde{U})\Tilde{B}\Tilde{D}^{1/2} \}_{+}\\
    &=- \Bar{R} B_u^\ast \Tilde{U} (\Tilde{A}(zI-\Tilde{A})^{-1}+I)\Tilde{B}\Tilde{D}^{1/2}\\
    &=-z \Bar{R} B_u^\ast \Tilde{U} (zI-\Tilde{A})^{-1}\Tilde{B}\Tilde{D}^{1/2} \label{eq:dl_linv}
\end{align}

Now, multiplying equation \eqref{eq:dl_linv} by the inverse of $\Tilde L$ \eqref{eq:linv}, we get:
\begin{align}
    \{\{\Delta \K_\circ\}_{-}\Tilde \L\}_{+}(z)\Tilde L^{-1}(z)&=-z \Bar{R}B_u^\ast \Tilde{U} (zI-\Tilde{A})^{-1} \Tilde{B} (I+\Tilde{C}(zI-\Tilde{A})^{-1}\Tilde{B})^{-1}\\
    &=
    -z \Bar{R}B_u^\ast \Tilde{U} (zI-\Tilde{A})^{-1} (I+\Tilde{B}\Tilde{C}(zI-\Tilde{A})^{-1})^{-1}\Tilde{B}\\
    &=-z \Bar{R} B_u^\ast \Tilde{U} (zI-\Tilde{A}_k)^{-1}\Tilde{B}\\
    &=-\Bar{R }B_u^\ast \Tilde{U} (I+(zI-\Tilde{A}_k)^{-1}\Tilde{A}_k)\Tilde{B}
\end{align}
where $\Tilde{A}_k=\Tilde{A}-\Tilde{B}\Tilde{C}$.

The inverse of $\Delta$ is given by $\Delta^{-1}(z)=(I- K_{lqr} (zI-A_k)^{-1}B_u)\Bar{R}^\ast$, and we know from lemma 4 in \cite{sabag2023regretoptimal} that $\{\Delta \K_\circ\}_{+}(z)=-\Bar{R}B_u^\ast P A (zI-A)^{-1}B_w - \Bar{R} B_u^\ast P B_w$. 

Then we can get the 2 terms of equation \eqref{eq:last}: 

\begin{equation}\label{eq:1}
    \Delta^{-1}(z)\{\Delta \K_\circ\}_{+}(z)=-K_{lqr}(zI-A_k)^{-1}(B_w-B_u \Bar{R}^\ast \Bar{R}B_u^\ast P B_w)-\Bar{R}^\ast \Bar{R}B_u^\ast P B_w
\end{equation}

and 
\begin{align}\label{eq:2}
    \Delta^{-1}(z)\{\{\Delta \K_\circ\}_{-}\Tilde \L\}_{+}(z)\Tilde L^{-1}(z)&= - ( I-K_{lqr}(zI-A_k)^{-1}B_u )\Bar{R}^\ast \Bar{R}B_u^\ast \Tilde{U} (zI-\Tilde{A}_k)^{-1}\Tilde{A}_k \Tilde{B}\\
    &+ K_{lqr}(zI-{A}_k)^{-1}B_u \Bar{R}^\ast \Bar{R}B_u^\ast \Tilde{U} \Tilde{B}\\
    &-\Bar{R}^\ast \Bar{R}B_u^\ast \Tilde{U} \Tilde{B}
\end{align}

Finally, summing equations \eqref{eq:1} and \eqref{eq:2}, we get the controller $K(z)$ in its rational form:
\begin{align}\label{eq:Kstatespace}
K(z)&=\underbrace{-\begin{bmatrix}
           \Bar{R}^\ast \Bar{R} B_u^\ast&-K_{lqr}
        \end{bmatrix}}_{\widetilde{H}} (zI-\underbrace{\begin{bmatrix}\Tilde{A_K} &0\\ B_u \Bar{R}^\ast \Bar{R} B_u^\ast & A_k\end{bmatrix}}_{\widetilde{F}})^{-1} \underbrace{\begin{bmatrix}\Tilde{A_K}\Tilde{B}\\-B_w+B_u\Bar{R}^\ast \Bar{R}B_u^\ast(P B_w+U_1 \Tilde{B}) \end{bmatrix}}_{\widetilde{G}} \nonumber \\
        &\underbrace{- \Bar{R}^\ast \Bar{R} B_u^\ast (PB_w+U_1 \Tilde{B})}_{\widetilde{J}}
\end{align}
which can be explicitly rewritten as in \eqref{eq:RationalK}.

\section{SDP Formulation for the Finite Horizon from \cite{cornell_drro_old}}
In this section, we state the SDP formulation of the finite-horizon DR-RO control problem  for a fixed horizon $T\>0$ presented in \cite{cornell_drro_old}, which is the main controller we compare against, to showcase the value of the infinite-horizon setting. 
This result highlights the triviality of non-causal estimation as opposed to causal estimation. In \cref{thm:finite horizon sdp}, we demonstrate that the finite-horizon DR-RO problem reduces to an SDP.

\begin{problem}[\textbf{Distributionally Robust Regret-Optimal (\DRO) Control in the Finite Horizon }]\label{prob:DR-RO-finite}
    Find a casual and time-invariant controller, $\K_T\in\causal_T$, that minimizes the worst-case expected regret in the finite horizon~\eqref{eq:worst case exp regret finite}, \ie,
    \begin{equation} \label{eq:DR-RO-finite}\vspace{-2mm}
    \inf_{\K_T \in \causal_T} {R}(\K_T,r)
\end{equation}
\end{problem}

\begin{theorem}[Adapted from \cite{cornell_drro_old}. An SDP formulation for finite-horizon DR-RO] \label{thm:finite horizon sdp}
    Let the horizon $T\>0$ be fixed and given the noncausal controller $\K_{\circ,T} \defeq - (\I_T + \F_T^\ast \F_T)^{\-1}\F_T^\ast \G_T,$ the \cref{prob:DR-RO-finite} reduces to the following SDP
    \begin{equation*}\label{eq:finite horizon sdp}
        \inf_{\substack{\K_T \in \causal_T,\\ \gamma\geq 0,\, \mathcal X_T \succ 0 }}\!\!\gamma (r_T^2 \-  \tr(\I_T)) \+ \tr(\mathcal X_T )\;\; \subjto \;\;  
        \begin{bmatrix}
         \mathcal X_T \!&\! \gamma \I_T \!&\! 0 \\
         \gamma \I_T  \!&\! \gamma \I_T \!&\!  (\K_T - \K_{\circ,T})^\ast \\
         0 \!&\! \K_T - \K_{\circ,T} \!&\! (\I_T\+\F_T^\ast\F_T)^\inv
        \end{bmatrix}\! \psdgeq\!0.
    \end{equation*}
    Moreover, the worst-case disturbance $\w_T^\star$ can be identified from the nominal disturbances $\w_{\circ,T}$ as $\w_T^\star = (\I_T-\gamma_{\star}^\inv \T_{\K_{\circ,T}}^\ast \T_{\K_{\circ,T}})^\inv \w_{\circ,T}$ where $\gamma_\star\>0$ and $\K_T^\star$ are the optimal solutions.

\end{theorem}

Note that the scaling of the SDP in \cref{thm:finite horizon sdp} with the time horizon is prohibitive for many time-critical real-world applications. Therefore, we compare our infinite-horizon controller to the finite-horizon one in the simulation sections \ref{sec:simul} and \ref{ap:sim}.

\section{Additional Simulations} \label{ap:sim}

\subsection{Note on Comparison with Other Methods in the Literature:}
As our work is the first to explore infinite-horizon distributionally robust control, our comparative experiments are constrained by the existing literature on finite-horizon distributionally robust control. Since the closest work to ours is that of \cite{DRORO}, our numerical experiments primarily compare with their finite-horizon version that utilizes an SDP formulation. 

Unfortunately, the framework in \cite{tacskesen2023distributionally} only allows for time-independent disturbances. While this approach is valuable for partially observed systems, it is widely acknowledged that the optimal distributionally robust controller for fully observed systems remains the same as the standard LQR controller as long as the disturbances are independent (though not necessarily identical) \cite{blackbook}. Therefore, in our setup, the results from Taskesen et al. simply reduce to the optimal LQR controller. This observation has also been noted in Taskesen et al.

While in the main text we simulated under the worst-case distributions corresponding to each controller being compared, we include in this section of the appendix other systems under the worst-case distributions, and also under other disturbance realizations (namely sinusoidal and uniform distributions).

\subsection{Additional Time Domain and Frequency Domain Simulations}
\paragraph{Time domain simulations:} 
We repeat the same experiment of section \ref{sec:simul} for 2 more systems, [REA4] and [HE3] \cite{aircraft}. [REA4] is a SISO system with 8 states and a stable $A$ matrix, while [HE3] has 4 states and an unstable $A$ matrix. The results are shown in figures \ref{fig:time_domain2},\ref{fig:time_domain3}. Similarly to our previous discussion, the infinite horizon DRRO controller achieves good performance across all systems, achieving the lowest cost under all considered noise scenarios.

\begin{figure*}[htbp]
    \centering
    \begin{subfigure}[b]{0.48\textwidth}
        \centering
        \includegraphics[width=0.8\textwidth]{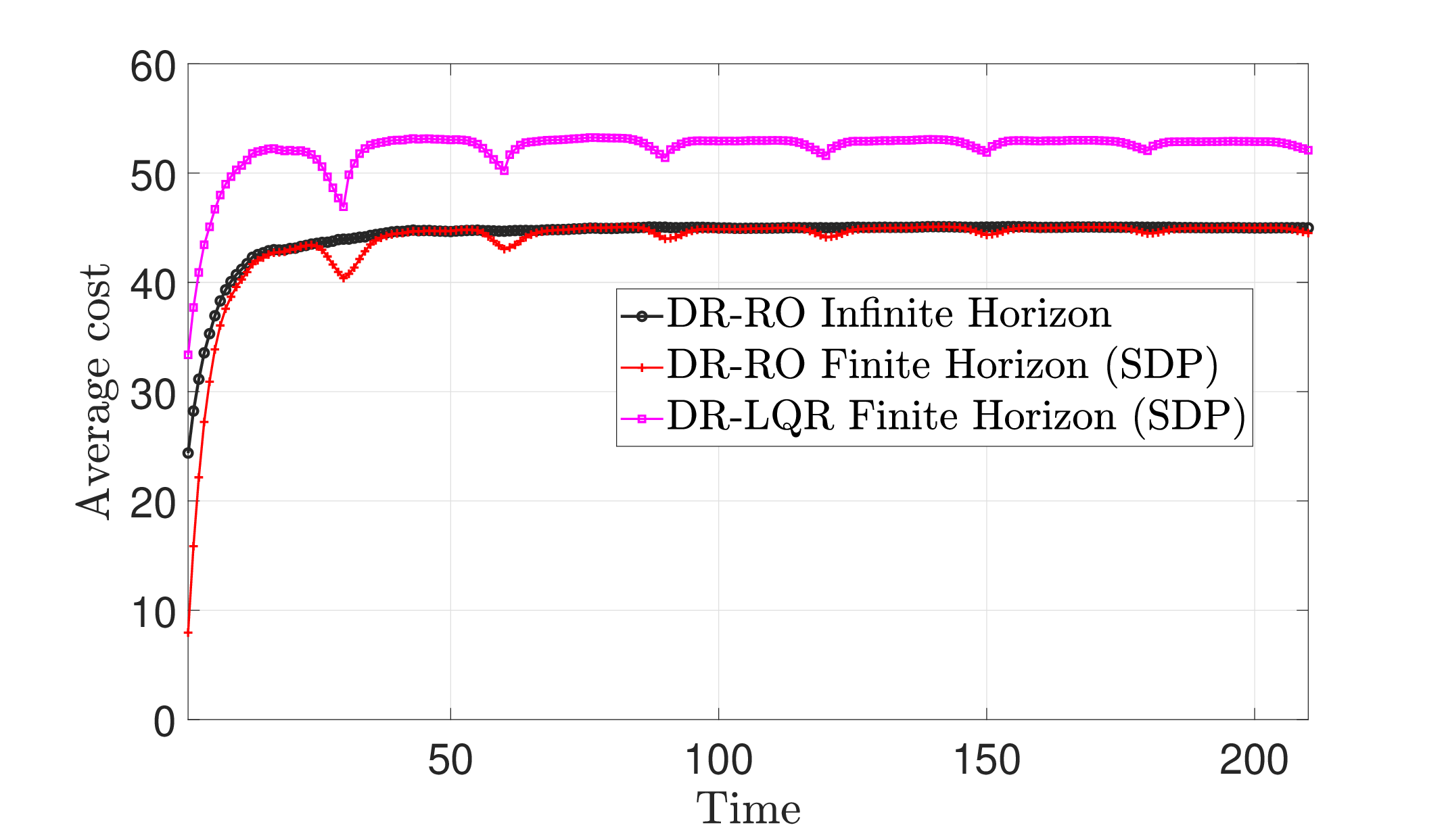}
        \caption{White noise}
        \label{fig:fig2}
    \end{subfigure}
    \hfill
    \begin{subfigure}[b]{0.48\textwidth}
        \centering
        \includegraphics[width=0.8\textwidth]{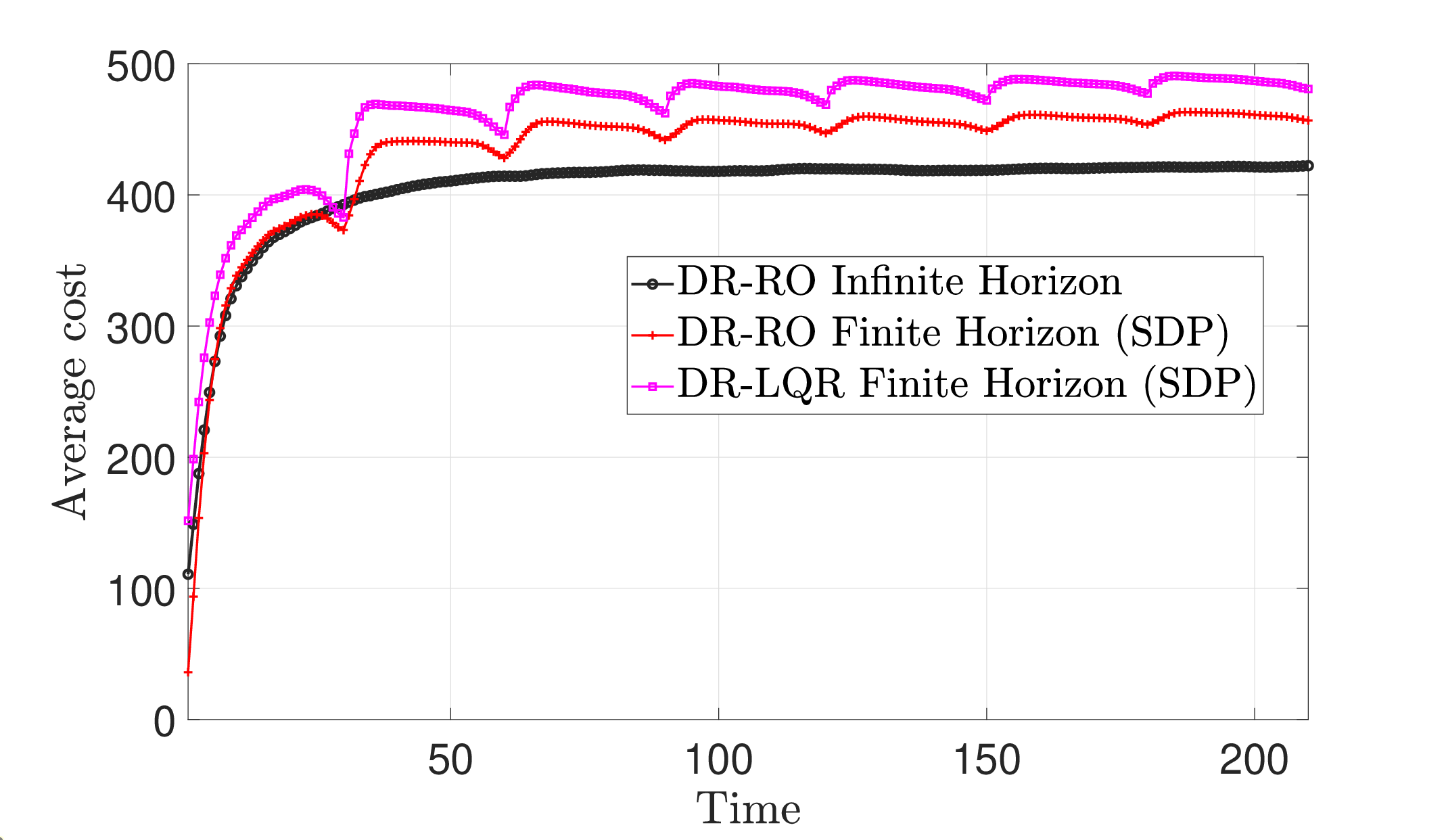}
        \caption{Worst disturbance for (I)}
        \label{fig:fig1}
    \end{subfigure}

    \vfill

    \begin{subfigure}[b]{0.48\textwidth}
        \centering
        \includegraphics[width=0.8\textwidth]{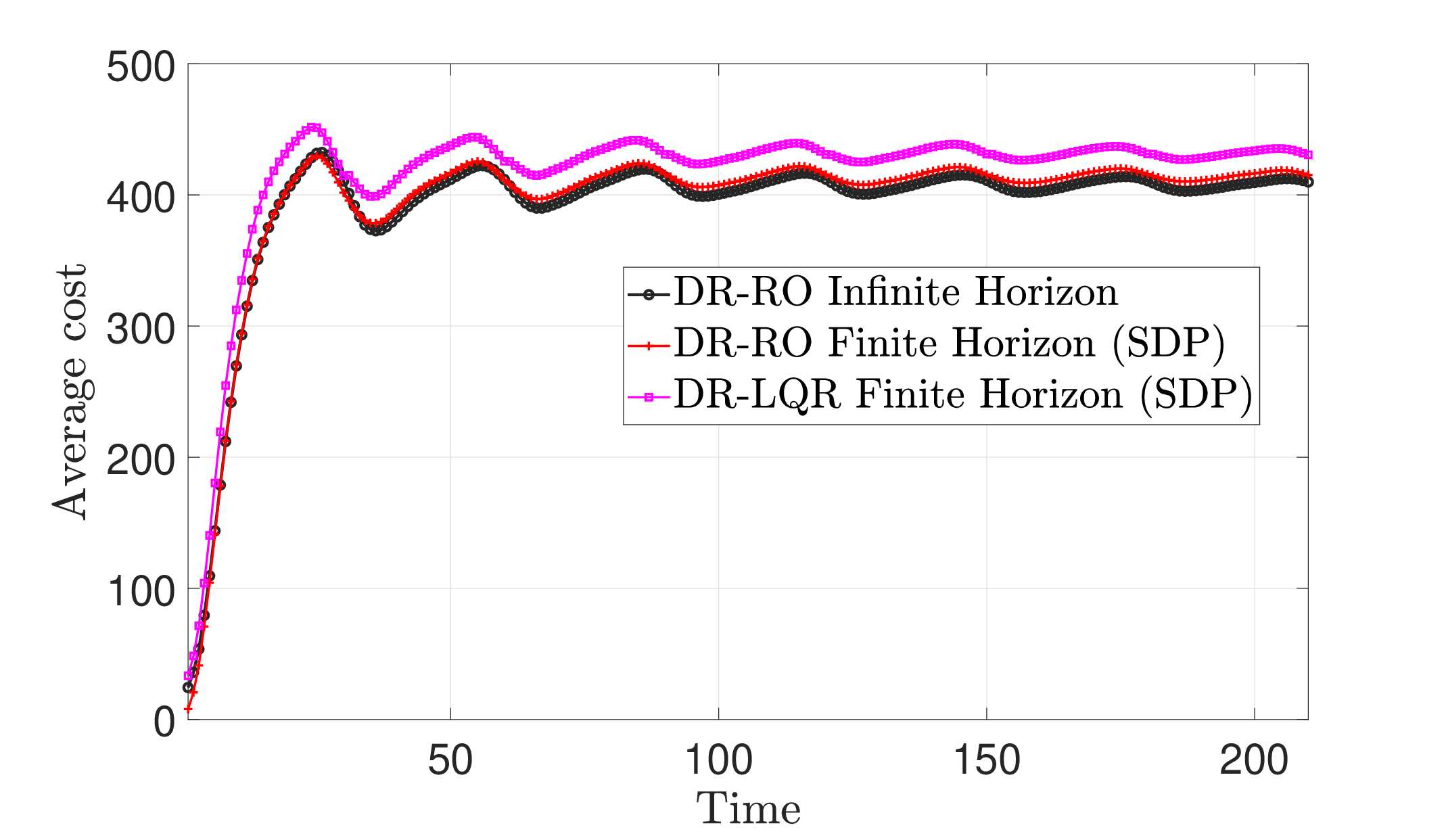}
        \caption{Worst disturbance for (II)}
        \label{fig:fig3}
    \end{subfigure}
    \hfill
    \begin{subfigure}[b]{0.48\textwidth}
        \centering
        \includegraphics[width=0.8\textwidth]{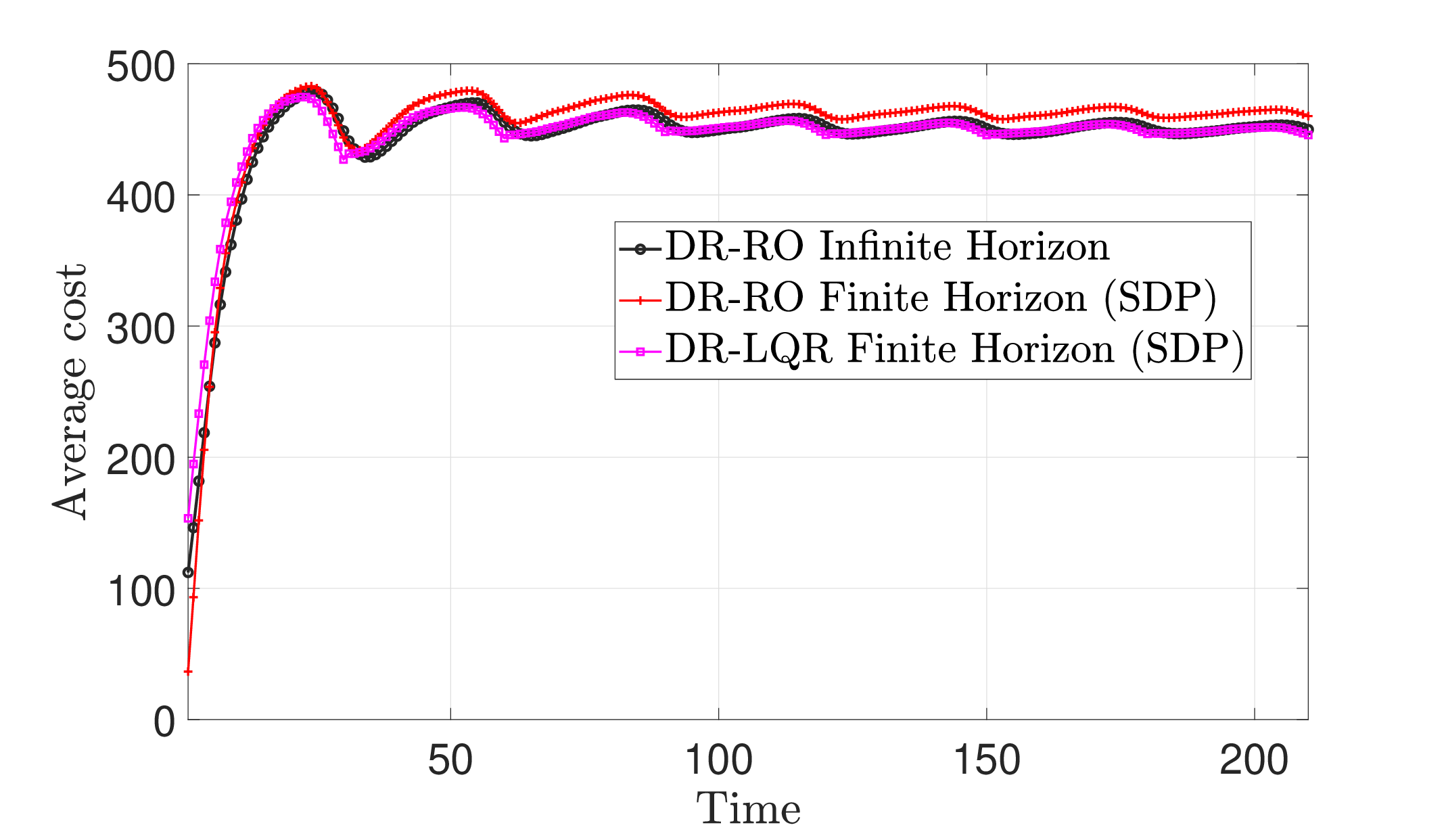}
        \caption{Worst disturbance for (III)}
        \label{fig:fig4}
    \end{subfigure}

    \caption{The control costs of the different DR controllers: (I) DR-RO in infinite horizon, (II) DR-RO in finite horizon and (III) DR-LQR in finite horizon under different disturbances for system [REA4] \cite{aircraft}. (a) is white noise, while (b), (c) and (d) are worst-case disturbances for each of the controllers, for $r=1.5$. The finite-horizon controllers are re-applied every $s=30$ steps. For all disturbances, the infinite horizon DRRO controller achieves lowest average cost, even in cases (c) and (d) where the finite horizon DR controllers are designed to minimize the cost.}
    \label{fig:time_domain2}
\end{figure*}

\begin{figure*}[htbp]
    \centering
    \begin{subfigure}[b]{0.48\textwidth}
        \centering
        \includegraphics[width=0.8\textwidth]{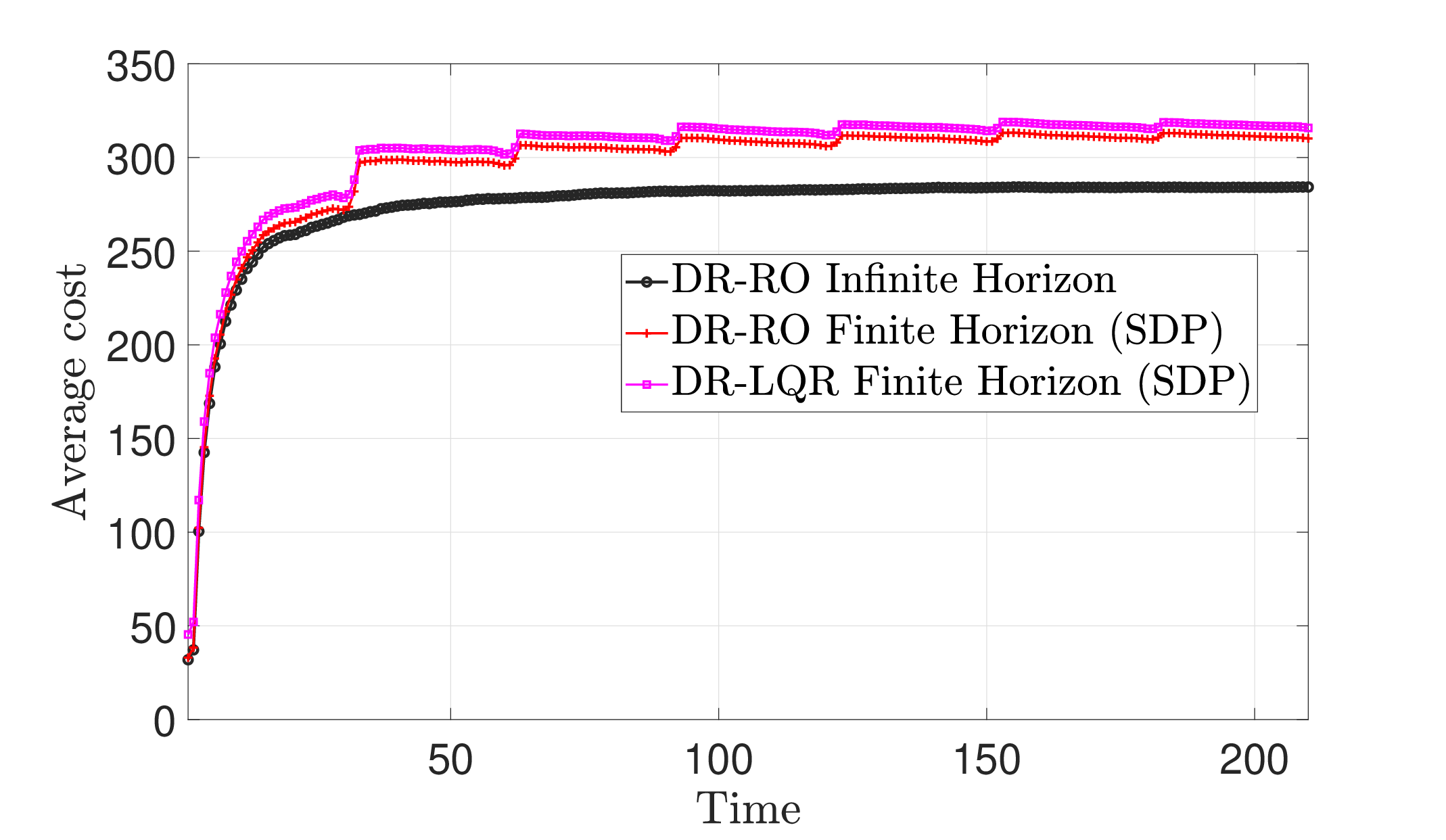}
        \caption{White noise}
        \label{fig:fig5}
    \end{subfigure}
    \hfill
    \begin{subfigure}[b]{0.48\textwidth}
        \centering
        \includegraphics[width=0.8\textwidth]{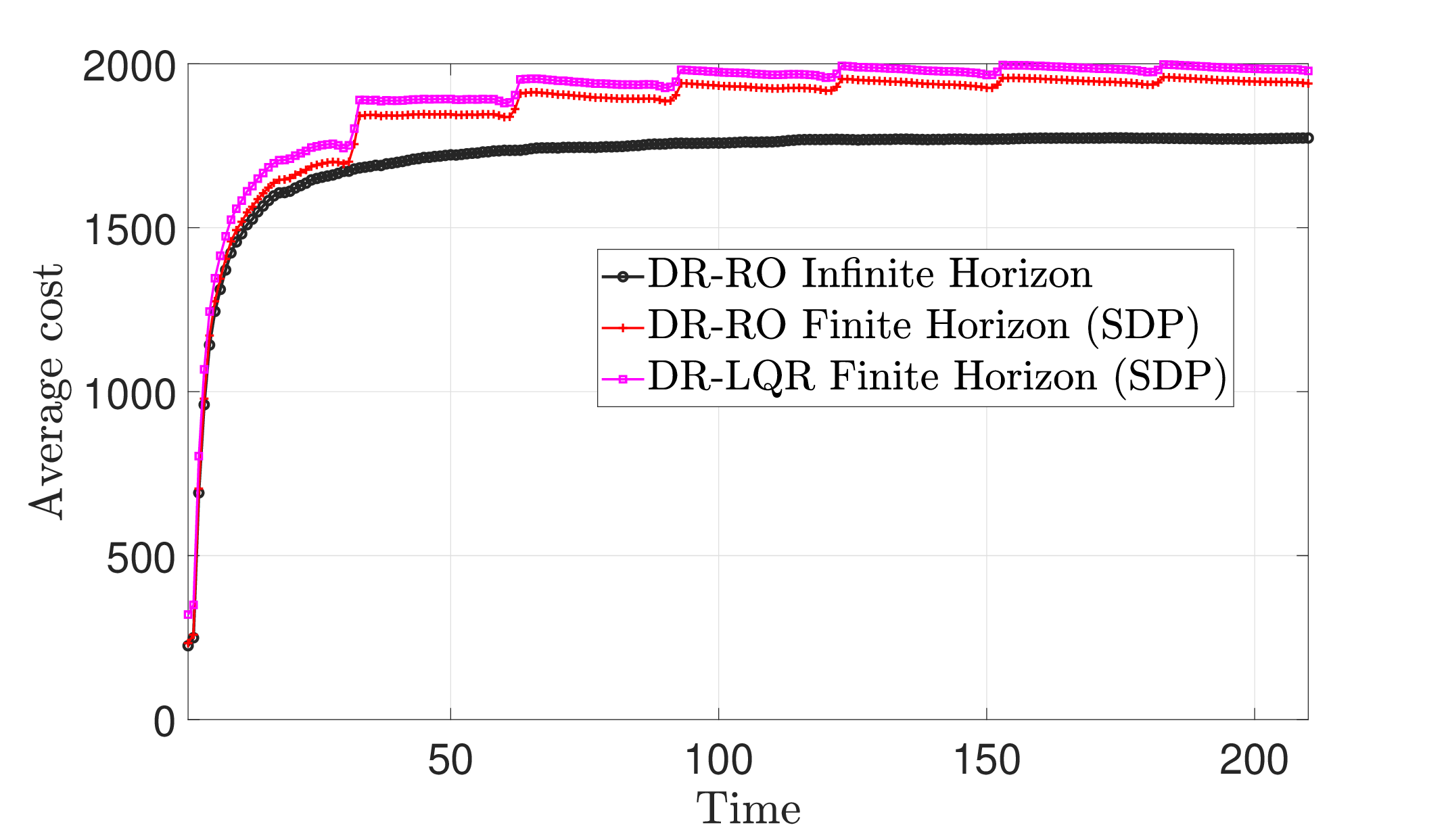}
        \caption{Worst disturbance for (I)}
        \label{fig:fig6}
    \end{subfigure}

    \vfill

    \begin{subfigure}[b]{0.48\textwidth}
        \centering
        \includegraphics[width=0.8\textwidth]{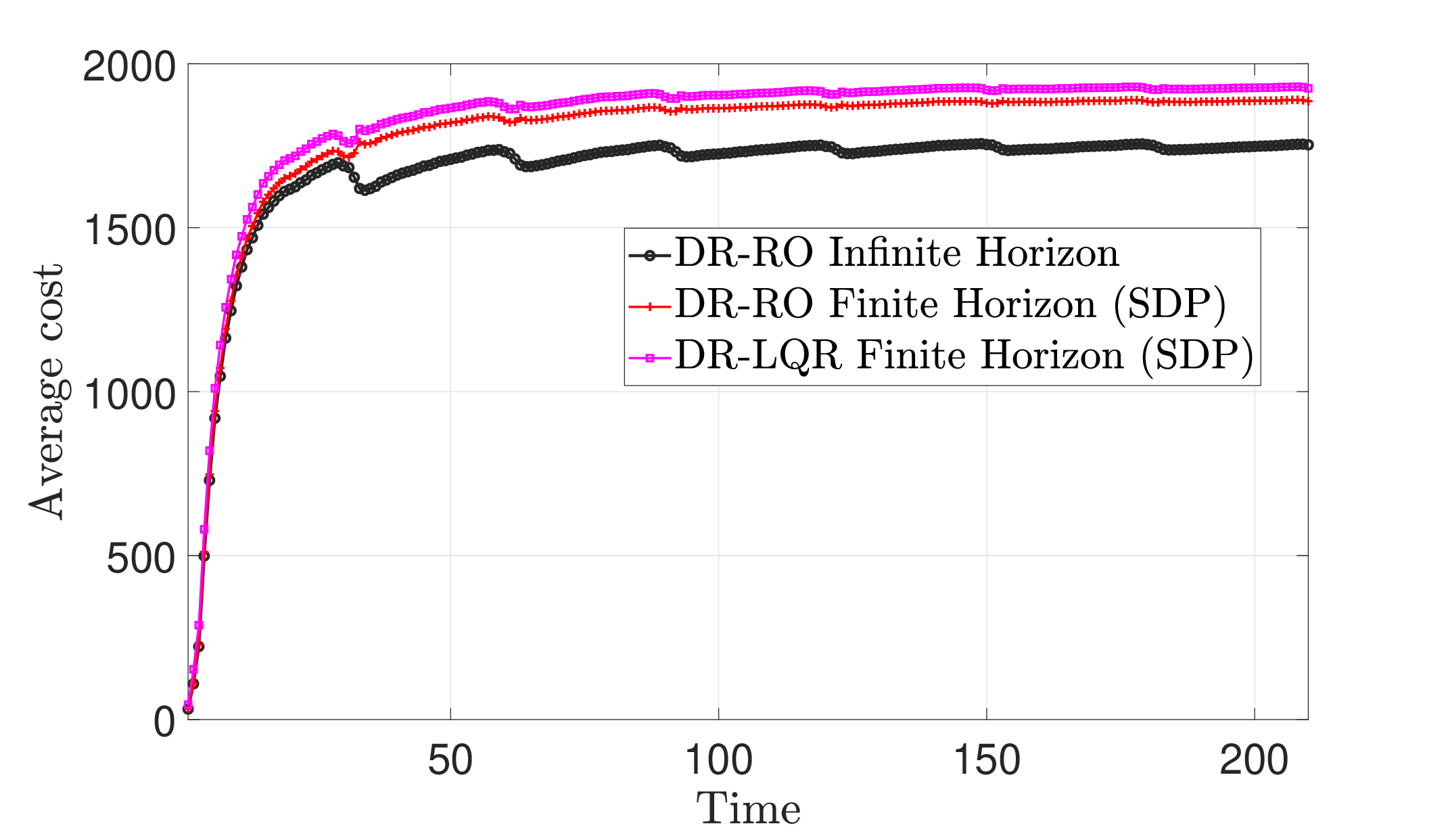}
        \caption{Worst disturbance for (II)}
        \label{fig:fig7}
    \end{subfigure}
    \hfill
    \begin{subfigure}[b]{0.48\textwidth}
        \centering
        \includegraphics[width=0.8\textwidth]{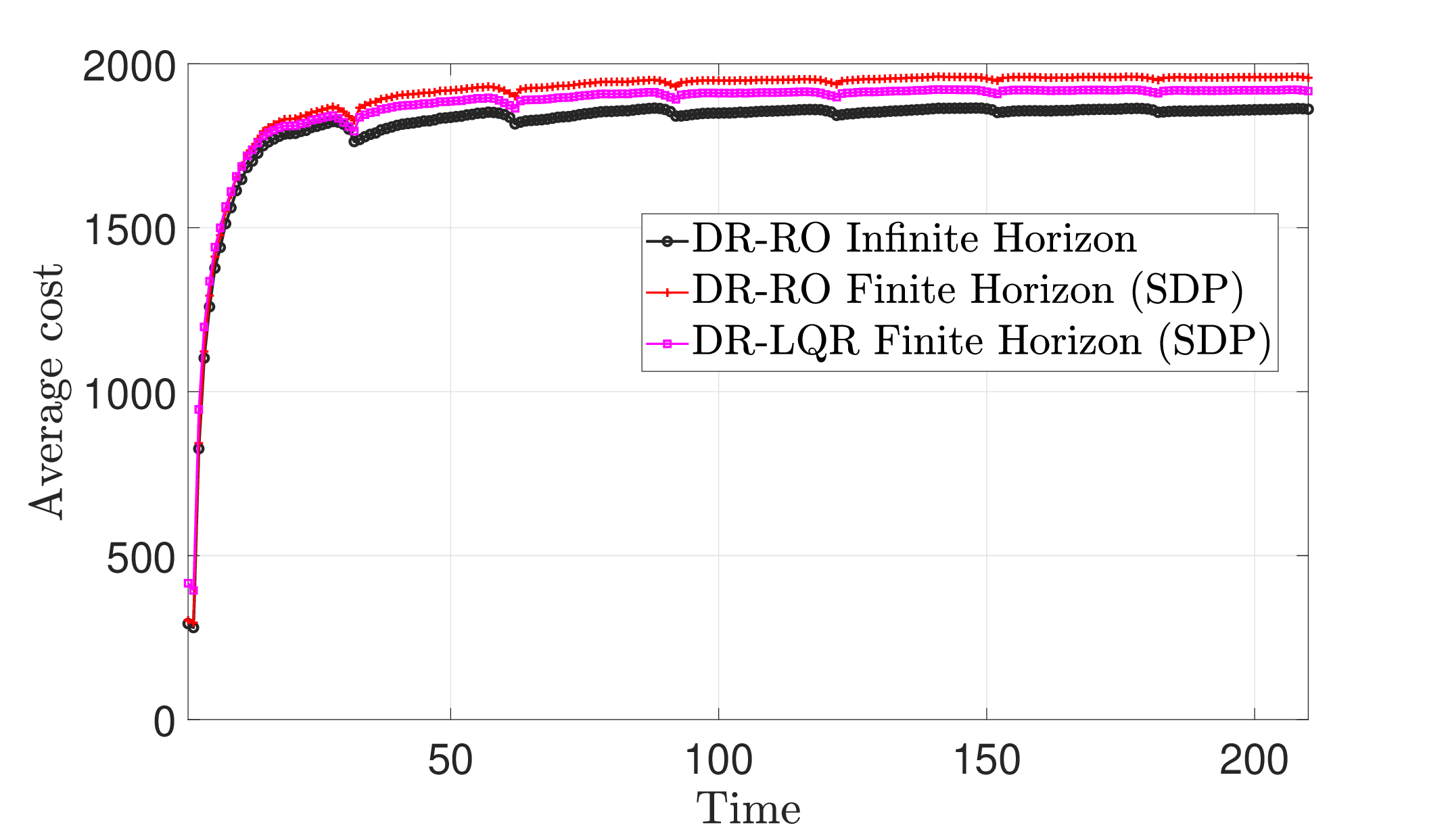}
        \caption{Worst disturbance for (III)}
        \label{fig:fig8}
    \end{subfigure}

    \caption{The control costs of the different DR controllers: (I) DR-RO in infinite horizon, (II) DR-RO in finite horizon and (III) DR-LQR in finite horizon under different disturbances for system [HE3] \cite{aircraft}. (a) is white noise, while (b), (c) and (d) are worst-case disturbances for each of the controllers, for $r=1.5$. The finite-horizon controllers are re-applied every $s=30$ steps. For all disturbances, the infinite horizon DRRO controller achieves lowest average cost, even in cases (c) and (d) where the finite horizon DR controllers are designed to minimize the cost.}
    \label{fig:time_domain3}
\end{figure*}

In figures \ref{fig:time_domain4} and \ref{fig:time_domain5}, we show the performance of the different DR controllers: (I) DR-RO in infinite horizon, (II) DR-RO in finite horizon and (III) DR-LQR in finite horizon under uniform noise and sinusoidal noise, respectively, for different systems. Note that the distributionally robust controller is guaranteed to perform better than other controllers under its own worst-case distribution, but has no guarantee of performance under other disturbances. Under uniform and sinusoidal noise, our infinite horizon DR-RO controller performs better than the finite horizon DR-LQR for systems [REA4] and [AC15], but worse than the finite horizon DR-LQR and on par with the  finite horizon DR-RO for system [HE3].
\begin{figure*}[htpb]
    \centering
        \begin{subfigure}[b]{0.32\textwidth}
        \centering
        \includegraphics[width=\textwidth]{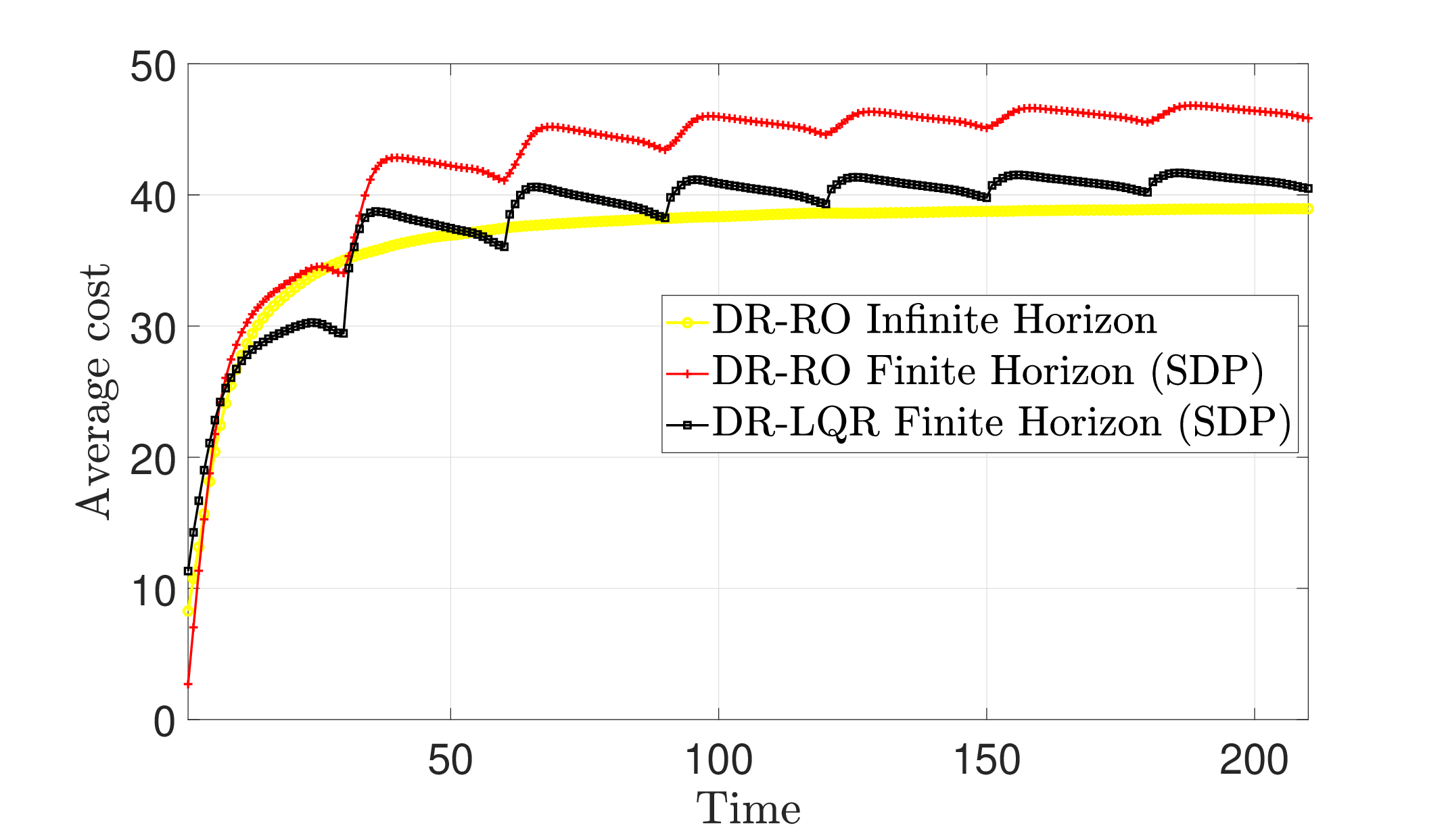}
        \caption{System [REA4]}
        \label{fig:figa2}
    \end{subfigure}
    \hfill
    \begin{subfigure}[b]{0.32\textwidth}
        \centering
        \includegraphics[width=\textwidth]{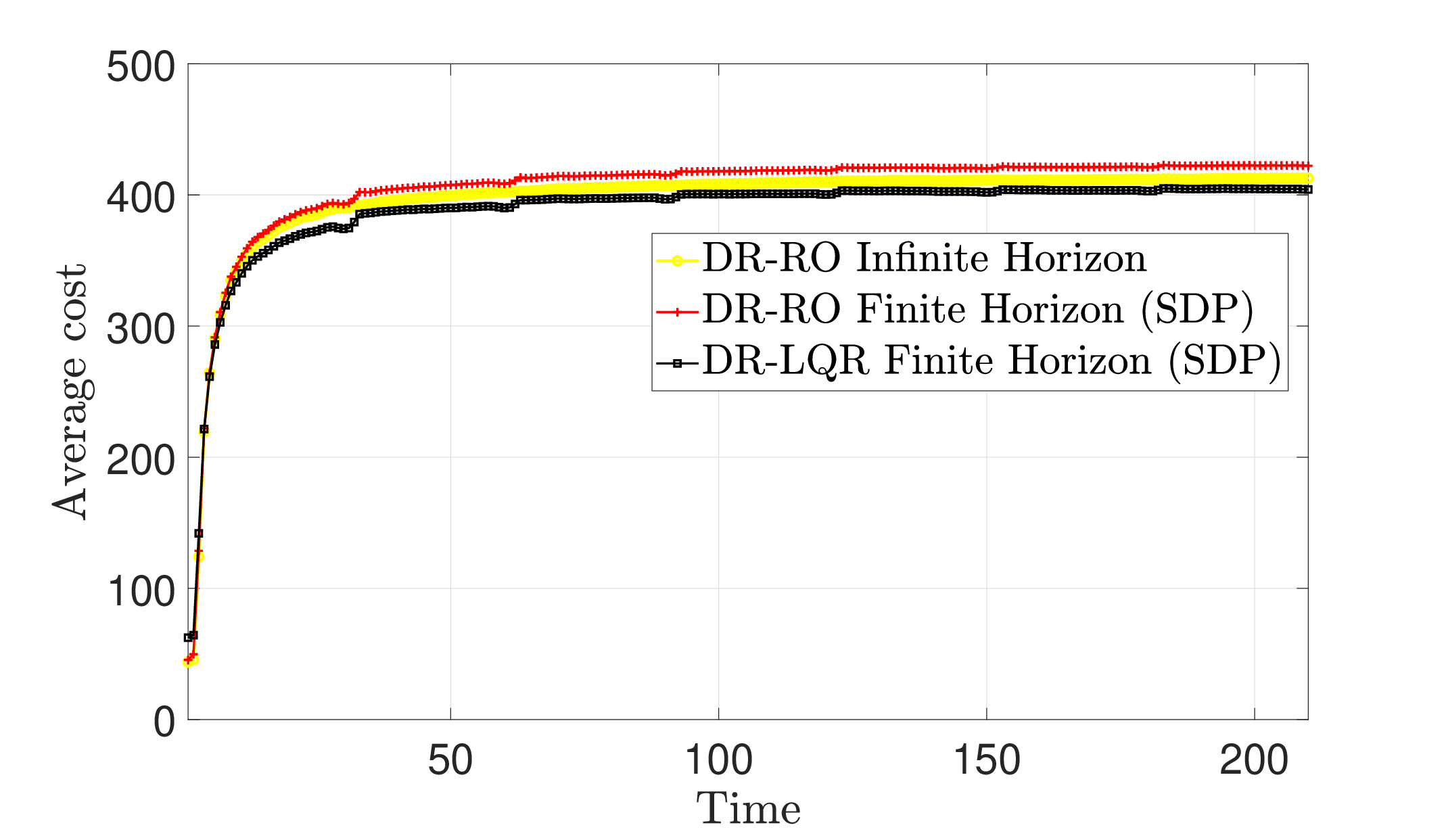}
        \caption{System [HE3]}
        \label{fig:figb2}
    \end{subfigure}
    \hfill
    \begin{subfigure}[b]{0.32\textwidth}
        \centering
        \includegraphics[width=\textwidth]{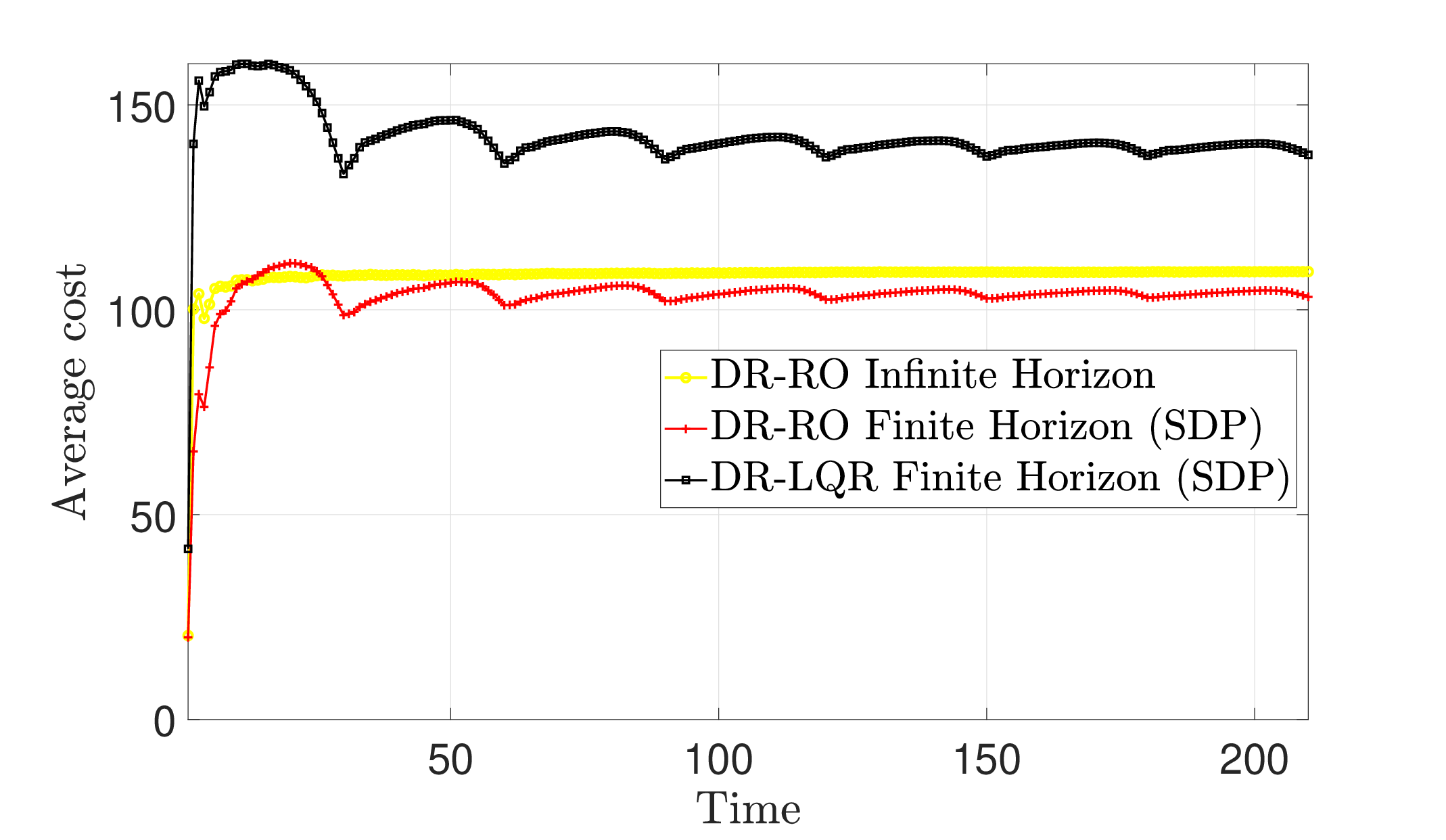}
        \caption{System [AC15]}
        \label{fig:figc2}
    \end{subfigure}
    \caption{The control costs of the different DR controllers: DR-RO in infinite horizon, DR-RO in finite horizon and DR-LQR in finite horizon under uniform noise distributions (with amplitude=2) for different systems, for $r=1.5$. }
    \label{fig:time_domain4}
\end{figure*}

\begin{figure*}[htpb]
    \centering
        \begin{subfigure}[b]{0.32\textwidth}
        \centering
        \includegraphics[width=\textwidth]{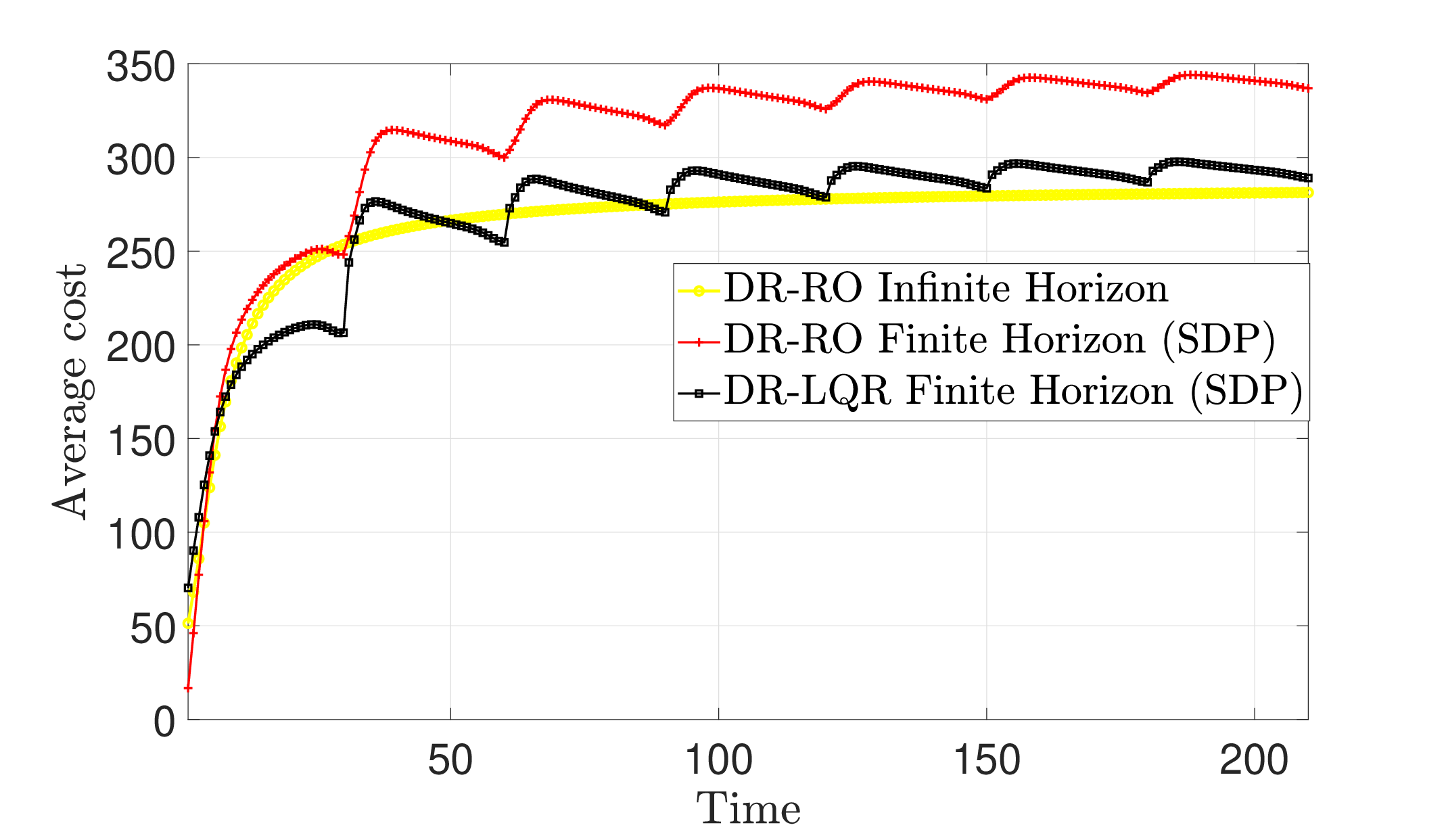}
        \caption{System [REA4]}
        \label{fig:figaa}
    \end{subfigure}
    \hfill
    \begin{subfigure}[b]{0.32\textwidth}
        \centering
        \includegraphics[width=\textwidth]{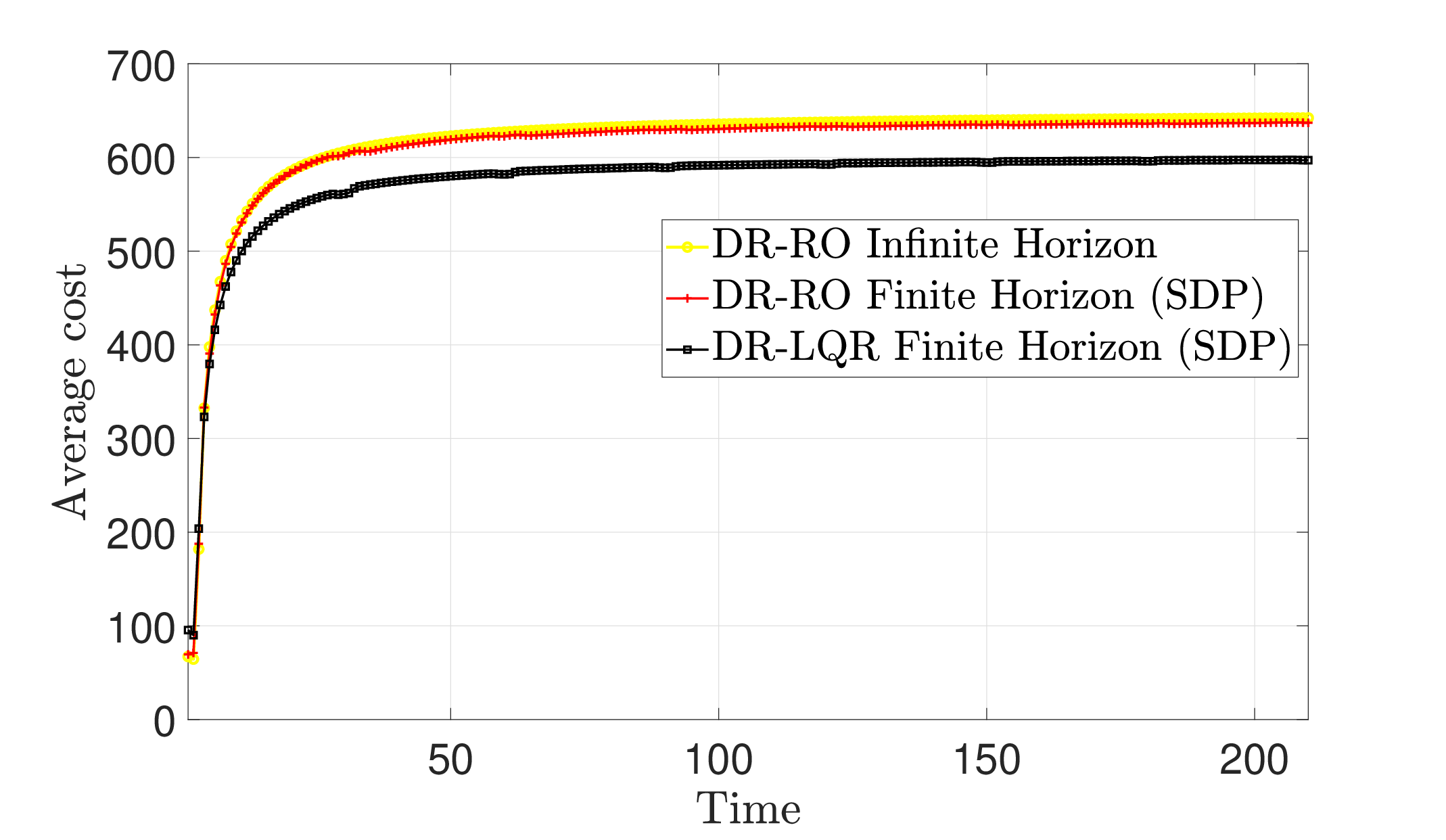}
        \caption{System [HE3]}
        \label{fig:figbb}
    \end{subfigure}
    \hfill
    \begin{subfigure}[b]{0.32\textwidth}
        \centering
        \includegraphics[width=\textwidth]{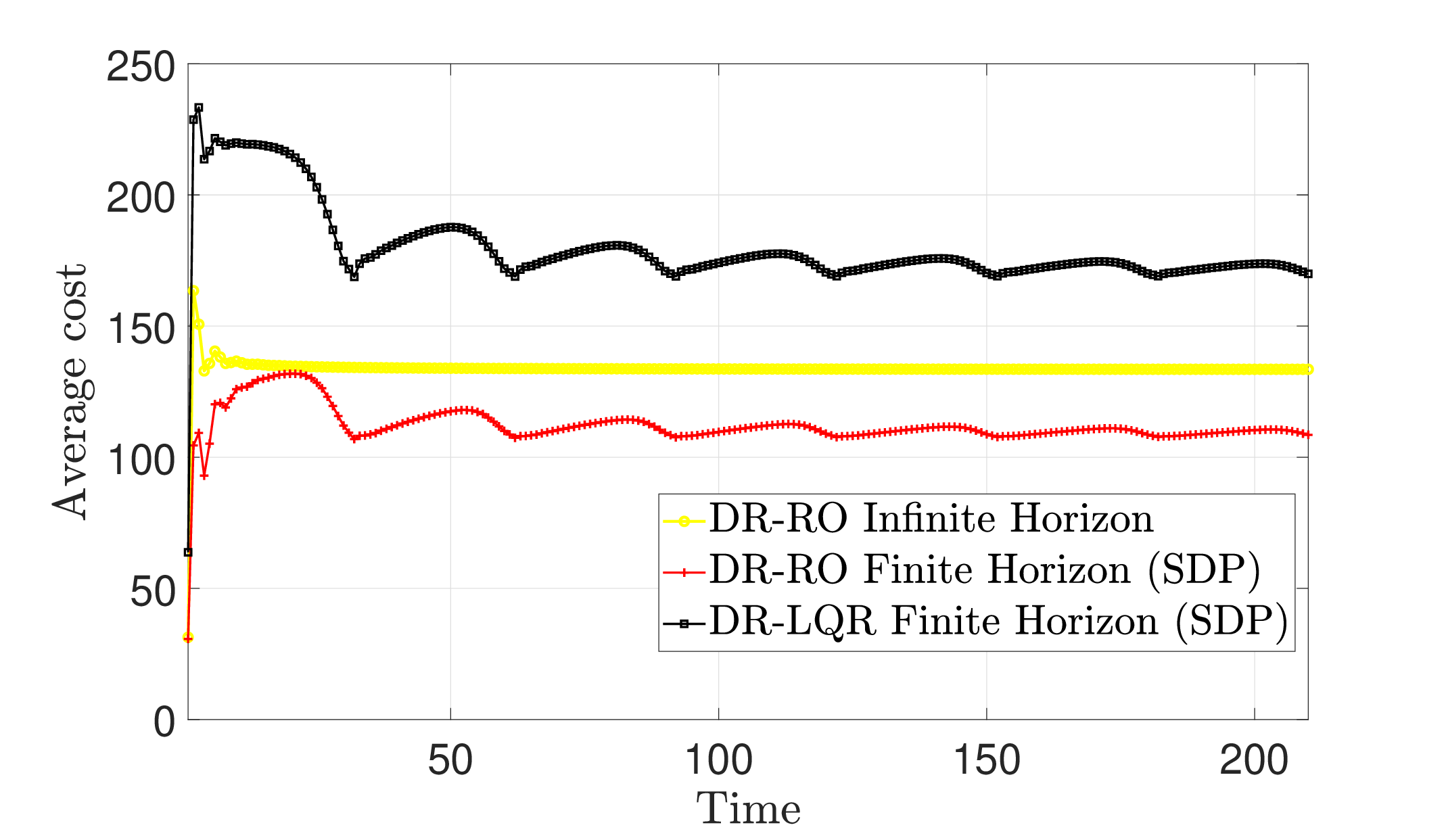}
        \caption{System [AC15]}
        \label{fig:figcc}
    \end{subfigure}
    \caption{The control costs of the different DR controllers: DR-RO in infinite horizon, DR-RO in finite horizon and DR-LQR in finite horizon under sinusoidal noise distributions (frequency=1, phase=$\pi/4$, amplitude=2) for different systems, for $r=1.5$. }
    \label{fig:time_domain5}
\end{figure*}

\paragraph{Frequency domain simulations}
We show in figure \ref{fig:freqdomain3} the frequency domain representation of the square of the norm of the DR-RO controller and its approximation for [AC15] and [HE3], demonstrating that lower order approximations of $m(e^{j\omega})$ provide good estimates.
 
\begin{figure*}[ht]
    \centering
    \begin{subfigure}[b]{0.3\textwidth}
        \centering
        \includegraphics[width=\textwidth]{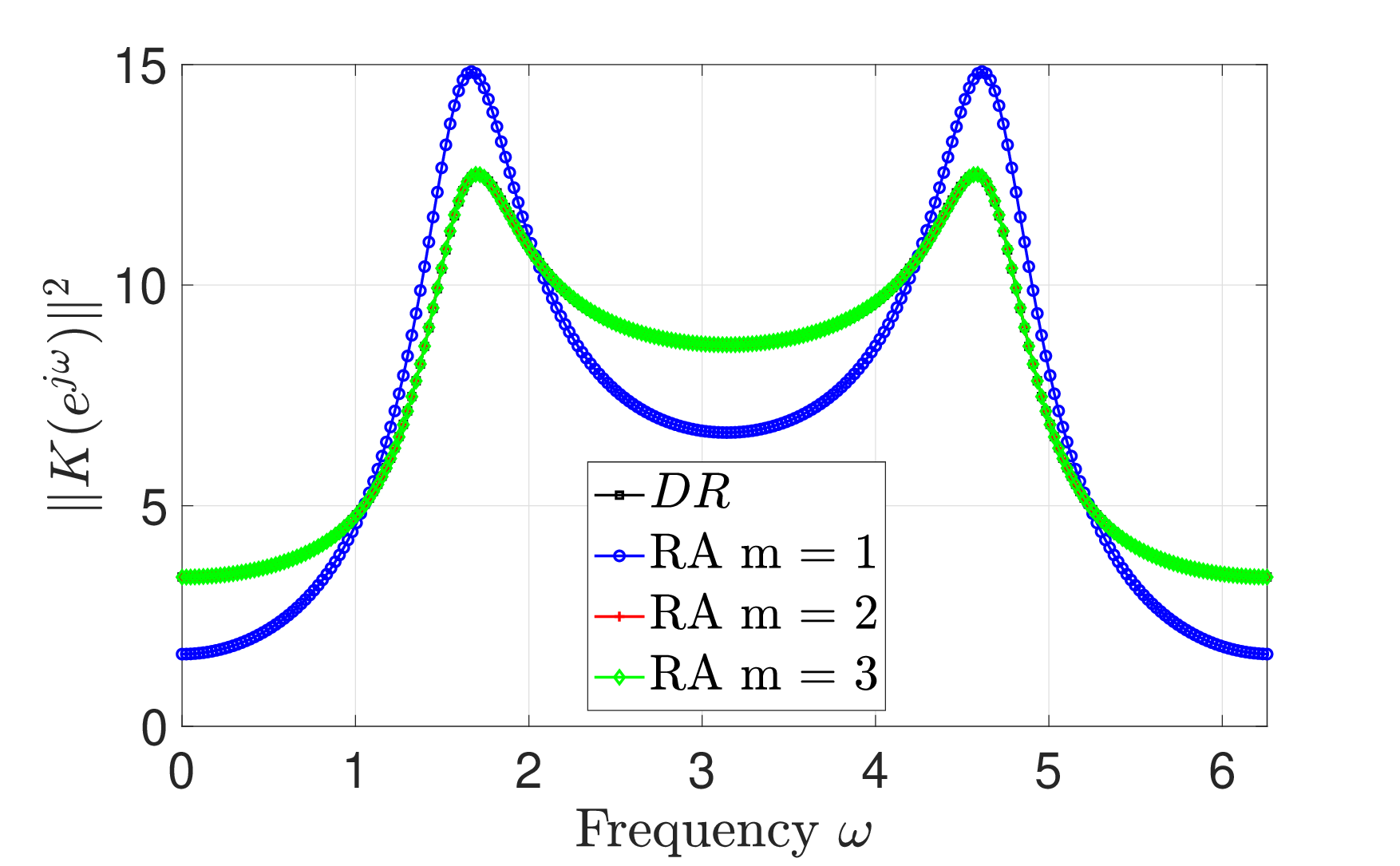}
        \caption{ }
        \label{fig:knorm_ac}
    \end{subfigure}
    \begin{subfigure}[b]{0.3\textwidth}
        \centering
        \includegraphics[width=\textwidth]{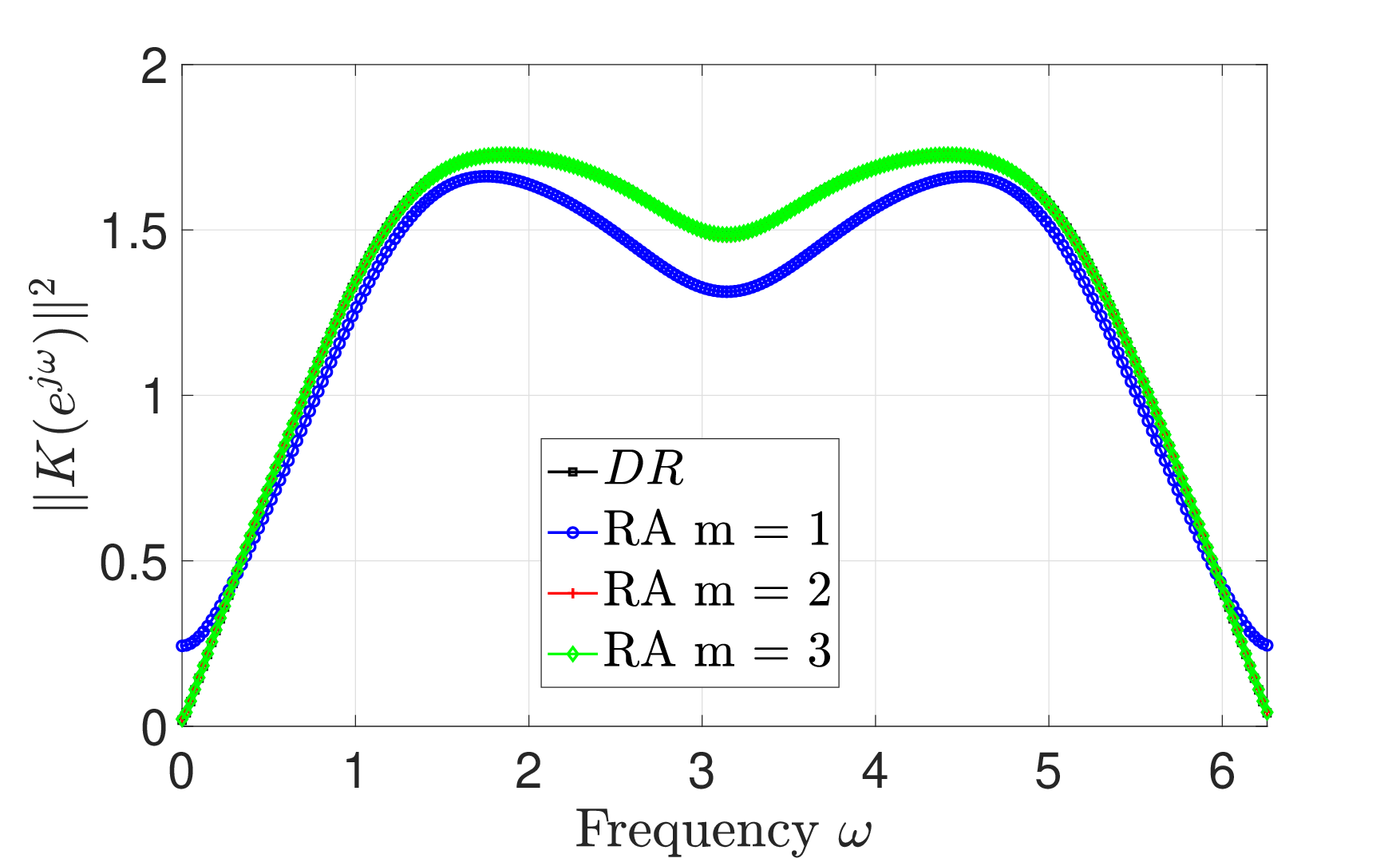}
        \caption{ }
        \label{fig:knowm_he}
    \end{subfigure}
    \begin{subfigure}[b]{0.3\textwidth}
        \centering
        \includegraphics[width=\textwidth]{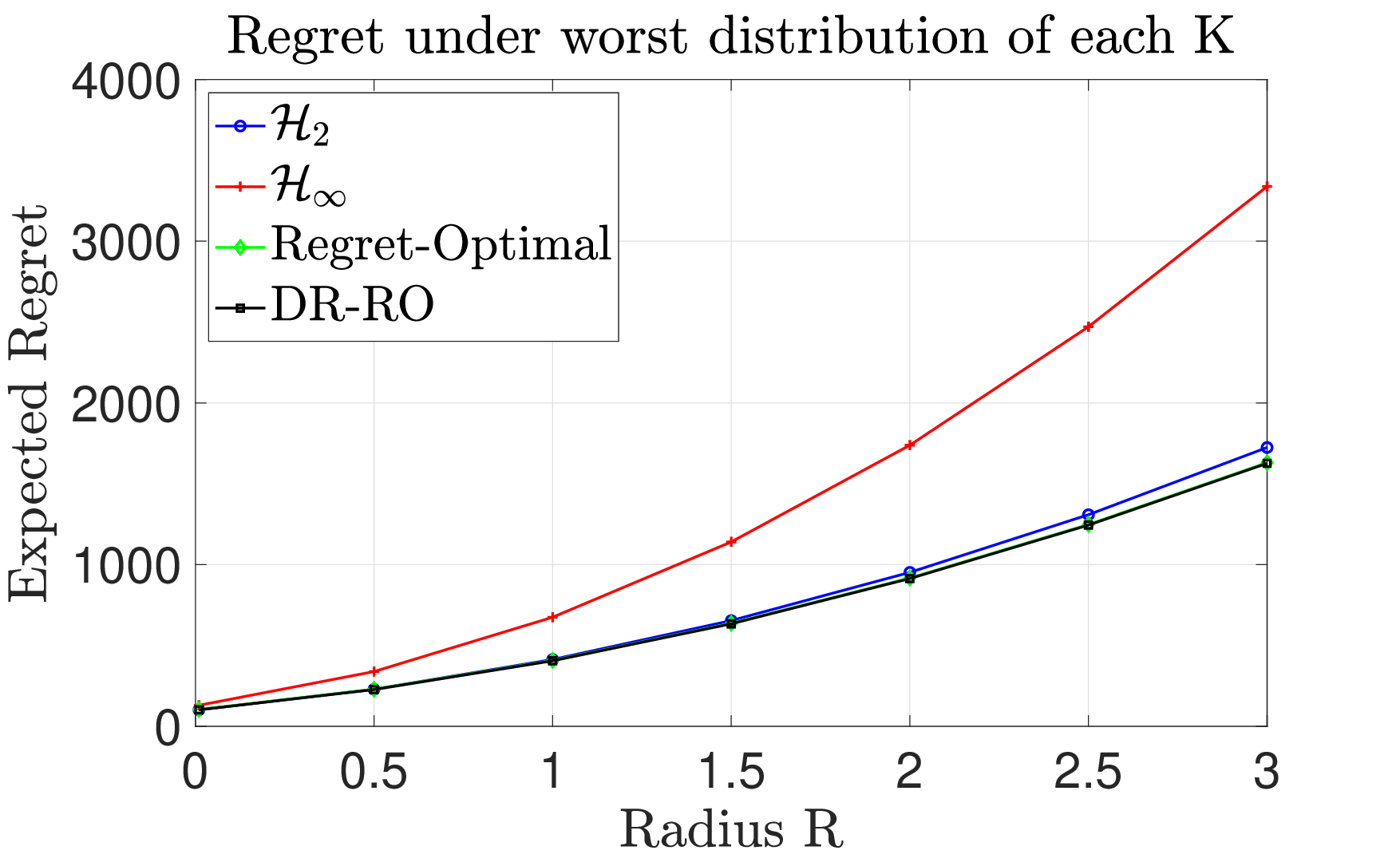}
        \caption{ }
        \label{fig:regret_he3}
    \end{subfigure}

    \caption{The frequency domain representation of the square of the norm of the DR-RO controller $K(e^{jw})$ and its approximation for [AC15] \ref{fig:knorm_ac} and [HE3] \ref{fig:knowm_he}. Figures \ref{fig:knorm_ac} and \ref{fig:knowm_he} reaffirm our conclusions that lower order approximations of $m(e^{jw})$ still yield good estimates of the same. Figure \ref{fig:regret_he3} represents the worst case expected regret of $\mathcal{H}_2, \mathcal{H}_{\infty}$ and the RO controller.}
    \label{fig:freqdomain3}
\end{figure*}



\end{document}